\documentclass[11pt]{amsart}
\usepackage{a4wide}
\usepackage{dsfont} 
\usepackage{amssymb}
\usepackage{graphicx}
\usepackage{MnSymbol,wasysym}

\usepackage{amsmath,amstext,amsfonts,amssymb,amsthm,amsrefs}

\usepackage[colorlinks,
             linkcolor=blue,
             citecolor=black!75!red,
             pdfproducer={pdfLaTeX},
             pdfpagemode=None,
             bookmarksopen=true
             bookmarksnumbered=true]{hyperref}

\usepackage{tikz}
\usetikzlibrary{arrows,calc,decorations.pathreplacing,decorations.markings,intersections,shapes.geometric,through,fit,shapes.symbols,positioning,decorations.pathmorphing}

\usepackage{amsmath,amsthm,stmaryrd}
\usepackage{cleveref}

\newtheorem{lemma}{Lemma}[section]
\newtheorem{theorem}[lemma]{Theorem}
\newtheorem{proposition}[lemma]{Proposition}
\newtheorem{corollary}[lemma]{Corollary}

\theoremstyle{definition} 

\newtheorem{definitionnodiamond}[lemma]{Definition}
\newtheorem{examplenodiamond}[lemma]{Example}
\newtheorem{remarknodiamond}[lemma]{Remark}

\newenvironment{definition}{\begin{definitionnodiamond}}{\hfill\ensuremath\blacklozenge\end{definitionnodiamond}}
\newenvironment{example}{\begin{examplenodiamond}}{\hfill\ensuremath\blacklozenge\end{examplenodiamond}}
\newenvironment{remark}{\begin{remarknodiamond}}{\hfill\ensuremath\blacklozenge\end{remarknodiamond}}

\makeatletter
\let\xx@thm\@thm
\AtBeginDocument{\let\@thm\xx@thm}
\makeatother

\renewcommand\qedhere{\qed}

\numberwithin{equation}{section}

\newcommand{\vtens}{\,\bar{\otimes}\,}

\crefname{section}{Section}{Sections}
\crefformat{section}{#2Section~#1#3} 
\Crefformat{section}{#2Section~#1#3} 

\crefname{subsection}{}{Subsections}
\crefformat{subsection}{\S#2#1#3} 
\Crefformat{subsection}{\S#2#1#3}

\crefname{definition}{Definition}{Definitions}
\crefformat{definition}{#2Definition~#1#3} 
\Crefformat{definition}{#2Definition~#1#3} 

\crefname{definitionnodiamond}{Definition}{Definitions}
\crefformat{definitionnodiamond}{#2Definition~#1#3} 
\Crefformat{definitionnodiamond}{#2Definition~#1#3}

\crefname{example}{Example}{Examples}
\crefformat{example}{#2Example~#1#3} 
\Crefformat{example}{#2Example~#1#3} 

\crefname{examplenodiamond}{Example}{Examples}
\crefformat{examplenodiamond}{#2Example~#1#3} 
\Crefformat{examplenodiamond}{#2Example~#1#3} 

\crefname{remark}{Remark}{Remarks}
\crefformat{remark}{#2Remark~#1#3} 
\Crefformat{remark}{#2Remark~#1#3} 

\crefname{remarknodiamond}{Remark}{Remarks}
\crefformat{remarknodiamond}{#2Remark~#1#3} 
\Crefformat{remarknodiamond}{#2Remark~#1#3} 

\crefname{convention}{Convention}{Conventions}
\crefformat{convention}{#2Convention~#1#3} 
\Crefformat{convention}{#2Convention~#1#3} 

\crefname{lemma}{Lemma}{Lemmas}
\crefformat{lemma}{#2Lemma~#1#3} 
\Crefformat{lemma}{#2Lemma~#1#3} 

\crefname{proposition}{Proposition}{Propositions}
\crefformat{proposition}{#2Proposition~#1#3} 
\Crefformat{proposition}{#2Proposition~#1#3} 

\crefname{corollary}{Corollary}{Corollaries}
\crefformat{corollary}{#2Corollary~#1#3} 
\Crefformat{corollary}{#2Corollary~#1#3} 

\crefname{theorem}{Theorem}{Theorems}
\crefformat{theorem}{#2Theorem~#1#3} 
\Crefformat{theorem}{#2Theorem~#1#3} 

\crefname{assumption}{Assumption}{Assumptions}
\crefformat{assumption}{#2Assumption~#1#3} 
\Crefformat{assumption}{#2Assumption~#1#3} 

\crefname{equation}{}{}
\crefformat{equation}{(#2#1#3)} 
\Crefformat{equation}{(#2#1#3)}

\crefname{align}{}{}
\crefformat{align}{(#2#1#3)} 
\Crefformat{align}{(#2#1#3)}

\crefname{proofstep}{Step}{Steps}
\crefformat{proofstep}{#2Step~#1#3} 
\Crefformat{proofstep}{#2Step~#1#3}

\newcommand{\cst}{\ifmmode\mathrm{C}^*\else{$\mathrm{C}^*$}\fi}
\newcommand{\wst}{\ifmmode\mathrm{C}^*\else{$\mathrm{W}^*$}\fi}
\newcommand{\st}{\;\vline\;}

\newcommand{\ot}{\otimes}
\newcommand{\bcls}{\mathrm{cls}}

\newcommand{\sA}{\mathsf{A}}
\newcommand{\sB}{\mathsf{B}}
\newcommand{\sC}{\mathsf{C}}

\newcommand{\sM}{\mathsf{M}}
\newcommand{\sN}{\mathsf{N}}
\newcommand{\sH}{\mathsf{H}}

\newcommand{\sX}{\mathsf{X}}
\newcommand{\sY}{\mathsf{Y}}

\newcommand{\sK}{\mathsf{K}}

\newcommand\bA{\mathbb A}
\newcommand\bB{\mathbb B}
\newcommand\bC{\mathbb C}
\newcommand\bD{\mathbb D}

\newcommand\bF{\mathbb F}
\newcommand\bG{\mathbb G}
\newcommand\bH{\mathbb H}

\newcommand\bK{\mathbb K}
\newcommand\bL{\mathbb L}
\newcommand\bM{\mathbb M}
\newcommand\bN{\mathbb N}

\newcommand\bP{\mathbb P}
\newcommand\bQ{\mathbb Q}
\newcommand\bR{\mathbb R}

\newcommand\bT{\mathbb T}

\newcommand\bZ{\mathbb Z}

\newcommand\cM{\mathcal M}

\newcommand\cO{\mathcal O}

\newcommand\cd{\textrm{cd}}

\newcommand{\id}{\textup{id}}

\newcommand{\I}{\mathds{1}}
\newcommand{\hh}[1]{\widehat{#1}}
\newcommand{\GG}{\mathbb{G}}
\newcommand{\RR}{\mathbb{R}}
\newcommand{\KK}{\mathbb{K}}

\newcommand{\cls}{\sigma-\textrm{cls}}
\newcommand{\HH}{\mathbb{H}}

\newcommand{\ww}{\mathrm{W}}
\newcommand{\WW}{{\mathds{V}\!\!\text{\reflectbox{$\mathds{V}$}}}}
\newcommand{\Ww}{\mathds{W}}
\newcommand{\wW}{\text{\reflectbox{$\Ww$}}\:\!}
\newcommand{\Vv}{\mathds{V}}
\newcommand{\vV}{\text{\reflectbox{$\Vv$}}\:\!}

\newcommand{\hGG}{\hh{\GG}}

\DeclareMathOperator{\C}{C}

\DeclareMathOperator{\M}{M}
\DeclareMathOperator{\Mor}{Mor}

\DeclareMathOperator{\Linf}{L^\infty\!\!\;}

\DeclareMathOperator{\Ltwo}{L^2\!\!\;}

\numberwithin{equation}{section}

\title{Fundamental isomorphism theorems for quantum groups}
\author{Alexandru Chirvasitu,\quad   Souleiman Omar Hoche,\quad Pawe{\l} Kasprzak}

\begin{document}

\date{}

\newcommand{\Addresses}{{
  \bigskip
  \footnotesize

  \textsc{Department of Mathematics, University
    of Washington, Seattle, WA 98195-4350, USA}\par\nopagebreak
  \textit{E-mail address}: \texttt{chirva@uw.edu}

  \medskip

  \textsc{Laboratoire de Math\'ematiques de Besan\c{c}on, Universit\'e de Bourgogne Franche-Comt\'e, 16, Route de Gray, 25030 Besan\c{c}on Cedex, France}\par\nopagebreak
  \textit{E-mail address}: \texttt{hoche.souleiman\_omar@univ-fcomte.fr}

  \medskip

  \textsc{Department of Mathematical Methods in Physics, Faculty of Physics, University of Warsaw, Poland}\par\nopagebreak
  \textit{E-mail address}: \texttt{pawel.kasprzak@fuw.edu.pl}

}}

\maketitle

\begin{abstract}
  The lattice of subgroups of a group is the subject of numerous
  results revolving around the central theme of decomposing the group
  into ''chunks'' (subquotients) that can then be compared to one
  another in various ways.  Examples of results in this class would be
  the Noether isomorphism theorems, Zassenhaus' butterfly lemma, the
  Schreier refinement theorem for subnormal series of subgroups, the
  Dedekind modularity law, and last but not least the Jordan-H\"older
  theorem.

  We discuss analogues of the above-mentioned results in the context
  of locally compact quantum groups and linearly reductive quantum
  groups. The nature of the two cases is different: the former is
  operator algebraic and the latter Hopf algebraic, hence the
  corresponding two-part organization of our study. Our intention is
  that the analytic portion be accessible to the algebraist and vice versa.

  The upshot is that in the locally compact case one often needs
  further assumptions (integrability, compactness, discreteness).  In
  the linearly reductive case on the other hand, the quantum versions
  of the results hold without further assumptions. Moreover the case
  of compact / discrete quantum groups is usually covered by both the
  linearly reductive and the locally compact framework, thus providing
  a bridge between the two.
\end{abstract}


\vspace{.5cm}



\section*{Introduction}

The theory of quantum groups has been a rich and fruitful one, as
evidenced by the many excellent monographs on the subject
\cites{CP,K,KS} and the references therein. As the vastness of the
field would by necessity make any attempt at documenting the
literature incomplete, we cite only a select few sources in this
introduction and instead refer the reader to the papers that are more
immediately relevant for us in the main body of the paper.

The appropriately ill-defined concept of a quantum group is flexible
enough to allow for several branches of the theory, that continue to
develop vigorously but largely independently. In this paper we draw a
rough distinction between two flavors of quantum-group-theoretic
results: those of an analytic nature, where the objects to be studied
are operator algebras ($\C^*$ or von Neumann) that mimic the behavior
of algebras of (continuous, essentially bounded, etc.) functions on a
locally compact group (see e.g. \cite{KV} and the references cited
there for this perspective), and those of a purely algebraic
character, whereby the quantum groups are recast as Hopf algebras
(\cite{D} itself, where the term `quantum group' seems to have been
coined, or numerous other sources, such as \cite{AD}, where a category
of quantum groups is defined explicitly).

While there is a common core of notions to the two branches
(irreducible representations, Pontryagin-type duality, etc.), the
techniques used in practice and the attendant technical difficulties
are often specific to either the analytic or the algebraic
framework. For instance, as we will see below, for locally compact
quantum groups one obstruction to obtaining the types of results we
seek will be the lack of {\it integrability} (in the sense of
\cite{KKS}) for a quantum group action on a non-commutative space.

On the other hand, for algebraic quantum groups, perhaps not
surprisingly, the representation theory of quantum groups (i.e. the
study of modules and comodules over the respective Hopf algebras) is
anchored to purely algebraic technical conditions such as flatness
over subalgebras or coflatness over quotient coalgebras
(\cites{tak,schn,MW} provide a selection).

In the present paper we study various problems relating to the lattice
of quantum subgroups of a quantum group, and do so as a case study in
the contrasts and similarities between the two approaches to quantum
group theory mentioned above. The results in question are analogues of
Noether's isomorphism theorems, the Dedekind modularity law for the
lattice of subgroups of a group, the so-called butterfly lemma due to
Zassenhauss, the Schreier refinement theorem and Jordan-H\"older
theorem; we also study ancillary topics such as normal series and
composition series for quantum groups.

We develop the necessary machinery in parallel in order to accommodate
both the operator-algebraic and the purely algebraic perspectives, but
the results overlap when we discuss compact or discrete quantum
groups, both of which are amenable to a non-analytic treatment via the
general theory of CQG algebras \cite{DK}.

We thus sometimes obtain two independent proofs for certain results
where the overlap occurs. This should allow the analytically-minded
reader to appreciate the advantages that the algebraic machinery
sometimes affords. Similarly, the more algebraically inclined reader
may catch a glimpse of the difficulties specific to the topological
nature of locally compact quantum groups. For these reasons, it is our
hope that the problems at hand might elicit the interest of both
operator algebraists and pure algebraists working on quantum groups
from typically different perspectives.

There are connections to prior work in various particular cases. In
\cite{MR3119236} S. Wang establishes the Third Isomorphism Theorem for
compact quantum groups. On the other hand, S. Natale \cite{natale}
proves a Second Isomorphism theorem, a Zassenhauss lemma, a Schreier
refinement theorem and a Jordan-H\"older theorem for finite
dimensional Hopf algebras. The First Isomorphism Theorem (holding
trivially in the algebraic case) was recently established in the
locally compact context under an \textit{integrability} assumption
(see \cite{KKS}).


The structure of the paper is as follows. 

\Cref{1} is devoted to some preliminary material, such as useful definitions and results. We expand on this background material in \Cref{2}, where we prove some general auxiliary results about the lattice of closed quantum subgroups which might be of some independent interest. We hope to develop some of these in future work. 

In \Cref{se.lr} we treat the case of linearly reductive quantum groups, proving analogues of various results from classical group theory: the Second and the Third  Isomorphism Theorems, Dedekind's modularity law, the Zassenhauss lemma, Schreier's refinement theorem and the Jordan-H\"older theorem. 

Finally \Cref{se.lc} parallels \Cref{se.lr} in terms of the results we prove: the Second and Third  Isomorphism Theorems, the  Dedekind modularity law for locally compact quantum group, etc. As mentioned before, some of the subtleties intrinsic to the locally compact topology will arise in the form of additional assumptions we will have to make in many of the results.

\subsection*{Acknowledgements} We thank A. Skalski and U. Franz for reading and improving the   early version of this paper. 
AC was partially supported by an AMS-Simons Travel Grant and by NSF grant DMS - 1565226.
PK was partially supported by the NCN (National Centre of Science) grant 2015/17/B/ST1/00085. 
SOH acknowledges support by the PHC PROCOPE 33446QL.  We thank A. Skalski for useful discussions

\section{Preliminaries}\label{1}

\subsection{Preliminaries for locally compact quantum groups}\label{subse.lc}
The theory of locally compact quantum groups is formulated in terms of operator algebras. Operator algebra  theory is divided into two parts. In order to explain this division  let us fix a  Hilbert space $\sH$. The set of all bounded operators acting on $\sH$ forms a normed $*$-algebra which we denote by $B(\sH)$. This algebra, except the norm topology carries a host of locally convex topologies:  strong, $\sigma$-strong, weak, $\sigma$-weak and others. Although the aforementioned  distinction does not depend on the choice of the topology listed above, we choose the   $\sigma$-weak topology on $B(\sH)$ for its description. In this paper the scalar product $(\cdot|\cdot)$ on a Hilbert space will be linear in the second variable.
\begin{definition} Let $I$ be a directed set and $\sH$   a Hilbert space. Let $(T_i)_{i\in I}$ be a net of bounded operators acting on $\sH$ and let $T\in B(\sH)$. We say that  $(T_i)_{i\in I}$ $\sigma$-weakly converges to $T\in  B(\sH)$ if for all sequences $(\xi_n)_{n\in\bN}, (\zeta_n)_{n\in\bN}\in\sH$ satisfying $\displaystyle\sum_{n=1}^\infty\|\xi_n\|^2<\infty, \sum_{n=1}^\infty\|\zeta_n\|^2<\infty$ we have 
\[\lim_{i}\sum_{n=1}^\infty(\xi_n|T_i\zeta_n) = \sum_{n=1}^\infty(\xi_n|T\zeta_n).\]  We say that  $(T_i)_{i\in I}$  $\sigma$-$*$ strongly converges to  $T$ if \begin{align*}\lim_{i}\sum_{n=1}^\infty \|(T-T_i)\zeta_n\|^2 =& 0,\\ \lim_{i}\sum_{n=1}^\infty \|(T^*-T^*_i)\zeta_n\|^2  =& 0.\end{align*} 
\end{definition}
\begin{definition}\label{def:CvN}
Let $\sH$ be a Hilbert space.  
\begin{itemize}
    \item[(i)] A  $*$-subalgebra $\sA$ of $B(\sH)$ which is closed in the norm topology is  called a concrete $\C^*$-algebra.
    \item[(ii)] A  unital $*$-subalgebra $\sN$ of $B(\sH)$ which is closed in the $\sigma$-weak topology is called von Neumann  algebra.
\end{itemize}
\end{definition}
 Usually  we shall skip the term concrete and say that $\sA\subset B(\sH)$  is a $\C^*$-algebra. 
 
Let  $X\subset B(\sH)$ be a non-empty subset. The commutant  $X'$ of $X$ is defined as 
\[X' = \{y\in B(\sH):xy = yx \textrm{ for all } x\in X\}.\] We shall write $X'' = (X')'$. The famous bicommutant theorem implies  that a   *-subalgebra $\sA\subset B(\sH)$ is a von Neumann algebra if and only if $\sA = (\sA')'$.

Let $\sY$ be a topological vector space and $\sX $  a non-empty subset of $\sY$. The closure of the linear span of $\sX$ will be denoted $\sX^{\bcls} $. If $\sX$ is a $\C^*$-algebra then the norm closure of the linear span of $\sX$ will also be denoted by $\sX^{\|\cdot\|-\bcls}$. If $\sX$ is a von Neumann algebra then the $\sigma$-weak closure of the linear span of $\sX$ will   be denoted by $\sX^{ \cls}$. 

Given a pair of  $\C^*$-algebras $\sA_1\subset B(\sH_1)$ and $\sA_2\subset B(\sH_2)$, the (spatial) tensor product $\sA_1\otimes  \sA_2\subset B(\sH_1\otimes\sH_2 )$ is defined as 
\[\sA_1\otimes  \sA_2 = \{x\otimes y:x\in \sA_1,y\in \sA_2\}^{\|\cdot\|-\bcls}.\]
Similarly, given a pair of  von Neumann  algebras $\sN_1\subset B(\sH_1)$ and $\sN_2\subset B(\sH_2)$,  we  define 
\[\sN_1\vtens   \sN_2 = \{x\otimes y:x\in \sN_1,y\in \sN_2\}^{\cls}.\]

The Banach dual of the Banach space $(B(\sH), \|\cdot\|)$ will be denoted by $B(\sH)^*$. For $\zeta,\xi\in \sH$ we define a bounded functional $\omega_{\zeta,\xi}\in B(\sH)^*$: $\omega_{\zeta,\xi}(T) = (\zeta|T\xi)$ for all $T\in B(\sH)$. 
Let us consider a subset  $X$ of $B(\sH)^*$:
\[X = \{\omega_{\zeta,\xi}:\zeta,\xi\in\sH\}.\]
We shall denote  $ B(\sH)_* = X^{\|\cdot\|-\bcls}$. We say that  $\omega\in B(\sH)_*$ is a normal functional on $B(\sH)$. We have $(B(\sH)_*)^* = B(\sH)$ and the $\sigma$-weak topology coincide with the weak $*$-topology on $B(\sH)$.  
 
There is an abstract version of a (concrete) concept of a $ \C^*$-algebra and a von Neumann algebra formulated in \Cref{def:CvN}.  
\begin{definition}
Let $\sA$ be a Banach $*$-algebra. We say that $\sA$ is a $\C^*$-algebra if the $\C^*$-identity $\|a^*a\| = \|a\|^2$ holds for all $a\in \sA$. Let $\sN$ be a $\C^*$-algebra. We say that   $\sN$ is   a $W^*$-algebra if $\sN$ admits a predual Banach space. 
\end{definition}
Every $\C^*$-algebra can be identified with a concrete $\C^*$-algebra. A $\C^*$-algebra $\sN$ can be identified with a von Neumann algebra if and only if $\sN$ is a $W^*$-algebra. The predual space of a $W^*$-algebra $\sN$ is uniquely determined by $\sN$ and it will be denoted by $\sN_*$
 
 In this paper we shall always consider concrete $\C^*$-algebras which are non-degenerate.
 \begin{definition}
 Let $\sA\subset B(\sH)$ be a concrete $\C^*$-algebra. We say that $\sA$ is   non-degenerate if   $\bigcap_{a\in \sA} \ker a = \{0\}$.

 Let $\sA\subset B(\sH)$ be a non-degenerate $ \C^*$-algebra. The   $ \C^*$-algebra 
 \[\M(\sA) = \{x\in B(\sH): xa,ax\in \sA\}\] is called a {\it multiplier $\C^*$-algebra} of $\sA$. \end{definition}
 It can be proved that the multiplier $\C^*$-algebra   $\M(\sA)$ of a concrete $\C^*$-algebra $\sA\subset B(\sH)$ does not depend on the embedding $\sA\subset B(\sH)$. To be more precise if  $\sK$ is a Hilbert space and  $\pi:\sA\to B(\sK)$ is an injective $*$-homomorphisms then  $\pi(\sA)$ is a $\C^*$-subalgebra of $B(\sK)$ and $\M(\sA)$ and $\M(\pi(\sA))$ are isomorphic (as $\C^*$-algebras). 
 
Let  $\sB$ be a $\C^*$-algebra  and $\sC$   a $\C^*$-subalgebra of $\M(\sB)$. The set 
\[\{cb:c\in\sC, b\in\sB\}^{\|\cdot\|-\bcls}\] will be denoted $\sC\sB$. Let $\pi:\sA\to\M(\sB)$ be a $*$-homomorphism. We say that $\pi$ is non-degenerate if $\pi(\sA)\sB = \sB$. The set of non-degenerate $*$-homomorphisms from $\sA$ to $\M(\sB)$ will  be denoted by $\Mor(\sA,\sB)$. It can be checked that there exists a unique $*$-homomorphism  $\overline{\pi}:\M(\sA)\to\M(\sB)$ such that for all $x\in \M(\sA)$ and $a\in \sA$ we have $ \pi (xa) = \overline{\pi}(x)\pi(a)$. In particular $\overline{\pi}$ extends $\pi$ and in what follows this extension will be denoted by $\pi$. Note that for $\pi\in\Mor(\sA,\sB)$ and $\rho\in\Mor(\sB,\sC)$ we can form $\rho\circ\pi\in\Mor(\sA,\sC)$. This composition  gives rise to the category of $\C^*$-algebras with $\Mor(\sA,\sB)$ being morphisms. 

Let $\sN$ and $\sM$ be von Neumann  algebras. A unital $*$-homomorphism $\pi:\sN\to\sM$ is said to be normal if  it is continuous in the $\sigma$-weak topologies. 
The set of positive elements of $\sN$ will be denoted by $\sN^+$.
\begin{definition}
Let $\sM$ be a von Neumann algebra. A weight on $\sM$ is a function $\psi:\sM^+\to\mathbb{R}_{\geq 0}\cup\{\infty\}$ such that $\psi(0) = 0$, $\psi(x+y)= \psi(x)+\psi(y)$ and $\psi(tx) = t\psi(x)$ for all $t\in  \mathbb{R}_{\geq 0}$ and $x,y\in \sM^+$. We say that $\psi$ is normal if it is lower semi-continuous in the $\sigma$-weak topology on $\sM^+$. We say that $\psi$ is semifinite if the set 
\[\{x\in\sM^+:\psi(x)<\infty\}\] is $\sigma$-weakly dense in $\sM^+$. We say that that $\psi$ is faithful if  $\psi(x) = 0\implies x = 0$. A normal semifinite faithful weight will be called an n.s.f. weight.
\end{definition}

Let $\psi$ be an n.s.f. weight on $\sM$. Then we define the following sets:
\begin{itemize}
\item $\mathcal{M}^+_\psi = \{x\in\sM^+:\psi(x)<\infty\}$,
\item  $\mathcal{N}_\psi = \{x\in\sM: \psi(x^*x)<\infty\}$,
\item  $\mathcal{M}_\psi = \textrm{Lin}\mathcal{M}^+_\psi$.
\end{itemize}
Let us note that $\mathcal{N}_\psi$ forms a left ideal in $\sN$. It can be checked that $\mathcal{M}_\psi = \textrm{Lin}\{x^*y:x,y\in\mathcal{N}_\psi\}$ and $\psi$ yields a  linear map $\psi:\mathcal{M}_\psi\to\mathbb{C}$. 

The GNS-construction based on $\psi$ is a triple $(\sH_\psi,\pi_\psi,\eta)$ where $\sH_\psi$ is a Hilbert space $\pi_\psi:\sN\to B(\sH_\psi)$ is a normal $*$-homomorphism and $\eta:\mathcal{N}_\psi\to\sH_\psi$ is a  $\sigma$-$*$ strongly closed linear map  such that
\begin{itemize}
\item $(\eta(x)|\eta(y)) = \psi(x^*y)$ for all $x,y\in\mathcal{N}_\psi$,
\item $\eta(xy) = \pi_\psi(x)\eta( y)$ for all $x\in\sN$ and $y\in\mathcal{N}_\psi$.
\end{itemize}
A GNS construction for $\psi$ always exists and is essentially unique. 

For  the theory of locally compact quantum groups we refer to \cites{univ,KV}. 
\begin{definition}
 A von Neumann algebraic locally compact quantum group is a quadruple $\mathbb{G} = (\sM, \Delta_\GG,\varphi_\GG,\psi_\GG)$, where $\sM$ is a von Neumann algebra, $\Delta:\sM\to\sM\vtens\sM$ is a normal injective $*$-homomorphism satisfying 
 \[(\Delta\otimes\id)\circ\Delta = (\id\otimes\Delta)\circ\Delta\] 
  and $\varphi_\GG$ and $\psi_\GG$ are, respectively,  normal semifinite  faithful left and right invariant   weights on $\sM$, i.e. 
\begin{align*}
\psi_\GG((\id\otimes\omega)(\Delta(x)) &= \psi_\GG(x)\\
  \varphi_\GG((\omega\otimes\id)(\Delta(x)) &= \varphi_\GG(x)
  \end{align*} for all $x\in \sM^+$ and $\omega\in \sM_*^+$. 
\end{definition}
Let $\GG$ be a locally compact quantum group.   We shall use a notation $\sM = \Linf(\GG)$.
 The GNS Hilbert space of the right Haar weight $\psi_\GG$ will be denoted by  $\Ltwo(\GG)$ and the corresponding GNS map will be denoted by $\eta_\GG$. $\GG$ is assigned with  the \emph{antipode}, the \emph{scaling group} and the \emph{unitary antipode} which are  denoted  by $S$, $(\tau_t)_{t\in \RR}$ and $R$.
A fundamental role in the theory of locally compact quantum groups  is played by the multiplicative unitary $\ww^\GG\in B(\Ltwo(\GG)\otimes \Ltwo(\GG))$, which   is a unique unitary operator such that 
\[\ww^\GG(\eta_\GG(x)\otimes\eta_\GG(y)) = (\eta_\GG \otimes\eta_\GG )(\Delta_\GG(x)(\I\otimes y))\] 
for all $x,y\in D(\eta_\GG)$;  
 $\ww^\GG$  satisfies the pentagonal equation $\ww^\GG_{12}\ww^\GG_{13}\ww^\GG_{23} = \ww^\GG_{23}\ww^\GG_{12}$. Note that we use the leg numbering notation, e.g. $\ww^\GG_{12} = \ww^\GG\otimes\I \in B(\Ltwo(\GG)\otimes \Ltwo(\GG)\otimes \Ltwo(\GG))$. 
 Using $\ww^\GG$  one can recover $\Linf(\GG)$ and $\Delta_\GG$
\[\begin{split}
\Linf(\GG) =&\bigl\{ (\omega\otimes\id)(\ww^\GG)\st\omega\in B(\Ltwo(\GG))_*\bigr\}^{\cls},\\
\Delta_\GG(x) =&\ww^\GG(x\otimes\I){\ww^\GG}^*.
\end{split}\] 
A locally compact quantum group $\GG$ admits a  $\C^*$-version, which  can also be recovered from $\ww^\GG$. For example the $\C^*$-algebra assigned to  $\GG$  denoted by $ \C_0(\GG)$ is given by 
\[ \C_0(\GG) = \bigl\{ (\omega\otimes\id)(\ww^\GG)\st\omega\in B(\Ltwo(\GG))_*\bigr\}^{\|\cdot\|-\bcls}.\] 
We say that $\GG$ is a compact quantum group if $\I\in\C_0(\GG)$. 
A locally compact quantum group admits a dual object $\hh\GG = (\Linf(\hh\GG),\Delta_{\hh\GG},\varphi_{\hh\GG},\psi_{\hh\GG})$. For the detailed description of the Haar weights $\varphi_{\hh\GG},\psi_{\hh\GG}$ we refer to \cite{KV}; let us only mention that we have  $\Ltwo(\hh\GG) = \Ltwo(\GG)$. The multiplicative unitary assigned to $\hh\GG$ 
  is given by $\ww^{\hh\GG}=\boldsymbol{\sigma}({\ww^{\GG}})^*$, where $\boldsymbol{\sigma}$ denotes the flipping morphism   $\boldsymbol{\sigma}(a\otimes b) = b\otimes a$.  In particular we have 
\[\begin{split} 
\Linf(\hh\GG)&=\bigl\{( \omega\otimes\id)(\ww^{\hh\GG})\st\omega\in B(\Ltwo(\GG))_*\bigr\}^{\cls},\\ \Delta_{\hh\GG}(x)&=\ww^{\hh\GG}(x\otimes\I){\ww^{\hh\GG}}^*.
 \end{split}\] Moreover \[\C_0(\hh\GG) = \bigl\{ (\omega\otimes\id)(\ww^{\hh\GG})\st\omega\in B(\Ltwo(\GG))_*\bigr\}^{\|\cdot\|-\bcls}\] and we have $\ww^\GG\in\M(\C_0(\hh\GG)\otimes \C_0(\GG))$. 
\begin{definition}
Let $\GG$ be a locally compact quantum group. The opposite locally compact quantum group  $\GG^{\textrm{op}}$ is defined as $(\Linf(\GG),\Delta_{\GG}^{\textrm{op}}, \psi_\GG,\varphi_\GG)$ where   $\Delta_{\GG}^{\textrm{op}} = \boldsymbol{\sigma}\circ\Delta_\GG$. We say that  $\GG$ is  abelian if $\Delta_\GG =\Delta_{\GG^{\textrm{op}}}$; in other words $\GG$   is abelian if and only if $\hh\GG$ is classical, i.e. $\Linf(\hh\GG)$ is commutative.
\end{definition} 
\begin{definition}
Let $\GG$ be a locally compact quantum group,  $\sN$ a von Neumann algebra and $\alpha:\sN\to\Linf(\GG)\vtens\sN$ a normal, unital injective $*$-homomorphism. We say that $\alpha$ is a  left  action of $\GG$ on $\sN$ if 
\[(\Delta_\GG\otimes\id)\circ\alpha = (\id\otimes\alpha)\circ\alpha.\] 
We say that the action $\alpha$ is {\it integrable} if the set 
\[\{x\in\sN^+: (\psi_\GG\otimes\id)(\alpha(x))\in\sN^+\}\] is $\sigma$-weakly dense in $\sN^+$. 
\end{definition}
For an action  $\alpha$ of $\GG$ on $\sN$ we have (see \cite[Corollary 2.6]{KS})
\begin{equation}\label{eq:pod_cond1}\sN = \{(\omega\otimes\id)(\alpha(x)):\omega\in\Linf(\GG)_*, x\in\sN\}^{\cls }.\end{equation} 

 We also have a right counterpart of the concept of an action and the integrability condition.  

In the sequel we shall often use the  right  adjoint action $\beta:\Linf(\hh\GG)\to\Linf(\hh\GG)\vtens\Linf(\GG)$ of $\GG$ on $\Linf(\hh\GG)$ where  
\begin{equation}\label{eq:adact}\beta (x) = \ww^{\GG}(x\otimes\I){\ww^\GG}^* \end{equation} for all $x\in \Linf(\hh\GG)$.  

A von Neumann subalgebra $\sN\subset\Linf(\GG)$ is called 
\begin{itemize}
\item \emph{Left coideal}  if $\Delta_\GG(\sN)\subset\Linf(\GG)\vtens\sN$;
\item \emph{Invariant subalgebra} if $\Delta_\GG(\sN)\subset\sN\vtens\sN$;
\item \emph{Baaj-Vaes subalgebra} if $\sN$ is an invariant subalgebra of $\Linf(\GG)$ which is preserved by unitary antipode $R$  and the scaling group $(\tau_t)_{t\in\RR} $ of $\GG$;
\item \emph{Normal} if $\ww^\GG(\I\otimes\sN){\ww^\GG}^*\subset\Linf(\hh\GG)\vtens\sN$;
\item \emph{Integrable} if the set of integrable elements with respect to the right Haar weight $\psi_\GG$ is dense in $\sN^+$; in other words, the restriction of $\psi_\GG$ to $\sN$ is semifinite.
\end{itemize}
In the sequel a von Neumann subalgebra of $\Linf(\GG)$ which is a left coideal will be called a $\GG$-coideal or simply a coideal.  Note that $\Delta_\GG|_{\sN}$ is an action of $\GG$ on $\sN$. In particular (see \Cref{eq:pod_cond1})
\begin{equation}\label{eq:pod_cond}\sN = \{(\omega\otimes\id)(\Delta_\GG(x)):\omega\in\Linf(\GG)_*, x\in \sN\}^{\cls }.\end{equation}

Let $\sN\subset \Linf(\GG)$ be a Baaj-Vaes subalgebra. The restriction of $\Delta_\GG$ to $\sN$ will be denoted by  $\Delta_{\sN} :\sN\to\sN\vtens\sN$. We shall often use  the so called  Baaj-Vaes theorem   \cite[Proposition 10.5]{BValg}, which states that  $(\sN,\Delta|_{\sN})$   admits a structure of a locally compact quantum group. To be more precise there exists a pair of  n.s.f. weights $ \varphi_{\sN},\psi_{\sN}$ on $\sN$  such that $(\sN,\Delta_\GG|_{\sN},\varphi_{\sN},\psi_{\sN})$ is a locally compact quantum group.
\begin{definition}\label{def:lattice_coideals}
Let $\GG$ be a locally compact quantum group.  The set of $\GG$-coideals will be denoted by  $ \mathcal{CI}(\GG)$. We equip  $\mathcal{CI}(\GG)$ with the poset structure: for $\sN,\sM\in\mathcal{CI}(\GG)$ we write $\sN\le\sM$ if  $\sN\subset\sM$. The poset $(\mathcal{CI}(\GG),\le)$  admits two operations $\wedge,\vee$ 
\begin{itemize}
\item $\sN\wedge\sM = \sN\cap\sM$,
\item $\sN\vee\sM = \{xy:x\in\sN,y\in\sM\}''$.
\end{itemize}
$(\mathcal{CI}(\GG),\le,\wedge,\vee)$ forms a lattice which will be called a {\it lattice of coideals} of $\GG$. 

   The subset of $\mathcal{CI}(\GG)$ of normal $\GG$-coideals  will be  denoted $\mathcal{NCI}(\GG)$. The subset of $\mathcal{CI}(\GG)$ of Baaj-Vaes subalgebras of $\Linf(\GG)$ will be denoted $\mathcal{BV}(\GG)$.
 
 It is easy to check that $\mathcal{NCI}(\GG)$ and $\mathcal{BV}(\GG)$ form sublattices of $\mathcal{CI}(\GG)$. Similarly $\mathcal{NCI}(\GG)\cap\mathcal{BV}(\GG)$ forms a sublattice of $\mathcal{CI}(\GG)$.
 
\end{definition}
\begin{remark}\label{rem:coduality0}
    Using \cite[Theorem 3.9]{embed} we get a bijective map  $\mathcal{CI}(\GG)\ni\sN\mapsto\tilde{\sN}\in\mathcal{CI}(\hh\GG)$ where 
    \[\tilde{\sN} =\sN'\cap\Linf(\hh\GG).\]  The coideal $\tilde{\sN}$
    is said to be a codual of $\sN$ and the map
    $\sN\mapsto\tilde{\sN}$ is denoted by $\cd$. Note that
    $\cd:\mathcal{CI}(\GG) \to \mathcal{CI}(\hh\GG)$ is an
    anti-isomorphism of lattices:
    \[ \begin{split}
        \cd(\sN)&\le\cd(\sM) \textrm{ iff }  \sM\le\sN,\\
        \cd(\sN\wedge\sM) &= \cd(\sN)\vee\cd(\sM),\\
        \cd(\sN\vee\sM) &= \cd(\sN)\wedge\cd(\sM).
    \end{split}\]
    Moreover $\cd^2 = \id$ (note that   the coduality $\mathcal{CI}(\hh\GG)\to\mathcal{CI}(\GG))$ is also denoted by $\cd$).
    \end{remark}

\subsection{Preliminaries for linearly reductive quantum groups}\label{subse.lr}

In \Cref{se.lr}   we work with Hopf algebras over an algebraically closed field $k$, regarded as the function algebras of the quantum groups in question. For this reason, we typically speak of either the quantum group $\bG$ or the Hopf algebra $\cO(\bG)$ associated to it. Unless specified otherwise, antipodes are assumed to be bijective. For general background on coalgebras, bialgebras or Hopf algebras (which we recall somewhat briefly and selectively) the reader may consult e.g. \cites{Sweedler,abe,dnr,rad_book}; the various papers we cite are also good sources on specific points that arise in the course of the discussion below.

We use Sweedler notation for the comultiplication of a Hopf algebra (or more generally coalgebra) $H$, writing 
\begin{equation*}
  \Delta:x\mapsto x_1\otimes x_2
\end{equation*}
for the comultiplication $\Delta:H\to H\otimes H$. The reader should note that the symbol  $\otimes$ has double meaning in this paper - one in the context of $\C^*$-algebras and other in the context of algebras. Counits and antipodes are denoted by $\varepsilon$ and $S$. Finally, for a linear subspace $V\subseteq H$ of $H$, we denote 
\begin{equation*}
  V^-:=\mathrm{ker}(\varepsilon|_V). 
\end{equation*}

We denote categories of left / right modules over an algebra $A$ by ${}_A\cM$ and $\cM_A$ respectively. Similarly, the categories of left / right $C$-comodules for a coalgebra $C$ are denoted by ${}^C\cM$ and $\cM^C$ respectively. Following standard terminology (see e.g. \cite[Definition 1.4]{git}), the quantum group $\bG$ is {\it linearly reductive} when $\cO(\bG)$ is cosemisimple, i.e. its category $\cM^{\cO(\bG)}$ is semisimple. 

Every coalgebra is the union of its finite-dimensional subcoalgebras (the so-called fundamental theorem of coalgebras; e.g. \cite[Theorem 2.2.1]{Sweedler}), and cosemisimple coalgebras are direct sums of their {\it simple} subcoalgebras, i.e. those that have no proper non-zero subcoalgebras \cite[$\S$14]{Sweedler}. This latter decomposition is dual to the usual decomposition of semisimple algebras as (finite) products of simple algebras. In fact, simple coalgebras are dual to simple algebras, and hence, since we are working over an algebraically closed field, all simple subcoalgebras are of the form $M_n^*$ (duals of the matrix algebras $M_n=M_n(k)$).   

We will also deal with {\it discrete} quantum groups in a slightly more general setting than in \Cref{se.lc}. \Cref{def.red}, summarizing our conventions, will be sufficient for our purposes.

\begin{definition}\label{def.red}
  The category $\mathcal{QG}$ of {\it quantum groups} over a fixed field $k$ is the category opposite to that of Hopf algebras over $k$ with bijective antipode. 

  The category $\mathcal{RQG}$ of {\it linearly reductive quantum groups} over a fixed field $k$ is the full subcategory of $\mathcal{QG}$ consisting of cosemisimple Hopf algebras. 

  The category $\mathcal{DQG}$ of {\it algebraic discrete quantum groups} over a fixed field $k$ is the opposite category $\mathcal{RQG}^{op}$.

We often drop the adjective `algebraic' below.
\end{definition}

\begin{remark}
In other words, we regard discrete quantum groups as dual to linearly reductive groups. This mimics the usual machinery in the locally compact case, except that we allow here arbitrary algebraically closed fields of arbitrary characteristic, and there are no $*$ structures.   
\end{remark}

In the spirit of \Cref{def.red}, we regard the underlying Hopf algebra
$\cO(\hh{\bG})$ of a linearly reductive quantum group $\hh{\bG}$ as
the group algebra of its discrete Pontryagin dual $\bG$ of $\hh{\bG}$
and (working over the algebraically closed field $k$) use the notation
\begin{equation*}
  k \bG  = \cO(\hh{\bG})
\end{equation*}
when we want to emphasize this point of view.

One particular class of cosemisimple Hopf algebras are the {\it CQG algebras} of \cite{DK}, which in the context of \Cref{subse.lc} are dense complex Hopf $*$-subalgebras of $L^\infty(\bG)$ for some {\it compact} quantum group $\bG$. Some of the results of \Cref{se.lr} only apply to CQG algebras.

\begin{definition}\label{def.lr-sbgp}
  A {\it quantum subgroup} of the linearly reductive quantum group $\bG$ is a quotient Hopf algebra $\cO(\bG)\to \cO(\bK)$.  

Let $\bG$ be a discrete quantum group.  A {\it quantum subgroup} of $\bG$ is a Hopf subalgebra $A\subseteq k\bG$.
\end{definition}

\begin{remark}
  Note that quantum subgroups of a discrete quantum group are automatically discrete, because cosemisimplicity is preserved by passing to Hopf subalgebras.  Thus denoting $A = k\bK$ in the second part of \Cref{def.lr-sbgp} we see that a quantum subgroup $\bK$ of a discrete quantum group $\GG$ is a Hopf subalgebra $ k\bK\subseteq k\bG$.  A  quantum subgroups of linearly reductive quantum groups need not be reductive, however: consider the classical situation whereby the function algebra of $GL_2(\bC)$ surjects onto that of the subgroup of upper triangular invertible matrices.  
\end{remark}

Quantum subgroups of a given quantum group form a lattice both in $\mathcal{QG}$ and in $\mathcal{DQG}$.

\begin{definition}\label{def.alg-latt}
  Let $\cO(\bG)\to \cO(\bH_i)$, $i=1,2$ be two quantum subgroups of $\bG\in \mathcal{QG}$. Then, the {\it intersection} $\bH_1\wedge \bH_2$ whose underlying Hopf algebra $\cO(\bH_1\wedge \bH_2)$ is defined as the quotient of $\cO(\bG)$ by the smallest Hopf ideal $I$ invariant under the inverse $S^{-1}$ of the antipode of $\bG$, and which contains both ideals
  \begin{equation}\label{eq:ghi}
    \mathrm{ker}\left(\cO(\bG)\to \cO(\bH_i)\right),\ i=1,2.
  \end{equation}
  Similarly, the subgroup $\bH_1\vee \bH_2$ {\it generated by $\bH_i$} is defined as the object in $\mathcal{QG}$ whose underlying Hopf algebra is the quotient of $\cO(\bG)$ by the largest Hopf ideal invariant under $S^{-1}$ contained in both \Cref{eq:ghi}.

Now let $k\bH_i\subseteq k\bG$ be two quantum subgroups of a discrete quantum group $\bG$. Then, the {\it intersection} $\bH_1\wedge \bH_2$ is the discrete quantum group whose underlying group algebra $k(\bH_1\wedge \bH_2)$ is the intersection of $k\bH_i$ in $k\bG$. 
  
Similarly, the subgroup {\it generated by $\bH_i$} is defined as the discrete quantum group whose underlying Hopf algebra is the Hopf subalgebra of $k\bG$ generated as an algebra by $k\bH_i$.    
\end{definition}
We leave it to the reader to check that in both cases the operations $\wedge$ and $\vee$ are well defined and turn the sets of quantum subgroups into lattices.

\begin{remark}\label{re.wedge}
Classically, the intersection $\bH\wedge \bK$ can
be defined as the pullback
\begin{equation*}
  \begin{tikzpicture}[auto,baseline=(current  bounding  box.center)]
    \path[anchor=base] (0,0) node (HcapK) {$\bH\wedge \bK$} +(2,.5) node (H) {$\bH$} +(4,0) node (G) {$\bG$} +(2,-.5) node (K) {$\bK$};
         \draw[right hook->] (HcapK) to [bend left=6] (H);
         \draw[right hook->] (H) to [bend left=6] (G);
         \draw[right hook->] (HcapK) to [bend right=6] (K);
         \draw[right hook->] (K) to [bend right=6] (G);
  \end{tikzpicture}
\end{equation*}
in whatever category of groups is convenient (linear algebraic, etc.). The analogue in the $\mathcal{QG}$ case of \Cref{def.alg-latt} is the observation that we have a pushout 
\begin{equation*}
  \begin{tikzpicture}[auto,baseline=(current  bounding  box.center)]
    \path[anchor=base] (0,0) node (HcapK) {$\cO(\bH\wedge \bK)$} +(2,.5) node (H) {$\cO(\bH)$} +(4,0) node (G) {$\cO(\bG)$} +(2,-.5) node (K) {$\cO(\bK)$};
         \draw[->>] (H) to [bend right=6] (HcapK);
         \draw[->>] (G) to [bend right=6] (H);
         \draw[->>] (K) to [bend left=6] (HcapK);
         \draw[->>] (G) to [bend left=6] (K);
  \end{tikzpicture}
\end{equation*}
in the category of algebras, or equivalently, that of Hopf algebras (or Hopf algebras with bijective antipode). In other words, the left hand corner is universal among quotients that make the diagram commutative. 
\end{remark}

We will make frequent (mostly implicit) use of an algebraic version of the correspondence $\cd$ from \Cref{subse.lc} throughout \Cref{se.lr}. We elaborate on the construction here. 

First, for any Hopf algebra $H$, define $\mathcal{CI}(H)$ as the set of {\it right} coideal subalgebras $A\subseteq H$, i.e. those subalgebras for which 
\begin{equation*}
  \Delta(A)\subseteq A\otimes H.
\end{equation*}
(in opposition to \Cref{subse.lc}, we use right rather than left coideals in order to preserve compatibility with much of the literature on Hopf algebras accessible through our references). 

Now, for each $ A\subseteq H$ in $\mathcal{CI}(H)$ denote 
\begin{equation*}
  \cd(A) = H/HA^-. 
\end{equation*}
 
This is a left module quotient coalgebra of $H$ in the sense of \Cref{def:left_mod_co} (e.g. \cite[Proposition 1]{tak}), which justifies denoting the set of such module quotients by $\mathcal{MQ}(H)$. 
\begin{definition}\label{def:left_mod_co}
  A (left) {\it module coalgebra} over a Hopf algebra $H$ is a coalgebra $C$ equipped with an $H$-module action
\begin{equation*}
  H\otimes C\to C
\end{equation*}
that is a coalgebra map. 

A (left) {\it module quotient coalgebra} is a module coalgebra $H\otimes C\to C$ as before equipped with a surjection $H\to C$ of module coalgebras, i.e. a surjection that is both a coalgebra morphism and a morphism of left $H$-modules. 
\end{definition}

On the other hand, given $\pi:H\to C$ in $\mathcal{MQ}(H)$, define $\cd(\pi)$ (or usually $\cd(C)$ by a slight abuse of notation) to be 
\begin{equation*}
  \{x\in H\ |\ \pi(x_1)\otimes x_2 = \pi(\I)\otimes x\}. 
\end{equation*}
 It can be shown to be an object in $\mathcal{CI}(H)$ (\cite[Proposition 1]{tak}). 

Note that we are using the symbol $\cd$ for two different maps, relating $\mathcal{CI}$ and $\mathcal{MQ}$ in two opposite directions. They are not, in general, mutual inverses; that requires additional technical conditions, as we now recall.

\begin{definition}\label{def.flat}
  Let $\iota:A\to H$ be an algebra map. $H$ is left (resp. right) {\it faithfully flat} if the functor 
  \begin{equation*}
    -\otimes_AH:\cM_A\to \cM_H
  \end{equation*}
  resp. 
  \begin{equation*}
    H\otimes_A-: {}_A\cM\to {}_H\cM
  \end{equation*}
  preserves morphism injectivity. 

  Dually, let $\pi:H\to C$ be a coalgebra map. Then, $H$ is left (resp. right) {\it faithfully coflat} if the functor 
  \begin{equation*}
    -\square_CH:\cM^C\to \cM^H
  \end{equation*}
  resp. 
  \begin{equation*}
    H\square_C-: {}^C\cM\to {}^H\cM
  \end{equation*}
  preserves morphism surjectivity.  
\end{definition}
The notion of tensor coproduct $\square$, dual to that of tensor product, (see e.g. \cite[$\S$1]{tak}) can be briefly described as follows: 

Given a right $C$-comodule 
\begin{equation*}
  \rho_V:V\to V\otimes C
\end{equation*}
and a left $C$-comodule 
\begin{equation*}
  \rho_W:W\to C\otimes W,
\end{equation*}
$V\square_CW$ is the subspace of $V\otimes W$ on which the two arrows
\begin{equation*}
  \begin{tikzpicture}[auto,baseline=(current  bounding  box.center)]
    \path[anchor=base] (0,0) node (vw) {$V\otimes W$} +(4,0) node (vcw) {$V\otimes C\otimes W$};
         \draw[->] (vw) to [bend left=6] node [pos=.5,auto] {$\scriptstyle \rho_V\otimes\mathrm{id}_W$} (vcw);
         \draw[->] (vw) to [bend right=6] node [pos=.5,auto,swap] {$\scriptstyle \mathrm{id}_V\otimes\rho_W$} (vcw);
  \end{tikzpicture}
\end{equation*}
agree.

Then, part of the content of \cite[Theorems 1]{tak} is that 
\begin{equation*}
  \cd^2:\mathcal{CI}(H) \to \mathcal{CI}(H)
\end{equation*}
restricts to the identity to those $A\in \mathcal{CI}(H)$ for which $H$ is left $A$-faithfully flat. 

Similarly, \cite[Theorem 2]{tak} says (among other things) that 
\begin{equation*}
  \cd^2:\mathcal{MQ}(H)\to \mathcal{MQ}(H)
\end{equation*}
restricts to the identity on those $\pi:H\to C$ over which $H$ is right faithfully coflat.

We will be concerned almost exclusively with situations where either $\iota:A\to H$ or $\pi:H\to C$ is a Hopf algebra map. For that reason, we make the following simple observation (whose proof, being a simple computation, we leave to the reader).

\begin{lemma}\label{le.ad-inv}
Let $\pi:H\to C$ be a quotient Hopf algebra, and set $A=\cd(\pi)$. Then, $A$ is invariant under the right adjoint action of $H$ on itself defined by 
\begin{equation*}
  a\triangleleft x = S(x_1) ax_2 
\end{equation*}
\qedhere
\end{lemma}

One consequence of \Cref{le.ad-inv} that we will use later is

\begin{corollary}\label{cor.ad-inv}
  Let $A,B\in \mathcal{CI}(H)$ and suppose $A=\cd(\pi)$ for some $\pi:H\to C$ in $\mathcal{MQ}(C)$. Then, the linear span
  \begin{equation*}
    BA = \mathrm{span}\{ba\ |\ a\in A,\ b\in B\}
  \end{equation*}
is a coideal subalgebra. 
\end{corollary}
\begin{proof}
Since it is clear that the space in question is a coideal, it suffices to show that it is a subalgebra. Specifically, we have to prove that for any $a\in A$ and $b\in B$, the product $ab$ belongs to $BA$. This follows from the identity
\begin{equation*}
  ab = b_1 (S(b_2) a b_3), 
\end{equation*}
together with the fact that according to \Cref{le.ad-inv} the parenthetic factors on the right hand side belong to $A$. 
\end{proof}

Let us also record the dual version of \Cref{le.ad-inv} (the proof is entirely analogous; we once more do not include it):

\begin{lemma}\label{le.coad-inv}
  Let $\iota:A\to H$ be an inclusion of Hopf algebras, and set
  $\pi:H\to C$ to be $\cd(A)$. Then, the left adjoint coaction of $H$
  on itself defined by
\begin{equation*}
  x\mapsto x_1 S(x_3)\otimes x_2
\end{equation*}  
descends to a coaction of $H$ on $C$ through the quotient $\pi:H\to C$.
\qedhere
\end{lemma}

We will often encounter the situation when {\it both} $\iota:A\to H$ and $\pi:H\to C$ are morphisms in $\mathcal{QG}$. Given the results recalled briefly above on the importance of (co)flatness, we fix our terminology as follows.

\begin{definition}\label{def.exact}
An {\it exact sequence} of quantum groups is a diagram 
\begin{equation}\label{eq:exact}
  k\to \cO(\bH)\to \cO(\bG)\to \cO(\bK)\to k
\end{equation}
in $\mathcal{QG}$ where the second arrow is an inclusion, the third arrow is a surjection, $\cd$ interchanges these two arrows, and moreover $\cO(\bG)$ is (co)flat over $\cO(\bH)$ (respectively $\cO(\bK)$). In this case we denote $\bH = \bG/\bK$.

The quantum subgroup $\cO(\bG)\to \cO(\bK)$ of $\bG$ is {\it normal} if it fits into an exact sequence \Cref{eq:exact}. 

The discrete quantum subgroup $k\bK\subseteq k\bG$ of $k\bG\in \mathcal{DQG}$ is {\it normal} if the inclusion in question is the second arrow in an exact sequence \Cref{eq:exact}. 
\end{definition}

\begin{remark}
Cf. \cite{AD}, where the definition of an exact sequence is the same, minus the (co)flatness conditions.   
\end{remark}

\section{Lattice of closed quantum subgroups: basic facts}\label{2}
\begin{definition}\label{def:lattice_coideals1}
Let $\GG$ be a locally compact quantum group.  The  lattice   $\mathcal{BV}(\hh\GG)$   will be denoted $\mathcal{QS}(\GG)$ and called a {\it  lattice of   quantum subgroups} of $\GG$.

The lattice $\mathcal{NCI}(\hh\GG)\cap\mathcal{BV}(\hh\GG)$ will be denoted $\mathcal{NQS}(\GG)$ and called a {\it  lattice of   normal   quantum subgroups} of $\GG$. 
\end{definition}
 Let $\sN\in\mathcal{QS}(\GG)$. Using  Baaj-Vaes theorem, we conclude the existence of   a locally compact quantum group $\HH$ such that $\sN = \Linf(\hh\HH)$. Thus when convenient we will   write $\HH\in\mathcal{QS}(\GG)$. Similarly for $\HH_1,\HH_2\in\mathcal{QS}(\GG)$ we  write $\HH_1\wedge\HH_2,\HH_1\vee\HH_2$. 
  \begin{remark}\label{rem:coduality}

    Let $\HH\in\mathcal{QS}(\GG)$ and  let  $\cd:\mathcal{CI}( \GG)\to \mathcal{CI}(\hh\GG)$ be the coduality (see \Cref{rem:coduality0}). Then $\cd(\Linf(\hh\HH))\in\mathcal{CI}(\GG)$ is denoted by $\Linf(\GG/\HH)$. It can be checked that (see e.g. \cite{KSProj})
    \begin{itemize}
    \item $\cd(\mathcal{QS}(\GG))\subset \mathcal{NCI}(\GG)$,
    \item $\cd(\mathcal{NQS}(\GG)) = \mathcal{NQS}(\hh\GG)$.
    \end{itemize} If   $\HH\in\mathcal{NQS}(\GG)$ then the normal quantum subgroup   $\cd(\HH)\in\mathcal{NQS}(\hh\GG)$  is denoted by $\hh{\GG/\HH}$. For the concept of short   exact sequence of locally compact quantum groups we refer to \cite[Definition 3.2.]{Vain_Vaes}. Up to natural isomorphisms all examples are of the form   
    \begin{equation}\label{short_exact_seq}  \bullet\to\HH\to \GG\to \GG/\HH\to\bullet\end{equation} where $\bullet$ denotes a trivial group. Since $\cd^2 = \id$ we also have the dual exact sequence 
   \begin{equation}\label{short_exact_seq1} \bullet\to\hh{\GG/\HH}\to \hh\GG\to \hh{\HH}\to\bullet.\end{equation}
 
  \end{remark}

Let $\GG$ be a locally compact quantum group.  As formulated in \Cref{def:lattice_coideals1},   a closed quantum subgroup of $\GG$ corresponds to    a Baaj-Vaes subalgebra of $\Linf(\hh\GG)$. In particular  a locally compact quantum group $\GG$ can be assigned with a quantum subgroup $\mathcal{Z}(\GG)\le\GG$ which is called a center of $\GG$: by definition $\Linf(\hh{\mathcal{Z}(\GG)})$ is the largest Baaj-Vaes subalgebra contained in the center of the von Neumann algebra $\Linf(\hh\GG)$. In particular  $\mathcal{Z}(\GG)$ is a normal quantum  subgroup of $\GG$ and  one can form the quotient group $\GG/\mathcal{Z}(\GG)$. For the detailed description of the corresponding exact sequence  \[ \bullet\to\mathcal{Z}(\GG)\to \bG\to \GG/\mathcal{Z}(\GG)\to\bullet\] see \cite{KSS}. In what follows we shall describe the quantum analog of the quotient of $\GG$ by its commutator subgroup.

\begin{proposition}\label{thm:quot_comm}
  Let $\GG$ be a locally compact quantum group and let us consider 
  \begin{equation}\label{def:M}\sM = \{x\in\Linf(\GG): (\id\otimes\Delta_\GG^{\textrm{op}})(\Delta_\GG(x)) = (\id\otimes\Delta_\GG)(\Delta_\GG(x))\}.\end{equation} 
  Then $\sM$ is a Baaj-Vaes subalgebra of $\Linf(\GG)$. The   quantum group $\HH$ such that $\sM = \Linf(\HH)$  is abelian. Let  $\sN$ be another Baaj-Vaes subalgebra  and $\bL$ be the locally compact quantum group assigned to $\sN$. If $\bL$ is abelian then  $\sN\subset \sM$.
\end{proposition}
\begin{proof}
  Clearly $\sM$ is a von Neumann subalgebra of $\Linf(\GG)$. We shall first show that $\Delta_\GG(\sM)\subset \Linf(\GG)\vtens\sM$. Let $x\in\sM$. Then 
  \[\begin{array}{cc}(\id\otimes\id\otimes\Delta_\GG)(\id\otimes\Delta_\GG)(\Delta_\GG(x)) &= (\Delta_\GG\otimes\id\otimes\id )(\id\otimes\Delta_\GG)(\Delta_\GG(x))\\&=(\Delta_\GG\otimes\id\otimes\id )(\id\otimes\Delta^{\textrm{op}}_{\GG})(\Delta_\GG(x))\\&=(\id\otimes\id\otimes\Delta^{\textrm{op}}_\GG )(\id\otimes\Delta_{\GG})(\Delta_\GG(x))
  \end{array}\] and we get $\Delta_\GG(x)\in\Linf(\GG)\vtens\sM$.
  
  Using \Cref{eq:pod_cond} and \Cref{def:M} we see that $\Delta_\GG|_{\sM} = \Delta^{\textrm{op}}_\GG|_{\sM}$. In particular $\Delta_\GG(\sM)\subset \sM\vtens\sM$. 
  
  The $\tau_t^\GG$-invariance of $\sM$ follows easily from the relation $\Delta_\GG\circ\tau_t^\GG = (\tau_t^\GG\otimes\tau_t^\GG)\circ \Delta_\GG$. Since $\Delta_\GG\circ R^\GG = (R^\GG \otimes R^\GG)\circ\Delta^{\textrm{op}}_\GG$ and $\Delta_\GG|_{\sM} = \Delta^{\textrm{op}}_\GG|_{\sM}$ we get $R^\GG(\sM)\subset \sM$. Summarizing $\sM$ forms a Baaj-Vaes subalgebra. 
  
  If $\sN\subset \Linf(\GG)$ is a Baaj-Vaes subalgebra such that  $\Delta_\GG|_{\sN} = \Delta^{\textrm{op}}_\GG|_{\sN}$ then it is clear that for all $x\in \sN$ the condition $(\id\otimes\Delta^{\textrm{op}}_\GG)(\Delta_\GG(x)) = (\id\otimes\Delta_\GG)(\Delta_\GG(x))$ holds, i.e. $\sN\subset \sM$.
\end{proof}

\begin{remark}\label{re.lrg-class}
  In other words, $\sM\in \Linf(\bG)$ introduced in \Cref{thm:quot_comm} is the largest cocommutative Baaj-Vaes subalgebra. This means that it corresponds to the largest classical closed quantum subgroup of $\hh\bG$.   
\end{remark}

\begin{example}
Let $\GG$ be a classical locally compact  group. Adopting the notation of \Cref{thm:quot_comm} we see that $f\in \sM$ if for all (up to measure zero subset) $(p,q,r)\in \bG^3$  we have $f(pqr) = f(prq)$. This condition is equivalent with $f(pqrq^{-1}r^{-1}) = f(p)$. Thus $f\in \sM$ if and only if $f$ is constant on the cosets of the commutator subgroup $\bN\subset \GG$ where $\bN$ is defined as the smallest closed subgroup of $\GG$ containing  $\{qrq^{-1}r^{-1}:q,r\in\bG\}$. In conclusion, we have $\sM = \Linf(\bG/\bN)$. 
\end{example}

\begin{remark}\label{re.no-comm}
Now, let $\GG$ be a locally compact quantum group and $\sM\subset \Linf(\GG)$ the Baaj-Vaes algebra described in \Cref{thm:quot_comm}. In general      a normal quantum subgroup $\bN\trianglelefteq\bG$ such that $\sM = \Linf(\GG/\bN)$ does not exist. Actually such $\bN$  exists if and only if  $\sM$, viewed as a coideal $\Linf(\GG)$, is normal; the normality of $\sM$ in turn  is  equivalent with the equality 
\begin{equation}\label{eq:normal_comm}
  \ww^\GG_{13}\ww^\GG_{14}(\id\otimes\Delta_\GG)(\Delta_\GG(x))_{234}{\ww_{14}^\GG}^*{\ww_{13}^\GG}^*  = \ww^\GG_{14}\ww^\GG_{13}(\id\otimes\Delta_\GG)(\Delta_\GG(x))_{234}{\ww^\GG_{13}}^*{\ww^\GG_{14}}^*
\end{equation} 
being satisfied for all $x\in \sM$. \Cref{ex.no-comm} below shows that \Cref{eq:normal_comm} does not always hold; when it does, we call $\bN \trianglelefteq\bG$ as above the {\it commutator subgroup} of $\GG$.   
\end{remark}

\begin{example}\label{ex.no-comm}
  As indicated in \Cref{re.no-comm} above, the largest cocommutative Baaj-Vaes subalgebra of $L^\infty(\bG)$ is not, in general, of the form $L^\infty(\bG/\bN)$ for a normal closed quantum subgroup $\bN\trianglelefteq \bG$. To see this, note that upon dualizing, the claim takes the form, that there exist locally compact quantum groups with the property that the largest classical closed quantum subgroup is not normal (see \Cref{re.lrg-class}). 

For examples of this latter phenomenon, consider one of the free unitary groups $U_n^+$ for some $n\ge 2$ (these are the quantum groups whose underlying CQG algebras $A_u(n)$ are defined in \cite{wang-free} as being freely generated by $n\times n$ unitary matrix of generators $u_{ij}$ such that $(u_{ij}^*)$ is also unitary). 

It's largest classical quantum subgroup is the ordinary unitary group $U_n$ obtained as the object dual to the largest commutative CQG quotient algebra of $A_u(n)$, whereas it is known \cite[Corollary 12]{free-unit} that proper normal quantum subgroups of $U_n^+$ are contained in the common center $\bT$ of $U_n<U_n^+$.  
\end{example}

Let us move on to the discussion of  {\it morphisms} of locally compact quantum groups. This  requires the universal $\C^*$-version of a given locally compact quantum group $\GG$ (see e.g. \cite{univ}).
 The universal version $\C_0^u(\GG)$  of $\C_0(\GG)$    is equipped  with a comultiplication $\Delta_\GG^u \in\Mor(\C_0^u(\GG),\C_0^u(\GG)\otimes \C_0^u(\GG))$. The multiplicative unitary $\ww^\GG\in\M(\C_0(\hh\GG)\otimes\C_0(\GG))$ admits the universal lift  $\WW^\GG\in\M(\C_0^u(\hh\GG)\otimes \C_0^u(\GG))$.  The reducing morphisms for $\GG$ and $\hh\GG$ will be denoted by 
 \[\begin{split}\Lambda_\GG&\in\Mor(\C_0^u(\GG),\C_0(\GG)),\\\Lambda_{\hh\GG}&\in\Mor(\C_0^u(\hh\GG),\C_0(\hh\GG))\end{split}\]  respectively. Then 
\[(\Lambda_{\hh\GG}\otimes\Lambda_\GG)(\WW^\GG)=\ww^\GG\]
We shall also use the half-lifted versions of $\ww^\GG$
 \begin{align*} \Ww^\GG&=(\id\otimes\Lambda_\GG)(\WW^\GG)\in\M(\C_0^u(\hh\GG)\otimes \C_0(\GG)),\\ \wW^\GG&=(\Lambda_{\hh\GG}\otimes\id)(\WW^\GG) \in\M(\C_0(\hh\GG)\otimes \C_0^u(\GG)) \end{align*} 
which satisfy the appropriate versions of the pentagonal equation 
 \begin{align*}
\Ww^\GG_{12}\Ww^\GG_{13}\ww^\GG_{23}&=\ww^\GG_{23}\Ww^\GG_{12},\\
\ww^\GG_{12}\wW^\GG_{13}\wW^\GG_{23}&=\wW^\GG_{23}\ww^\GG_{12}.
\end{align*} 
The half-lifted versions of the comultiplications will be denoted by  $\Delta_r^{r,u}\in\Mor(\C_0(\GG),\C_0(\GG)\otimes \C_0^u(\GG))$ 
and 
$\hh{\Delta}_r^{r,u}\in\Mor(\C_0(\hh\GG),\C_0(\hh\GG)\otimes \C_0^u(\hh\GG))$, e.g.  
\[\begin{split}
\Delta_r^{r,u}(x)&=\wW^\GG(x\otimes\id){\wW^\GG}^*,~~~~~x\in\C_0(\GG).  
\end{split} \]
  We have 
\begin{equation}\label{LDDrL}
\begin{split}
(\Lambda_\GG\otimes\id)\circ\Delta_\GG^u &= \Delta_r^{r,u}\circ\Lambda_\GG,\\
(\Lambda_{\hh\GG}\otimes\id)\circ\Delta_{\hh\GG}^u &= \hh{\Delta}_r^{{r,u}} \circ\Lambda_{\hh\GG}.
\end{split}
\end{equation}

Given two locally compact quantum groups $\GG$ and $\HH$, a morphism  $\Pi:\HH\to\GG$ (see e.g. \cite{SLW12}) is represented by a \cst-morphism $\pi\in \Mor(\C_0^u(\GG), \C_0^u(\HH))$ intertwining the respective coproducts:
\[ (\pi \ot \pi) \circ \Delta_{\GG} = \Delta_{\HH} \circ \pi.\] It can be equivalently described via:
\begin{itemize} 
\item a  \emph{bicharacter} from $\HH$ to $\GG$, i.e.\ a unitary $V \in \Linf(\hh{\GG})\vtens\Linf(\HH) $ such that
\[ (\Delta_{\hGG} \ot \id_{\C_0(\HH)}) (V) = V_{23} V_{13},\]
\[ (\id_{\C_0(\hGG)} \ot \Delta_{\HH}) (V) = V_{12} V_{13}.\] In fact $V\in\M(\C_0(\hh\GG)\otimes\C_0(\HH))$ and 
$V= (\id \ot \Lambda^{\HH} \circ \pi)(\wW^{\GG})$. We shall also use $\vV = (\id \ot  \pi)(\wW^{\GG})\in\M(\C_0(\hh\GG)\ot\C_0^u(\HH))$. 
\item   a \emph{right quantum group homomorphism} i.e.\  an action  $\alpha:\Linf(\bG)\to\Linf(\bG)\vtens\Linf(\bH)$ of $\bH$ on $\Linf(\bG)$ satisfying 
\[(\Delta_{\bG}\otimes\id)\circ\alpha = (\id\otimes\alpha)\circ\Delta_{\bG}\] 
In fact
$\alpha(x) = V(x\otimes\I)V^*$. We shall also use the obvious  left version of the concept of a right quantum group homomorphism, which is refereed to  as a \emph{left quantum group homomorphism}.
\end{itemize}
 Let $\Pi:\HH\to \GG$. The right quantum group homomorphism assigned to $\Pi$ will be denoted  $\alpha_{\Pi}$ or $\alpha_{\HH\to\GG}$ when convenient. 
\begin{example}\label{cl_q_sub_mor}
Let $\GG$ be a locally compact quantum group and $\HH\le \GG$. Since $\Linf(\hh\HH)$  is a Baaj-Vaes subalgebra of $\Linf(\hh\GG)$,  the multiplicative unitary $\ww^\HH\in\Linf(\hh\HH)\vtens\Linf(\HH)$ can viewed as an element   $V\in \Linf(\hh\GG)\vtens\Linf(\HH)$. Since  $V$ is a bicharacter from $\HH$ to $\GG$,   $\HH\le \GG$ yields a morphism from $\HH$ to $\GG$. 

Let $\GG$ and $\HH$ be locally compact quantum groups and  $\Pi:\HH\to \GG$ a morphism. We say that $\Pi$ identifies $\HH$ with a closed quantum subgroup of $\GG$ if there exists a normal injective $*$-homomorphism $\gamma:\Linf(\hh\HH)\to\Linf(\hh\GG)$ such that $V = (\gamma\otimes\id)(\ww^\HH)$ (see \cite{DKSS}). 
\end{example}

Clearly, a closed  quantum subgroup $\HH\in\mathcal{QS}(\GG)$ is normal if and only 
if 
\[\beta(\Linf(\hh\HH))\subset\Linf(\hh\HH)\vtens\Linf(\GG) \] where $\beta$ is the adjoint action \Cref{eq:adact}.
Let us also note the following result whose classical version is well known. 
\begin{proposition}\label{pr.norm-ab}
Let $\GG$ be a locally compact quantum group and $\bN\trianglelefteq\GG$ an abelian normal quantum  subgroup of $\GG$. Then for every $x\in\Linf(\hh\bN)$ we have 
\[\ww^\GG(x\otimes\I)\ww^{\GG*}\in\Linf(\hh\bN)\vtens\Linf(\GG/\bN)\] In particular the restriction of the adjoint action $\beta$ to $\Linf(\hh\bN)$ gives rise to  the action of $\GG/\bN$ on $\Linf(\hh\bN)$. Conversely if the adjoint action restricted to $\Linf(\hh\bN)$ gives rise to the action of $\GG/\bN$ then $\bN$ is abelian. 
\end{proposition}
\begin{proof}Let $\alpha:\Linf(\GG)\to \Linf(\GG)\vtens\Linf(\bN)$ be the right quantum group homomorphism assigned to $\bN\trianglelefteq\GG$. Since $(\id\otimes\alpha)(\ww^\GG) = \ww^\GG_{12}\ww^{\bN}_{13}$ and $\bN$ is abelian (i.e. $\Linf(\hh\bN)$ is commutative), we conclude that 
\[(\id\otimes\alpha)(\ww^\GG(x\otimes\I)\ww^{\GG*}) = \ww^\GG_{12}\ww^{\bN}_{13}(x\otimes\I\otimes\I)\ww^{\bN*}_{13}\ww^{\GG*}_{12} = (\ww^\GG(x\otimes\I)\ww^{\GG*})\otimes\I\]
for all  $ x\in\Linf(\hh\bN)$. Thus $\ww^\GG(x\otimes\I)\ww^{\GG*}\in\Linf(\hh\bN)\vtens\Linf(\GG/\bN)$. 

Conversely, the condition  \[\ww^\GG_{12}\ww^{\bN}_{13}(x\otimes\I\otimes\I)\ww^{\bN*}_{13}\ww^{\GG*}_{12} = (\ww^\GG(x\otimes\I)\ww^{\GG*})\otimes\I\] holds  if and only if $\ww^{\bN}(x\otimes\I)\ww^{\bN*}=(x\otimes\I)$ for all $x\in\Linf(\hh\bN)$, which is equivalent to $\bN$ being abelian. 
\end{proof}

Let $\Pi:\bH\to \bG$ be a morphism of locally compact quantum groups. It turns out that $\Pi$ cannot (in general) be assigned with  a   quantum analog of the kernel subgroup  $\ker\Pi\le \bH$ (the case $\bH=\bG$ and $\Pi$ being a projection $\Pi^2 =\Pi$ was thoroughly studied in \cite{KSProj}). In particular $\Pi$ cannot be assigned with    the exact sequence
\begin{equation}\label{ex_seq_basic}
  \bullet\to\ker\Pi\to \bH\to \bH/\ker\Pi\to\bullet.
\end{equation}  
 As noted in  \cite{KKS},  the quantum analog of $\bH/\ker\Pi$ can  always be constructed. In what follows we shall provide a number of  descriptions  of $\bH/\ker\Pi$ and formulate the condition  which yields the existence of  $\ker\Pi$  entering the exact sequence  \Cref{ex_seq_basic}. 

The von Neumann algebra  $\Linf(\bH/\ker\Pi)$ is defined as (see \cite[Definition 4.4]{KKS}
\begin{equation}\label{Hker1}\Linf(\bH/\ker\Pi) = \{(\omega\otimes\id)(V):\omega\in\Linf(\hh\bG)_*\}^{\sigma - \textrm{cls}}.\end{equation}
To be more precise the right hand side of \Cref{Hker1} forms a      Baaj-Vaes subalgebra of $\Linf(\bH)$, thus yields a locally compact quantum group which we denote $\bH/\ker\Pi$.  Since $V= (\id \ot \Lambda^{\HH} \circ \pi)(\wW^{\GG})$ the following holds 
\begin{equation}\label{def:2quotker}\Linf(\bH/\ker\Pi) = \{\Lambda_{\HH} ( \pi(x)): x\in\C_0^u(\GG)\}^{\sigma - \textrm{cls}}\end{equation} which is the second description of $\Linf(\bH/\ker\Pi)$. The third description is the subject of \cite[Theorem 4.7]{KKS}:
\begin{equation}\label{eq:quot_ker}\Linf(\bH/\ker\Pi) =\{(\omega\otimes\id)(\alpha(x)):\omega\in\Linf(\bG)_*, x\in\Linf(\bG)\}''.\end{equation}
In what follows we shall give a simple proof of a slightly stronger version of  \Cref{eq:quot_ker}.
\begin{lemma}\label{q_def}
Given a morphism $\Pi:\bH\to \bG$ we have 
\[\Linf(\bH/\ker\Pi) =\{(\omega\otimes\id)(\alpha(x)):\omega\in\Linf(\bG)_*, x\in\Linf(\bG)\}^{\sigma - \textrm{cls}}\]
\end{lemma}
\begin{proof}
 
The bicharacter equation for $V$ yields
\[V_{23}\ww^\GG_{12}V_{23}^* = \ww^\GG_{12}V_{13}.\]
In particular, since 
\[\Linf(\GG) = \{(\mu\otimes\id)(\ww^\GG):\mu\in\Linf(\hh\bG)_*\}^{\sigma - \textrm{cls}}\] we have
\[\begin{split}
\{(\omega\otimes\id)&(\alpha(x)):\omega\in\Linf(\bG)_*, x\in\Linf(\bG)\}^{\sigma - \textrm{cls}}\\& =\{(\mu\otimes \omega\otimes\id)(V_{23}\ww^\GG_{12}V_{23}^*):\mu\in\Linf(\hh\GG)_*,\omega\in\Linf(\bG)_* \}^{\sigma - \textrm{cls}}\\& =\{(\mu\otimes \omega\otimes\id)(\ww^\GG_{12}V_{13}):\mu\in\Linf(\hh\GG)_*,\omega\in\Linf(\bG)_* \}^{\sigma - \textrm{cls}}\\& =\{(\mu\otimes \omega\otimes\id)( V_{13}):\mu\in\Linf(\hh\GG)_*,\omega\in\Linf(\bG)_*\}^{\sigma - \textrm{cls}}\\& =\{( \omega\otimes\id)(V):\omega\in\Linf(\bG)_*, \}^{\sigma - \textrm{cls}} = \Linf(\bH/\ker\Pi)
\end{split}\]
where in third  equality ${\sigma - \textrm{cls}}$   in the subscript    and unitarity of $\ww^\GG $   enabled us to  absorb  $\ww^\GG$ into the   functional  $\mu\otimes \omega$ without changing the resulting set. 
\end{proof}
Let $\Pi:\HH\to\GG$. 
Then the  embedding   $\Linf(\bH/\ker\Pi)\subset \Linf(\bH)$ can be interpreted as  $\hh{\bH/\ker\Pi}\le\hh\HH$. In particular  a short  exact sequence \Cref{short_exact_seq1} starting with 
\[\bullet\to \hh{\bH/\ker\Pi}\to \hh{\bH}\] exists  if and only if   $\hh{\bH/\ker\Pi}\in\mathcal{NQS}(\hh\HH)$. In this case  defining   $\Linf(\hh{\ker\Pi}) = \cd(\Linf(\bH/\ker\Pi))\in\mathcal{NQS}( \HH)$  we  get a short  exact sequence of locally compact quantum groups \Cref{ex_seq_basic}. 

A morphism  $\Pi:\HH\to\GG$ is assigned with the  dual morphism $\hh\Pi:\hh\GG\to\hh\HH$ which in terms of bicharacter is given by $\hh{V} = \sigma(V)^*$. The locally compact quantum group  $\hh\GG/\ker\hh\Pi $ will be denoted  by $\overline{\textrm{im}\Pi}$ (see   \cite[Definition 4.3]{KKS}). In particular, using \eqref{Hker1} we can see that  $V\in\Linf(\hh\GG)\vtens\Linf(\HH)$ is actually an element of $\Linf(\hh{\overline{\textrm{im}\Pi}})\vtens\Linf(\HH/\ker\Pi)$. Using the inclusions  \begin{align*}\Linf(\hh{\overline{\textrm{im}\Pi}})\vtens\Linf(\HH/\ker\Pi)&\subset\Linf(\hh\GG)\vtens\Linf(\HH/\ker\Pi)\\ \Linf(\hh{\overline{\textrm{im}\Pi}})\vtens\Linf(\HH/\ker\Pi)&\subset\Linf(\hh{\overline{\textrm{im}\Pi}})\vtens\Linf(\HH)\end{align*} we  see that a morphism   $\Pi_{\HH\to\GG}:\HH\to\GG$ induces three morphisms  $ \Pi_{\HH/\ker\Pi\to \overline{\textrm{im}\Pi}}$,  $\Pi_{\HH/\ker\Pi\to\GG}$ and $\Pi_{{\HH}\to {\overline{\textrm{im}\Pi}}}$. Using \cite[Theorem  6.2, Corollary 6.5]{KKS}  we shall now formulate the First Isomorphism Theorem for locally compact quantum groups. 
\begin{theorem}\label{thm:first_iso_thm}
  Let $\HH$ and $\GG$ be locally compact quantum groups, $\Pi:\HH\to\GG$  a morphism and let  $\Pi_{\HH/\ker\Pi\to \overline{\textrm{im}\Pi}}$,   $\Pi_{\HH/\ker\Pi\to\GG}$,   $\hh\Pi_{\hh{\overline{\textrm{im}\Pi}}\to\hh\HH}$ be the morphisms induced by $\Pi$ as described above. Then the following conditions are equivalent:
  \begin{itemize}
  \item[(i)] $\Pi_{\HH/\ker\Pi\to \overline{\textrm{im}\Pi}}$ is an isomorphism; \item[(ii)] the action $\alpha:\Linf(\overline{\textrm{im}\Pi})\to\Linf(\overline{\textrm{im}\Pi})\vtens\Linf(\HH/\ker\Pi)$ is integrable;
  \item[(iii)] $\Pi_{\HH/\ker\Pi\to\GG}$ identifies $\HH/\ker\Pi$ with a closed quantum subgroup of $\GG$;
  \item[(iv)] $\hh\Pi_{\hh{\overline{\textrm{im}\Pi}}\to\hh\HH}$ identifies $\hh{\overline{\textrm{im}\Pi}}$  with a closed quantum subgroup of $\hh\HH$. 
  \end{itemize}
\end{theorem}

 \begin{remark}
   Let $\Pi:\HH\to\GG$ be a morphism of locally compact quantum groups. Clearly $\overline{\textrm{im}\Pi}$ is abelian if and only if $\HH/\ker\Pi$ is abelian. Let $\sM\subset\Linf(\GG)$ be the   Baaj-Vaes algebra described in \Cref{thm:quot_comm}. Then $\HH/\ker\Pi$ is abelian if and only if $\Linf(\HH/\ker\Pi)\subset \sM$. For further discussion let us suppose that  $\ker\Pi\trianglelefteq\GG$ exists  (see the paragraph containing \Cref{ex_seq_basic}) and   there exists a $\bN\trianglelefteq\GG$ such that $\sM =\Linf(\bG/\bN)$. Then $\Linf(\HH/\ker\Pi)\le \Linf(\bG/\bN)$ if and only if $\bN\le\ker\Pi$. Thus in the discussed  case we get a quantum analog of the well known classical fact: the closed image of  $\Pi:\HH\to \GG$ is abelian if and only if the kernel $\ker\Pi\trianglelefteq\HH$ contains the commutator subgroup $\bN\trianglelefteq\HH$.
\end{remark}
The next lemma will be needed further. 
\begin{lemma}\label{codual_alpha}Let $\Pi:\HH\to\GG$. 
 Then 
  \[ \cd(\Linf(\HH/\ker\Pi))= \{y\in\Linf(\hh\HH):\hh\alpha(y) = y\otimes\I\}\]
\end{lemma}
\begin{proof} Let $V\in\Linf(\hh\GG)\vtens\Linf(\HH)$ be  the bicharacter corresponding  to $\Pi$. The right quantum group homomorphism $\hh\alpha:\Linf(\hh\HH)\to\Linf(\hh\HH)\vtens\Linf(\hh\GG)$ corresponding to $\hh\Pi$ is given by   
\[\hh\alpha(y) = \boldsymbol{\sigma}(V^*(\I\otimes y)V)\] 
for all $y\in\Linf(\hh\HH)$. In particular  $\hh\alpha(y) = y\otimes\I$ if and only if 
\[ (\I\otimes y)V = V(\I\otimes y)\] 
We conclude using \Cref{Hker1}. 
\end{proof}
 
\begin{example}\label{normal_subgroup}
Let us consider  $\HH\in\mathcal{NQS}(\GG)$ and the exact sequence
\[\bullet\to\HH\to \GG\to \GG/\HH\to\bullet.\]
Let us denote the morphism $\GG\to \GG/\HH$   by $\Pi$. Since $\Linf(\GG/\HH)$ is defined as a Baaj-Vaes subalgebra of $\Linf(\GG)$,  the dual morphism $\hh\Pi:\hh{\GG/\HH}\to \hh\GG$ identifies $\hh{\GG/\HH}$ with a closed quantum subgroup of $\hh\GG$.  In particular  $\overline{\textrm{im}\Pi} = \GG/\HH$ and    $\GG/\ker\Pi=\GG/\HH$. 

Let $\hh\alpha:\Linf(\hh\GG)\to\Linf(\hh\GG)\vtens\Linf(\hh{\GG/\HH})$ be the right quantum group homomorphism assigned to   $\hh\Pi:\hh{\GG/\HH}\to \hh{\GG}$. Since $\hh\Pi$ identifies $\hh{\GG/\HH}$ with a closed quantum subgroup of $\hh{\GG}$ we get  
\begin{equation}\label{basic_eq}\Linf(\hh{\GG/\HH}) = \{(\omega\otimes\id)(\hh\alpha (a)):\omega\in\Linf(\hh\GG)_*, a\in\Linf(\hh \GG)\}^{\sigma - \textrm{cls}}. \end{equation}
Finally, using \Cref{codual_alpha} we see that 
\begin{equation}\label{normal_subgroup1}\Linf(\hh\HH) = \{y\in\Linf(\hh\GG): \hh\alpha(y) = y\otimes\I\}.\end{equation}
\end{example}
\begin{lemma}\label{restriclemma} Let $\bG$ and $\bK$ be locally compact quantum groups
 $\Pi_1:\GG\to\KK$ a homomorphism  and $\hh\alpha_1:\Linf(\hh\bG)\to\Linf(\hh\bG)\vtens\Linf(\hh\bK)$ the corresponding right quantum group homomorphism. Let  $\HH\in\mathcal{QS}(\GG)$,   $\Pi_2:\HH\to\KK$  the restriction of $\Pi_1$ to $\HH\le\GG$ and $\hh\alpha_2:\Linf(\hh\bH)\to\Linf(\hh\bH)\vtens\Linf(\hh\bK)$ the  right quantum group homomorphism corresponding to $\hh\Pi_2$. Then   $\hh\alpha_2 = \hh\alpha_1|_{\Linf(\hh\HH)}$. 
\end{lemma}
\begin{proof} Let $V\in\Linf(\hh\bK)\vtens\Linf(\GG)$ and $U\in\Linf(\hh\bK)\vtens\Linf(\HH)$ be the  bicharacters corresponding to $\Pi_1$ and  $\Pi_2$ respectively. 
Let  $\vV\in\M(\C_0(\hh\KK)\otimes\C_0^u(\GG))$ be the  universal lift of $V$:
\begin{equation}\label{restr1} V_{12}^*\wW^{\bG}_{23}V_{12} =  \vV_{13}\wW^{\bG}_{23}.\end{equation} Applying quantum group morphism  $\pi\in\Mor(\C_0^u(\GG),\C_0^u(\HH))$ (corresponding to $\HH\le \GG$) to the third leg of \Cref{restr1} and reducing the result we get 
 \begin{equation}\label{restr} V_{12}^* \ww^{\bH}_{23}V_{12} = U_{13}\ww^{\bH}_{23} = U_{12}^* \ww^{\bH}_{23}U_{12}\end{equation} (note that we use the embedding $\Linf(\hh\bH)\subset \Linf(\hh\bG)$ on the left side of \Cref{restr}). We conclude by recalling that  $\hh\alpha_1$ is implemented by $\hh V$ and $\hh\alpha_2$ is implemented by $\hh U$. 
\end{proof}

We shall also need the following 
 \begin{lemma}\label{normalization}
   Let $\GG$ be a locally compact quantum group, $\sN\in\mathcal{NCI}(\GG)$ and $\sM\in\mathcal{BV}(\GG)$. Let $\HH$ be a locally compact quantum group such that $\sM=\Linf(\HH)$. Then 
   \begin{equation}\label{lemma:norm_bv1}
     \ww^\HH(\I\otimes\sN)\ww^{\HH*}\subset \Linf(\hh\HH)\vtens\sN.
   \end{equation} 
   In particular $\sN\wedge\Linf(\HH)\in\mathcal{NCI}(\HH)$. Moreover
   \begin{equation}\label{lemma:norm_bv}
     \sN\vee\Linf(\HH)=\{xy:x\in\sN,y\in\sM\}^{\sigma - \textrm{cls}}.
   \end{equation}
 \end{lemma}
\begin{proof}
 As explained in \Cref{cl_q_sub_mor} the embedding $\Linf(\HH)\subset \Linf(\GG)$ corresponds to a morphism $\Pi:\GG\to\HH$. Let $\hh\Pi:\hh\HH\to\hh\GG$ be the dual morphism and $\hh\alpha:\Linf(\hh\GG)\to\Linf(\hh\GG)\vtens\Linf(\hh\HH)$ the corresponding right quantum group homomorphism. Applying $(\hh\alpha\otimes\id)$ to the normality condition 
\[\ww^\GG(\I\otimes\sN)\ww^{\GG*}\subset \Linf(\hh\GG)\vtens\sN\] and using 
\[(\hh\alpha\otimes\id)(\ww^\GG) = \ww^\HH_{23}\ww^\GG_{13}\] we get 
\begin{equation}\label{Vnorm1}\ww^\HH_{23}\ww^\GG_{13}(\I\otimes\I\otimes\sN)\ww^{\GG*}_{13}\ww^{\HH*}_{23}\in\Linf(\hh\GG)\vtens\Linf(\hh\HH)\vtens\sN.\end{equation} 
Using \Cref{eq:pod_cond1} in the context of the $\hh\GG$-action 
\[\alpha:\sN\ni x\to \ww^\GG(\I\otimes x)\ww^{\GG*}\] we get 
\[\sN = \{(\omega\otimes\id)(\ww^\GG(\I\otimes x)\ww^{\GG*}):\omega\in\Linf(\hh\GG)_*, x\in\sN\}^{\cls }.\] Thus \Cref{Vnorm1}  implies that 
 \[ \ww^\HH(\I\otimes\sN)\ww^{\HH*}\subset \Linf(\hh\HH)\vtens\sN.\]
 
Let us fix $x\in\sN$ and $y\in\Linf(\HH)$ of the form $y= (\omega_{p,q}\otimes\id)(\ww^{\HH*})$ where $p,q\in\Ltwo( \HH)$. In order to check that \Cref{lemma:norm_bv} holds it suffices to check that 
 \begin{equation}\label{lemma:norm_bv2}xy\in\{\Linf(\HH)\sN\}^{\sigma - \textrm{cls}}.\end{equation} Indeed the latter enables us to conclude that the right hand side of \Cref{lemma:norm_bv} forms a von Neumann algebra  and this suffices since $\sN\vee\Linf(\HH)$ is the   von Neumann algebra generated by  $\sN$ and $\Linf(\HH)$. In the following computation we fix an orthonormal basis $(e_i)_{i\in I}$ of $\Ltwo(\HH)$ 
 \[\begin{array}{rl}
 xy & = x (\omega_{p,q}\otimes\id)(\ww^{\HH*})\\
 & = (\omega_{p,q}\otimes\id)(\ww^{\HH*}(\ww^\HH(\I\otimes x)\ww^{\HH*}))\\
&= \sum_{i\in I}(\omega_{p,e_i}\otimes\id)(\ww^{\HH*})(\omega_{e_i,q}\otimes\id)(\ww^\HH(\I\otimes x)\ww^{\HH*}).
 \end{array}\]
 This computation together with \Cref{lemma:norm_bv1} shows that \Cref{lemma:norm_bv2} indeed holds which ends the proof. 
\end{proof}
\begin{definition}\label{def:normalization}
Let $\GG$ be a locally compact quantum group and $\HH,\bM\in\mathcal{QS}(\GG)$. We say that $\HH$ is normalized by $\bM$ if $\HH\in\mathcal{NQS}(\HH\vee \bM)$. 

\end{definition}
Let $\HH$ and $\bM$ be as in \Cref{def:normalization}. Using \Cref{normalization} we get 
\begin{equation}\label{eq:weak_normal}\ww^{\hh\bM}(\I\otimes\Linf(\hh\HH))\ww^{\hh\bM*}\subset \Linf(\bM)\vtens\Linf(\hh\HH).\end{equation}
More generally the  following holds. 

\begin{lemma}\label{norminter}
  Let $\GG$ be a locally compact quantum group, $\HH,\bM\in\mathcal{QS}(\GG)$ and suppose that $\HH$ is normalized by $\bM$. Let $\bL\le\bM$. Then $\HH$ is normalized by $\bL$. Moreover \[\Linf(\hh{\HH\vee\bL}) = \{xy:x\in\Linf(\hh{\bL}),y\in\Linf(\hh\HH)\}^{\sigma - \textrm{cls}} \] and $\HH\wedge\bL\in\mathcal{NQS}(\bL)$. 
In particular if   $\HH\in\mathcal{NQS}(\GG)$  then $\HH$ is normalized by every $\bL\in\mathcal{QS}(\GG)$ and $\HH\wedge\bL\in\mathcal{NQS}(\bL)$.
\end{lemma}

\begin{proposition}\label{lemma:normalization_quotient}
 Let $\GG$ be  a locally compact quantum group and  $\HH,\bM\in\mathcal{QS}(\GG)$ and suppose that  $\HH$ is normalized by $\bM$. Then  $\alpha_{\bM\to\bG}(\Linf(\GG/\HH))\subset \Linf(\GG/\HH)\vtens\Linf(\bM)$.
\end{proposition}
\begin{proof}Let us first recall that 
\[\alpha_{\bM\to\bG}(x) = \ww^{\bM}(x\otimes\I)\ww^{\bM*}\] Let us fix $x\in\Linf(\GG/\HH)$, i.e. $x\in\Linf(\GG)$ and 
  \[\ww^{\bH}(x\otimes\I) = (x\otimes\I)\ww^{\bH}.\] We have to prove that
  \begin{equation}\label{eq_loc_lemma}\ww^{\bH}_{13}\ww^{\bM}_{12}(x\otimes\I\otimes\I)\ww^{\bM*}_{12}=\ww^{\bM}_{12}(x\otimes\I\otimes\I)\ww^{\bM*}_{12}\ww^{\bH}_{13}.\end{equation} Using \Cref{eq:weak_normal} we see that 
  \begin{equation}\label{eq:norm_coond_w}\ww^{\bM*}_{12}\ww^{\bH}_{13}\ww^{\bM}_{12}\subset \Linf(\hh\HH)\vtens\Linf(\bM)\vtens\Linf(\bH)\end{equation} and we compute
  \[\begin{array}{cl}
      \ww^{\bH}_{13}\ww^{\bM}_{12}(x\otimes\I\otimes\I)\ww^{\bM*}_{12} &=  \ww^{\bM}_{12}\ww^{\bM*}_{12}\ww^{\bH}_{13}\ww^{\bM}_{12}(x\otimes\I\otimes\I)\ww^{\bM*}_{12} \\
       &=  \ww^{\bM}_{12}(x\otimes\I\otimes\I)\ww^{\bM*}_{12}\ww^{\bH}_{13}\ww^{\bM}_{12}\ww^{\bM*}_{12}\\ &=  \ww^{\bM}_{12}(x\otimes\I\otimes\I)\ww^{\bM*}_{12}\ww^{\bH}_{13} 
  \end{array}\]
  where in the second equality we use \Cref{eq:norm_coond_w} and the fact that  $x\in\Linf(\hh\HH)'$  
\end{proof}

\begin{remark}\label{re.alt-norm}
  Note that in the proof of \Cref{lemma:normalization_quotient} we needed somewhat less than
  \Cref{def:normalization}: it is enough to have 
  \begin{equation*}
    \ww^{\hh\bM}(\I\vtens \Linf(\hh\bH))\ww^{\hh\bM*} \subset
    \Linf(\bM)\otimes \Linf(\hh\bH),
  \end{equation*}
  which is more akin to the classical notion of one group normalizing
  another. 
\end{remark}

The following simple observation regarding a universal property of
quotient quantum groups will come in handy repeatedly in \Cref{se.lc}.

\begin{lemma}\label{le.univ_quot}
  Let $\Pi:\bG\to \bP$ be a morphism of quantum groups, and $\bN\in\mathcal{NQS}(\bG)$. Then, $\Pi$ factors as 
\begin{equation*}
  \begin{tikzpicture}[auto,baseline=(current  bounding  box.center)]
    \path[anchor=base] (0,0) node (g) {$\bG$} +(2,-.5) node (g/k) {$\bG/\bN$} +(4,0) node (p) {$\bP$};
         \draw[->] (g) to [bend left=6] node[pos=.5,auto] {$\scriptstyle \Pi$} (p);
         \draw[->] (g) to [bend right=6] (g/k);
         \draw[->] (g/k) to [bend right=6] (p);
  \end{tikzpicture}
\end{equation*}
if and only if the composition $\bN\to\bG\to\bP$ is trivial (i.e the image of the composition is the trivial group). 
\end{lemma}
\begin{proof}
  The direct implication is clear. Conversely, suppose the composition
  $\bN\to \bG\to \bP$ is trivial. We will apply \Cref{restriclemma} to
  $\Pi_1=\Pi$ and $\bH=\bN\in\mathcal{QS}(\GG)$. The right quantum
  group homomorphism
  $\hh\alpha:\Linf(\hh\GG)\to\Linf(\hh\GG)\vtens\Linf(\hh\bP)$
  assigned to $\hh\Pi$ is given  by  
  \begin{equation*}
  \hh\alpha(x) = \boldsymbol{\sigma} (V^*)  (x\otimes 1) \boldsymbol{\sigma} (V^*),
  \end{equation*}where $x\in\Linf(\hh\GG)$  
   and 
$ V\in \Linf(\hh\bP)\vtens \Linf(\bG)$
  is the bicharacter assigned $\Pi$. Using   
  \Cref{restriclemma} we conclude that $V$ is contained in
  $ \Linf(\hh\bP)\vtens (\Linf(\hh\bN)'\cap\Linf(\bG))$, and hence must be contained in
 $ \Linf(\hh\bP)\vtens \Linf(\bG/\bN)$. 
In particular $V$ viewed as an element $\Linf(\hh\bP)\vtens \Linf(\bG/\bN)$ defines a morphism $\bG/\bN\to \bP$.   Running through the way in which bicharacters, regarded as morphisms, compose in the category of locally compact quantum groups, this means precisely that $\Pi$ factors through $\bG/\bN$. 
\end{proof}

\begin{remark}\label{re.univ_quot}
  From the perspective of the category of locally compact quantum groups, \Cref{le.univ_quot} simply says that $\bG\to \bG/\bN$ is the coequalizer of the inclusion $\bN\to \bG$ and the trivial map $\bN\to \mathds{1}\to \bG$.  

Moreover, by the self-duality of the category of locally compact quantum groups, we can conclude that the inclusion $\bN\to \bG$ of a normal subgroup is the equalizer of the arrows $\bG\to \bG/\bN$ and $\bG\to \mathds{1}\to \bG/\bN$. 
\end{remark}

In fact, we can improve on \Cref{le.univ_quot} somewhat. For future
reference, we record the result in \Cref{le.univ_quot+} below. Before its formulation let us consider  an   action $ \alpha:\sN\to \sN\vtens \Linf(\bG)$ of  a locally compact quantum group $\GG$ on a von Neumann algebra $\sN$. Then given a left quantum group homomorphism $\gamma:\Linf(\GG)\to\Linf(\HH)\vtens\Linf(\GG)$  there exists a unique action $\beta:\sN\to \sN\vtens \Linf(\bH)$ such that 
\[(\id\otimes\gamma)\circ\alpha = (\beta\otimes\id)\circ\alpha.\] In particular, given $\HH\le\GG$ we get $\beta$ which we denote by   $\alpha|_\HH:\sN\to\sN\vtens \Linf(\bH)$ and we say that  $\alpha|_\HH$ is the restriction of $\alpha$ to $\HH$. The details yielding the existence of $\beta$ are left to the reader. 
\begin{lemma}\label{le.univ_quot+}
  Let $\alpha:\sN\to \sN\vtens \Linf(\bG)$ and $\bN\trianglelefteq
  \bG$ a closed normal subgroup. Then, $\alpha$ factors as
\begin{equation*}
  \begin{tikzpicture}[auto,baseline=(current  bounding  box.center)]
    \path[anchor=base] (0,0) node (n) {$\sN$} +(3,-.5) node (ngn)
    {$\sN\vtens \Linf(\bG/\bN)$} +(6,0) node (ng) {$\sN\vtens \Linf(\bG)$};
         \draw[->] (n) to [bend left=6] node[pos=.5,auto] {$\scriptstyle \alpha$} (ng);
         \draw[->] (n) to [bend right=6] (ngn);
         \draw[->] (ngn) to [bend right=6] (ng);
  \end{tikzpicture}
\end{equation*}
through an action by $\bG/\bN$ on $\sN$ if and only if $\bN$ acts
trivially on $\sN$. 
\end{lemma}
\begin{proof}
  Once again, one implication is trivial, so we prove the other one;
  that is, we assume that the restriction $\alpha|_\bN$ of $\alpha$ to $\bN$ is
  trivial. Let $\gamma:\Linf(\GG)\to\Linf(\bN)\vtens\Linf(\GG)$ be the left quantum group homomorphism assigned to $\bN\trianglelefteq \GG$. Using the identity  $(\alpha|_{\bN}\otimes\id)\circ\alpha = (\id\otimes\gamma)\circ\alpha$ and the trviality of $\alpha|_{\bN}$ we conclude that $\alpha(\sN)\subset \sN\vtens\Linf(\bN\backslash \GG)$ where 
  \[\Linf(\bN\backslash \GG) = \{x\in\Linf(\GG):\gamma(x) = \I\otimes x\}.\] We conclude by noting that  normality of $\bN$ yields $ \Linf(\bN\backslash \GG) = \Linf(\GG/\bN)$ (see e.g. \cite[Eq. (1.4)]{KalKS}). 
 
\end{proof}

\begin{remark}
  In phrasing of \cite[$\S$4.3]{KKS}, \Cref{le.univ_quot+} says that
  the quotient of $\bG$ by the kernel of $\alpha$ factors through
  $\bG/\bN$ if and only if $\alpha|_\bN$ is trivial. 

  In addition to \Cref{le.univ_quot}, \Cref{pr.norm-ab} above is also
  a consequence of \Cref{le.univ_quot+}; in effect, the intuitive content of
  that proposition is that since the abelian normal subgroup
  $\bN\trianglelefteq \bG$ acts trivially on itself by conjugation,
  the conjugation action of $\bG$ on $\bN$ descends to a
  $(\bG/\bN)$-action.
\end{remark}

Now, we shall discuss  open quantum subgroups. Let us begin with \cite[Definition 2.2]{KalKS}. 
\begin{definition}
Let $\HH$ and $\GG$ be locally compact quantum groups and $\pi:\Linf(\GG)\to\Linf(\HH)$ a normal surjective $*$-homomorphism. We say that $\pi$ identifies $\HH$ with an open quantum  subgroup of $\GG$ if $\Delta_\HH\circ\pi = (\pi\otimes\pi)\circ\Delta_\GG$.
\end{definition}
Let $\HH$ be a locally compact quantum group which is  identified with an open quantum subgroup of $\GG$ via $\pi:\Linf(\GG)\to\Linf(\HH)$. The central support of $\pi$ (i.e. the smallest projection of $P\in\Linf(\GG)$  such that  $\pi( P) = \I$) will be denoted by $\I_\HH$ and it will be refereed to as a group-like projection assigned to $\pi$. The morphism $\pi$ defines a morphism $\Pi:\HH\to\GG$ which in terms of the bicharacter is given by $V = (\id\otimes\pi)(\ww^\GG)\in\Linf(\hh\GG)\vtens\Linf(\HH)$. Using \cite[Theorem 3.6]{KalKS} we see that $\Pi$ identifies $\HH$ with a closed quantum subgroup of $\GG$ as described in \Cref{cl_q_sub_mor}. 

Let $\HH\in\mathcal{QS}(\GG)$. Then as proved in \cite[Corollary 3.4]{KKS}, $\HH$ can be identified with an open quantum subgroup of $\GG$, (we shall say shortly that $\HH$ is open in $\GG$) if and only if the Haar weight $\psi^{\hh\GG}$ restricts to the Haar weight on $\Linf(\hh\HH)$. In other words if and only if the restriction of $\psi^{\hh\GG}$ to $\Linf(\hh\HH)\subset\Linf(\hh\GG)$ is semifinite. This in turn is equivalent with the existence of a conditional expectation $T:\Linf(\hh\GG)\to\Linf(\hh\GG)$ onto $\Linf(\hh\HH)$ satisfying (see \cite[Theorem 6.1]{KK}
\begin{equation}\label{eq:expect}(T\otimes\id)\circ\Delta_{\hh\GG} = \Delta_{\hh\GG}\circ T = (\id\otimes T)\circ\Delta_{\hh\GG}.\end{equation} 
The subset of $\mathcal{QS}(\GG)$ that consists of open quantum subgroups of $\GG$ will be denoted by $\mathcal{OQS}(\GG)$. In what follows we shall investigate the structure of $\mathcal{OQS}(\GG)$ showing in particular that it forms a lattice. 

\begin{proposition}\label{pr.aux_discr_lc}
  Let $\bG$ be a locally compact quantum group,   $\bH\in\mathcal{OQS}(\bG)$  and $\bM\le\mathcal{QS}(\bG)$. Then
   $ \bH\wedge \bM \in\mathcal{OQS}(\bM)$. 

Moreover, if 
\begin{equation*}
  T:\Linf(\hh{\bG})\to \Linf(\hh{\bH})\text{ and } T':\Linf(\hh{\bM})\to \Linf(\hh{\bH \wedge \bM})
\end{equation*}
are the expectations associated to the respective Haar weights, we have
\begin{equation*}
  T' = T|_{\Linf(\hh{\bM})}.  
\end{equation*}
\end{proposition}
\begin{proof}
   Let $T :\Linf(\hh\GG)\to\Linf(\hh\HH)$ be the conditional expectation assigned to $\HH\in\mathcal{OQS}(\bG)$ (see \Cref{eq:expect}). 
Consider the restriction of $T$ to $\Linf(\hh{\bM})\le
\Linf(\hh{\bG})$. The equality (see \Cref{eq:pod_cond}) 
\begin{equation*}
\Linf(\hh{\bM}) =\{(\omega\otimes \id )(\Delta_{\hh{\bG}}(x)): \omega\in \Linf(\hh{\bG})_*, x\in \Linf(\hh{\bM})\}^{\sigma - \textrm{cls}}
\end{equation*} together with (see \Cref{eq:expect})
 \[(\id\otimes T)\circ \Delta_{\hh{\bG}}=(T\otimes \id)\circ \Delta_{\hh{\bG}}\] imply  that 
$T(\Linf(\hh{\bM}))\subset \Linf(\hh{\bM})$. Thus we get an inclusion $T(\Linf(\hh{\bM}))\subset  \Linf(\hh{\bH})\wedge \Linf(\hh{\bM})$. 
On the other hand, this restriction is the identity on
\begin{equation*}
  \Linf(\hh{\bH})\wedge \Linf(\hh{\bM}) = \Linf(\hh{\bH\wedge
  \bM}). 
\end{equation*}
In conclusion, $T|_{\Linf(\hh{\bM})}$ is an expectation onto
$L^{\infty}(\hh{\bH\wedge \bM})$ which clearly satisfies
\Cref{eq:expect}. Both claims of the conclusion now follow from   \cite[Theorem 6.1]{KK}. 
\end{proof}

Let us also note in passing that in the particular case when $\bH\le
\bM$ we obtain:

\begin{corollary}\label{cor.passing}
If $\bH\le \bM\le \bG$ is a sequence of closed quantum group
embeddings with $\bH$ open in $\bG$ then $\bH\le \bM$ is also open.  
\qedhere  
\end{corollary}

In fact, we can improve on \Cref{cor.passing} as in the next result,
proving a strong transitivity of openness.

\begin{proposition}\label{pr.open_trans}
  Let $\bH\le \bM\le \bG$ be a chain of closed embeddings of locally
  compact quantum groups. Then, $\bH$ is open in $\bG$ if and only if 
  \begin{equation*}
    \bH\le \bM \quad \text{and} \quad \bM\le \bG
  \end{equation*}
are both open. 
\end{proposition}
\begin{proof}
 In order to see  the leftward implication `$\Leftarrow$'  we use \cite[Corollary 3.4]{KKS}. 
 
Let us check the rightward implication `$\Rightarrow$'. 
We have already seen in \Cref{cor.passing} that $\bH\le \bM$ is
open. On the other hand, \cite[Theorem 3.3]{KKS} shows that $\bM\le
\bG$ is open if and only if there is some non-zero element of
\begin{equation*}
  \Linf(\hh{\bM})\le \Linf(\hh{\bG})
\end{equation*}
that is square-integrable with respect to the (either left or right)
Haar weight of $\hh{\bG}$. An application of the same result in the
opposite direction shows that there are such elements in 
\begin{equation*}
  \Linf(\hh{\bH})\le \Linf(\hh{\bM})\le \Linf(\hh{\bG}). 
\end{equation*}
This concludes the proof.
\end{proof}
Using \Cref{pr.aux_discr_lc} and \Cref{pr.open_trans} we get: 
\begin{corollary}
  The set $\mathcal{OQS}(\GG)$ forms a sublattice of $\mathcal{ QS}(\GG)$.
\end{corollary}
\begin{remark} 
Let $\bH\le \bG$ be an open quantum subgroup and   $\omega\in \C_0^u(\widehat{\bG})^* $  the idempotent state assigned to $\bH\le \bG$ as described in \cite[Remark 6.3]{KK}. The group-like projection $\I_\HH\in \Linf(\GG)$ assigned to $\bH\le \bG$ is given by 
\[\I_\HH = (\omega\otimes\id)(\Ww^\GG).\] 
 Let us denote   $\I_\HH^u =(\omega\otimes\id)(\WW)\in \M(\C_0^u(\GG))$; applying the   reducing morphism $\Lambda_\GG\in\Mor(\C_0^u(\bG),\C_0(\bG))$ we get  $\Lambda_\GG(\I_\HH^u) =\I_\HH$. Denoting by $\pi^u\in \M(\C_0^u(\bG),\C_0^u(\bH))$ the morphism assigned to $\bH\le \bG$ we see that 
\begin{equation}\label{comput}\begin{split}(\pi^u\otimes \id)(\Delta^u(\I_\HH^u)) &=(\pi^u\otimes \id)(\Delta^u((\omega\otimes\id)(\WW)))\\ &=(\omega\otimes\id\otimes\id)((\id\otimes\pi^u)(\WW_{12}) \WW_{13})\\& = \I\otimes \I_\HH^u\end{split}\end{equation}
where in the third equality we  used the identity  $\omega(x\hat\pi^u(y)) = \omega(x)\hh\varepsilon(y))$ which holds for all $x\in \C^u_0(\hh\bG)$ and $y\in \C^u_0(\hh\bH)$. The latter can be easily concluded from the fact that the image of the conditional expectation $T^u = (\id\otimes\omega)\circ\hh\Delta^u$ is equal to $ \hat\pi^u(\C_0^u(\hh\HH))$. 

Similarly we can prove that $(\id\otimes \pi^u)(\Delta^u(\I_\HH^u)) = \I_\HH^u\otimes\I$. In particular using  \Cref{comput} and \Cref{LDDrL}   we get 
\begin{equation}\label{comput2}\begin{split}(\I_\HH\otimes\I)\Delta^{r,u}_r(\I_\HH)& = \I_\HH\otimes \I_\HH^u,\\ (\I\otimes \I_\HH)\Delta^{u,r}_r(\I_\HH) & = \I_\HH^u\otimes \I_\HH.
\end{split}
\end{equation}

Now let $\bM\le \bG$ be a closed quantum subgroup and $\rho\in\Mor( \C_0^u(\bG), \C_0(\bM))$ the corresponding morphism. The group-like projection   assigned to the open containment  $\bH\wedge
  \bM\le \bM$ (see \Cref{pr.aux_discr_lc}) will be denoted by $\I_{\HH\wedge\bM}\in \Linf(\bM)$. Using \Cref{pr.aux_discr_lc} we get 
\begin{equation}\label{qpu}
\I_{\HH\wedge\bM} =  \rho(\I_{\HH}^u)\end{equation} 
which follows from the computation  \[(\I\otimes \I_{\HH\wedge\bM})  \ww^{\bM}=(T'\otimes\id)(\ww^{\bM}) = (T\otimes\rho)(\wW^{\bG}) = (\I\otimes\rho(\I_{\HH}^u))\ww^{\bM}\] (in this computation we use the embedding $\Linf(\hh\bM)\subset \Linf(\hh\bG)$ and the notation of the proof of  \Cref{pr.aux_discr_lc}). 
Denoting the  action
\begin{equation*}
  \alpha\in\Mor(\C_0(\bG),\C_0(\bG)\otimes \C_0(\bM))
\end{equation*} describing the embedding $\bM\le \bG$ we get 
\[(\I_{\HH}\otimes\I)\alpha(\I_{\HH}) = \I_{\HH}\otimes \I_{\HH\wedge\bM}\]  
Indeed, this follows \Cref{comput2}, \Cref{qpu} and the the identity $\alpha = (\id\otimes\rho)\circ\Delta^{r,u}_r $. Similarly we get \[(\I\otimes \I_{\HH})\alpha(\I_{\HH}) = \I_{\HH\wedge\bM}\otimes \I_{\HH}.\] 
\end{remark}

\subsection{Well positioned quantum subgroups}\label{subse.well}

For subgroups $\bH\le \bG$, we will be working with the quantum
homogeneous spaces (see \Cref{rem:coduality})
\begin{equation*}
  A_\bH  = \Linf(\bG/\bH)=\cd(\Linf(\hh\HH))\subseteq \Linf(\bG). 
\end{equation*}

\begin{definition}\label{def.well}
  Let $\bH$ and $\bM$ be two closed quantum subgroups of a locally
  compact quantum group $\bG$. We say that $\bH$ and $\bM$ are {\it
  (relatively)  well positioned}  if we
  have the equality
  \begin{equation}\label{eq:well}
    A_\bH \vee A_\bM  = \{A_\bH A_\bM\}^{\cls}
  \end{equation}
(or equivalently its analogue with $\bH$ and $\bM$ reversed).  
\end{definition}

As we will see in \Cref{th.mod_lc}, the well positioning
property is relevant to the modular law for quantum subgroups of a
locally compact quantum group. Here, we discuss sufficient conditions
that ensure well positioning. Let us also note that in the algebraic context  the counterpart of  well positioning always holds as noted in \Cref{cor.ad-inv}.

\begin{proposition}\label{pr.well}
  The closed quantum subgroups $\bH,\bM\le \bG$ are relatively well positioned if 
  \begin{enumerate}
    \renewcommand{\labelenumi}{(\alph{enumi})}
     \item $\bG$ is classical;
     \item one of $\bH$ and $\bM$ is compact;
     \item one of $\bH$ and $\bM$ is normal;
     \item $\bG$ is dual-classical. 
  \end{enumerate}
\end{proposition}
\begin{proof}
We prove the different points separately, as they require different
techniques. 

\vspace{.5cm}

{\bf (a)} This is immediate: $\Linf (\bG)$ is then commutative, and
hence it does not matter in which order we multiply elements of
$A_\bH$ and $A_\bM$.

\vspace{.5cm}

{\bf (b)} The condition is symmetric, so let us assume that $\bH$ is
compact and show that $A_\bM A_\bH$ is linearly dense in $A_\bH\vee A_\bM$. We
will adapt the proof of \Cref{normalization}.

We write $\ww=\ww^\bG$ and $\varphi$ for the   Haar state on the compact
quantum group $\bH$. We further denote by $\alpha$ the canonical
coaction 
\begin{equation*}
 \Linf (\bG)\to \Linf (\bG)\vtens \Linf (\bH)
\end{equation*}
and by $V\in \Linf (\bG)\vtens \Linf (\bH)$ the bicharacter corresponding to $\alpha$.  

Note that the elements of the form 
\begin{equation}\label{eq:gen}
  x=(\mathrm{id}\otimes\varphi)\alpha((\omega_{p,q}\otimes\mathrm{id})\ww)
  = (\omega_{p,q}\otimes\mathrm{id}\otimes \varphi)(\ww_{12} V_{13})
\end{equation}
span a dense subset of $A_\bH$ for $p,q\in \Ltwo(\bG)$. Hence, it
suffices to show that an element obtained by multiplying $x$ as in
\Cref{eq:gen} and an arbitrary $a\in A_\bM$ belongs to the closure of linear span on 
$A_\bM A_\bH$. 

With this purpose in mind, we first use
\begin{equation*}
  (\mathrm{id}\otimes\alpha) \ww = \ww_{12} V_{13}
\end{equation*}
again to write
\begin{equation*}
  xa = (\omega_{p,q}\otimes\mathrm{id}\otimes \varphi)(\ww_{12}V_{13}
  (1\otimes a\otimes 1))= (\omega_{p,q}\otimes\mathrm{id}\otimes \varphi)(\ww_{12} (1\otimes a\otimes 1)\ww_{12}^*\ww_{12} V_{13})
   ). 
\end{equation*}
Using 
\begin{equation*}
  \omega_{pq}(\bullet -) = \sum_i \omega_{p,e_i}(\bullet) \omega_{e_i,q}(-), 
\end{equation*}
the expression turns into
\begin{equation*}
  \sum_i (\omega_{p,e_i}\otimes\mathrm{id})(\ww_{12}(1\otimes a)\ww_{12}^*)\cdot
  (\omega_{e_i,q}\otimes\mathrm{id}\otimes \varphi)(\ww_{12}V_{13}). 
\end{equation*}
Now, the left hand side of the `$\cdot$' symbol belongs to $A_\bM$ by
the normality condition
\begin{equation*}
  \ww(1\otimes A_\bM)\ww^*\subseteq \Linf(\hh\bG)\otimes A_\bM,
\end{equation*}
whereas the right hand side is of the same form as \Cref{eq:gen} and hence belongs
to $A_\bH$.

\vspace{.5cm}

{\bf (c)} Once again the condition is symmetric, so for the sake of
making a choice we assume $\bH$ is normal. But then $A_\bH$ is a
Baaj-Vaes subalgebra of $\Linf(\bG)$, and hence the desired result
follows from an application of \Cref{normalization} (in the form of
\Cref{lemma:norm_bv}) to $\sN=A_\bM$ and $\sM=A_\bH$.

\vspace{.5cm}

{\bf (d)} Since $\bG$ is abelian,     quantum subgroups of $\GG$ are   normal and part
(c) applies.    
\end{proof}

\begin{remark}
Let $\GG$ be a locally compact quantum group and $\sN\subset B(\Ltwo(\GG))$  a von Neumann algebra such that $\ww(\I\otimes \sN)\ww^*\subset \Linf(\hh\GG)\vtens\sN$. Let $\sM$ be a von Neumann subalgebra of $\Linf(\GG)$ equipped with a conditional expectation $E:\Linf(\GG)\to \sM$. Let $\sN\vee\sM$ be    the von Neumann algebra generated by $\sN$ and $\sM$. 
Using the  method of the proof of point (b) of \Cref{pr.well} we get
 \[\sN\vee\sM =  \{\sN \sM\}^{\cls}.\] 
\end{remark}

\section{Lattices of quantum subgroups of a linearly reductive quantum group}\label{se.lr}

In this section we tackle some analogues of the group isomorphism theorems in the setting of (mostly linearly reductive) linear algebraic quantum groups.

\subsection{The second isomorphism theorem}

We will prove a version of the second isomorphism theorem \cite[Theorem 2.26]{rot}  for function algebras of linearly reductive quantum groups, i.e. cosemisimple Hopf algebras (see \Cref{def.red}). The general setup is as follows. 

Recall from \Cref{subse.lr} that unless specified otherwise, we work over an algebraically closed field $k$ of arbitrary characteristic. $\cO(\bG)$ denotes a cosemisimple Hopf algebra, to be thought of as the algebra of regular functions on a quantum group $\bG$. We fit the latter into an exact sequence 
\begin{equation*}
  1\to \bK\to \bG\to \bG/\bK\to 1
\end{equation*}
in the sense that we have an exact sequence 
\begin{equation*}
  k\to \cO(\bG/\bK)\to \cO(\bG)\to \cO(\bK)\to k
\end{equation*}
as in \Cref{def.exact}. Note that $\cO(\bG/\bK)$ is automatically cosemisimple (being a Hopf subalgebra of a cosemisimple Hopf algebra), and hence (\cite[Theorem 2.1]{chi_cos}) the inclusion $\cO(\bG/\bK)\to \cO(\bG)$ is automatically faithfully flat both on the left and the right. It then follows \cite[Theorem 2.5]{chi_cos} that $\cO(\bK)$ is itself cosemisimple.

Assume now that we have another linearly reductive quantum subgroup $\bH\le \bG$, i.e. a quotient cosemisimple Hopf algebra $\cO(\bG)\to \cO(\bH)$. We will now examine the issue of whether or not the intersection $\bH\wedge \bK$ from \Cref{def.alg-latt} is linearly reductive.

First, define the Hopf subalgebra $\bullet$ of $\cO(\bH)$ so as to make the following diagram commute.  
\begin{equation}\label{eq:hopf}
  \begin{tikzpicture}[auto,baseline=(current  bounding  box.center)]
    \path[anchor=base] (0,0) node (G/K) {$\cO(\bG/\bK)$} +(2,.5) node (G) {$\cO(\bG)$} +(4,0) node (H) {$\cO(\bH)$} +(2,-.5) node (H/HcapK) {$\bullet$};
         \draw[right hook->] (G/K) to [bend left=6] (G);
         \draw[->>] (G/K) to [bend right=6] (H/HcapK);
         \draw[->>] (G) to [bend left=6] (H);
         \draw[right hook->] (H/HcapK) to [bend right=6] (H);
  \end{tikzpicture}
\end{equation}
$\bullet$ is then a Hopf subalgebra of $\cO(\bH)$, and hence automatically cosemisimple. It is also invariant under the adjoint actions of $\cO(\bH)$ on itself, since $\cO(\bG/\bK)$ is ad-invariant in $\cO(\bG)$. This means that $\bullet$ is of the form 
\begin{equation*}  
\cO(\bH/\bN)\subseteq \cO(\bH)
\end{equation*}
for some normal linearly reductive quantum subgroup $\cO(\bH)\to \cO(\bN)$ of $\bH$. By construction, we have a morphism
\begin{equation*}
  \bH/\bN\to \bG/\bK. 
\end{equation*}
Our goal is to argue that we have  
\begin{equation*}
  \bH\wedge \bK= \bN,
\end{equation*}
for the intersection operation $\wedge$ as in
\Cref{def.alg-latt}. According to \Cref{re.wedge}, this is achieved by the following result.

\begin{proposition}\label{pr.colim}
  In the setting above, the diagram 
\begin{equation}\label{eq:push}
  \begin{tikzpicture}[auto,baseline=(current  bounding  box.center)]
    \path[anchor=base] (0,0) node (HcapK) {$\cO(\bN)$} +(2,.5) node (H) {$\cO(\bH)$} +(4,0) node (G) {$\cO(\bG)$} +(2,-.5) node (K) {$\cO(\bK)$};
         \draw[->>] (H) to [bend right=6] (HcapK);
         \draw[->>] (G) to [bend right=6] (H);
         \draw[->>] (K) to [bend left=6] (HcapK);
         \draw[->>] (G) to [bend left=6] (K);
  \end{tikzpicture}
\end{equation}
is a pushout in the category of algebras, or equivalently, that of bialgebras, or Hopf algebras, or Hopf algebras with bijective antipode. 
\end{proposition}
\begin{proof}
  For the fact that the forgetful functor from Hopf algebras to bialgebras or algebras is a left adjoint and hence preserves colimits (such as pushouts) we refer to \cite{porst_limcolim}. Hence, we will focus on showing that the diagram is a pushout of algebras. 

The exactness of the sequence 
\begin{equation*}
  k\to \cO(\bG/\bK)\to \cO(\bG)\to \cO(\bK)\to k
\end{equation*}
implies that the kernel of the surjection $\cO(\bG)\to \cO(\bK)$ is the ideal 
\begin{equation*}
  \cO(\bG)\cO(\bG/\bK)^- = \cO(\bG/\bK)^-\cO(\bG),
\end{equation*}
where the `$-$' superscripts indicate kernels of counits. But this means that the pushout of the two right hand arrows of \Cref{eq:push} is the quotient of $\cO(\bH)$ by the ideal generated by the kernel of the counit of 
\begin{equation*}
  \cO(\bH/\bN) := \mathrm{Im}(\cO(\bG/\bK)\to \cO(\bH)). 
\end{equation*}
Finally, the general theory of exact sequences of Hopf algebras as covered in \cite{AD} and recalled in \Cref{subse.lr} above says that this is precisely right hand quotient in the sequence 
\begin{equation*}
  k\to \cO(\bH/\bN)\to \cO(\bH)\to \square \to k,
\end{equation*}
which is by definition our $\cO(\bH)\to \cO(\bN)$.
\end{proof}

\begin{remark}
  The substance of \Cref{pr.colim} is that the algebra colimit in question is automatically cosemisimple as a Hopf algebra. This is analogous to the classical fact that a normal subgroup of a linearly reductive linear algebraic group is automatically linearly reductive, as follows easily, for instance, from the classification of linearly reductive groups \cite{nagata}.  
\end{remark}

Finally, suppose $\bH$ and $\bK$ generate $\bG$ in the following representation-theoretic sense: a linear map $f:V\to W$ between comodules $V,W\in \cM^{\cO(\bG)}$ is a $\bG$-intertwiner if and only if it is both an $\bH$- and a $\bK$-intertwiner (see e.g. \cite{BCV} for the identical notion of topological generation for compact quantum groups, or \cite{chi_rfd}, where the same property is phrased in terms of the injectivity of the map from $\cO(\bG)$ into the product $\cO(\bK)\times \cO(\bH)$ in the category of Hopf algebras).

\begin{remark}
  One can show that the condition above is equivalent to $\bH\vee\bK=\bG$, for the operation `$\vee$' from \Cref{def.alg-latt}. 
\end{remark}

With all of this in place, we have 

\begin{theorem}\label{th.2nd_bis}
Let $\bH\le \bG$ and $\bK\trianglelefteq \bG$ be linearly reductive quantum subgroups of a linearly reductive quantum group. If $\bH$ and $\bK$ generate $\bG$, then the canonical morphism $\bH/\bH\wedge \bK\to \bG/\bK$ is an isomorphism. 
\end{theorem}
\begin{proof}
 By construction,
 \begin{equation*}
   \cO(\bG/\bK)\to \cO(\bH/\bH\wedge\bK)
 \end{equation*}
is onto. In order to prove injectivity and complete the proof, it suffices to show that the functor from $(\bG/\bK)$-representations to $\bH$-representations induced by the the upper composition in \Cref{eq:hopf} is full, i.e. it induces a bijection between sets of morphisms. The fact that this condition is equivalent to the bijectivity of a map of coalgebras that is known to be onto follows e.g. from \cite[Lemmas 2.2.12, 2.2.13]{schau_tann}.

Let $V$ and $W$ be finite-dimensional $(\bG/\bK)$-representations, and $f:V\to W$ an $\bH$-intertwiner between them. Since the category $\cM^{\cO(\bG/\bK)}$ of comodules over $\cO(\bG/\bK)$ is the full subcategory of $\cM^{\bG}$ consisting of objects that break up as copies of the trivial representation when restricted to $\bK$, $f$ is also a $\bK$-intertwiner. But then, by the hypothesis that $\bH$ and $\bK$ generate $\bG$, $f$ is a $\bG$- and hence a $(\bG/\bK)$-intertwiner. This completes the proof.
\end{proof}

\begin{remark}
When working over $\bC$ and all Hopf algebras in sight are CQG, \Cref{th.2nd_bis} gives an alternate proof of the case of \Cref{th.2nd} when all quantum groups are compact.   
\end{remark}

We also have a version of \Cref{th.2nd_bis} taking place in $\mathcal{DQG}$.

\begin{theorem}\label{th.2nd_bis_discr}
  If the subgroup $\bH$ and the normal subgroup $\bK$ generate the discrete quantum group $\bG$, then the canonical morphism $\bH/\bH\wedge \bK\to \bG/\bK$ is an isomorphism. 
\end{theorem}
\begin{proof}
 The hypothesis that $\bH\vee \bK=\bG$ means, in the context of algebraic discrete quantum groups, that we have
 \begin{equation*}
   k\bG = k\bH k\bK,
 \end{equation*}
and the surjectivity of the canonical map 
\begin{equation*}
  k(\bH/\bH\wedge \bK) = k\bH/(k\bH (k\bH\cap k\bK)^-)\to k\bG/k\bG k\bK^- = k(\bG/\bK)
\end{equation*}
follows from this. 

As for injectivity, it amounts to showing that those simple $k\bH$-comodules that become trivial (i.e. break up as direct sums of copies of the trivial comodule) over $k(\bG/\bK)$ are precisely those corresponding to subcoalgebras of $k\bH\cap k\bK$; this is immediate, using the fact that a $k\bG$-comodule becomes trivial over $k(\bG/\bK)$ if and only if it is a $k\bK$-comodule. 
\end{proof}

Note incidentally that a trivial version of the first isomorphism theorem is implicit in the proof of \Cref{th.2nd_bis}. For a morphism $\Pi:\bH\to \bG$  of locally compact quantum  groups,   $\bH/\ker \Pi$ is essentially the smallest   ``quotient LCQG'' $\bH\to ?$   for which $\Pi$ factors as 
\begin{equation}\label{eq:fact}
  \begin{tikzpicture}[auto,baseline=(current  bounding  box.center)]
    \path[anchor=base] (0,0) node (H) {$\bH$} +(4,0) node (G) {$\bG$} +(2,-.5) node (H/ker) {$?$};
         \draw[->] (H) to node[pos=.5,auto] {$\scriptstyle \Pi$} (G);
         \draw[->>] (H) to [bend right=6] (H/ker);
         \draw[->] (H/ker) to [bend right=6] (G);
  \end{tikzpicture}
\end{equation} (see e.g. \Cref{def:2quotker})
Similarly,   $\overline{\mathrm{im} \Pi}$   is the smallest  $?\le\GG$   such that $\Pi$ factors similarly to \Cref{eq:fact} as
\begin{equation}\label{eq:fact1}
  \begin{tikzpicture}[auto,baseline=(current  bounding  box.center)]
    \path[anchor=base] (0,0) node (H) {$\bH$} +(4,0) node (G) {$\bG$} +(2,-.5) node (H/ker) {$?$};
         \draw[->] (H) to node[pos=.5,auto] {$\scriptstyle \Pi$} (G);
         \draw[->] (H) to [bend right=6] (H/ker);
         \draw[right hook->] (H/ker) to [bend right=6] (G);
  \end{tikzpicture}
\end{equation}
In the algebraic case, the image of a Hopf algebra morphism $\cO(\bG)\to \cO(\bH)$ clearly has both factorization universality properties, and hence by default the two concepts coincide. For this reason, we do not state a First Isomorphism Theorem in the present section.

\subsection{The modular law and Zassenhaus lemma}

Throughout this subsection $\bG$ denotes a linearly reductive quantum group. We will be interested in studying its poset of subgroups $\cO(\bG)\to \cO(\bH)$.

First, recall the {\it modular law} for subgroups of a discrete group $G$: whenever $M$, $N$ and $H$ are subgroups of $G$ with $N\le H$ we have 
\begin{equation*}
  H\cap MN = (H\cap M)N,
\end{equation*}
where the juxtaposition $AB$ of subgroups $A,B\le G$ means the set 
\begin{equation*}
  \{ab\ |\ a\in A,\ b\in B\}. 
\end{equation*}
We will be interested in cases where the set products in question are actually subgroups. To this end, we first prove

\begin{proposition}\label{pr.mod}
  Let $\bN\le \bH$ and $\bM$ be normal subgroups of the linearly reductive quantum group $\bG$. Then, we have 
  \begin{equation}\label{eq:mod}
    \bH\wedge (\bM\vee\bN) = (\bH\wedge\bM)\vee\bN. 
  \end{equation}
\end{proposition}
\begin{proof}
We will dualize the picture, and study quotient Hopf algebras $\cO(\bG)\to \cO(\bullet)$ corresponding to normal quantum subgroups from the perspective of the corresponding Hopf subalgebras $A_\bullet=\cO(\bullet\backslash \bG)\subseteq \cO(\bG)$. 

This dualization procedure reverses the lattice operations on quotient Hopf algebras and Hopf subalgebras. For this reason, the Hopf subalgebra 
\begin{equation*}
  A=A_{ \bH\wedge (\bM\vee\bN)}
\end{equation*}
corresponding to the left hand side of \Cref{eq:mod} is equal to
\begin{equation}\label{eq:modl}
  A_\bH(A_\bM\wedge A_\bN). 
\end{equation}
Similarly, the Hopf subalgebra 
\begin{equation*}
  B=A_{(\bH\wedge \bM) \vee\bN}
\end{equation*}
corresponding to the right hand side is 
\begin{equation}\label{eq:modr}
  (A_\bH A_\bM)\wedge A_\bN.
\end{equation}
Now note that \Cref{eq:modl} is the sum of those simple subcoalgebras of $\cO(\bG)$ whose simple comodules 
\begin{equation*}
  V\le W\otimes X,
\end{equation*}
where $W$ is a simple $A_\bH$-comodule and $X$ is a simple comodule over both $A_\bM$ and $A_\bN$ (see the discussion on cosemisimple coalgebras in \Cref{subse.lr}). 

On the other hand, the simple comodules of \Cref{eq:modr} are characterized by the fact that they are $A_\bN$-comodules, and also embed into tensor products of the form $W\otimes X$, for simple comodules
\begin{equation}\label{eq:wx}
  W\in \cM^{A_\bH},\ X\in \cM^{A_\bM}. 
\end{equation}
Clearly, the latter property for a simple comodule $V\in \cM^{\cO(\bG)}$ is weaker than the former, and hence $A\le B$.

On the other hand, suppose the simple comodule $V\in \cM^{A_\bN}$ embeds into $W\otimes X$ with $W$ and $X$ as in \Cref{eq:wx} (and hence $V\in \cM^B$). Then we have a non-zero morphism
\begin{equation*}
  V\to W\otimes X,
\end{equation*}
which by duality gives a non-zero morphism 
\begin{equation*}
  W^*\otimes V\to X
\end{equation*}
(automatically an epimorphism, since $X$ is assumed to be simple). But since 
\begin{equation*}
  V\in \cM^{A_\bN},\ W\in \cM^{A_\bH}\subseteq \cM^{A_\bN},
\end{equation*}
we get $X\in \cM^{A_\bN}$, and hence $X$ is actually a comodule over
\begin{equation*}
  A_\bM\wedge A_\bN. 
\end{equation*}
This means that $V$ is actually an $A$-comodule, and the proof is complete. 
\end{proof}

\begin{remark}
  Alternatively, we can restate \Cref{pr.mod} as saying that the normal quantum subgroups of a linearly reductive quantum group form a modular lattice. 
\end{remark}

We can prove somewhat more when $\bG$ is a compact quantum group. As noted above, the identity
\begin{equation*}
  H\cap MN = (H\cap M)N
\end{equation*}
holds for {\it all} subgroups $N\le H\le G$ and $M\le G$. Our version (\Cref{pr.mod_cpct} below) will still not be as general as this, but we will impose just enough restrictions to ensure that classically, the product sets $MN$ and $(H\cap M)N$ are actually subgroups. To this end, we need
\begin{definition}\label{def.norm}
  A quantum subgroup $\bL\le \bG$ {\it normalizes} another $\bM\le
  \bG$ if the latter is normal in the quantum subgroup $ 
  \bM\vee\bL$. 
\end{definition}

We are now ready to state

\begin{proposition}\label{pr.mod_cpct}
  Let $\bG$ be a compact quantum group, with quantum subgroups $\bL\le \bH\le \bG$ and $\bM\le \bG$ such that $\bL$ normalizes $\bM$. Then, the equality $\bH\wedge (\bM\vee\bL) = (\bH\wedge\bM)\vee\bL$ holds. 
\end{proposition}
\begin{proof}
  As in the proof of \Cref{pr.mod}, the goal is to show that we have
  \begin{equation}\label{eq:goal}
    A_\bH(A_\bM\wedge A_\bL) = A_\bH A_\bM\wedge A_\bL,
  \end{equation}
or rather that the right hand side is contained in the left hand side (the opposite inclusion being immediate). Note that we have used \Cref{le.ad-inv,cor.ad-inv} implicitly in order to conclude that the subspace products in \Cref{eq:goal} are both coideal subalgebras.  

For any quantum subgroup $\pi:\cO(\bG)\to \cO(\bK)$ we have an expectation $E_\bK=\cO(\bG)\to A_\bK$ defined as 
\begin{equation*}
  \begin{tikzpicture}[auto,baseline=(current  bounding  box.center)]
    \path[anchor=base] (0,0) node (1) {$\cO(\bG)$} +(3,1) node (2) {$\cO(\bG)\otimes \cO(\bG)$} +(7,1) node (3) {$\cO(\bK)\otimes \cO(\bG)$} +(10,0) node (4) {$\cO(\bG)$} + (5,-1) node (5) {$A_\bK$};
         \draw[->] (1) to[bend left=6] node[pos=.5,auto] {$\scriptstyle \Delta$} (2);
         \draw[->] (2) to node[pos=.5,auto] {$\scriptstyle \pi\otimes\id$} (3);
         \draw[->] (3) to[bend left=6] node[pos=.5,auto] {$\scriptstyle h_\bK\otimes\id$} (4);
         \draw[->] (1) to[bend right=6] node[pos=.5,auto,swap] {$\scriptstyle E_\bK$} (5);
         \draw[right hook->] (5) to[bend right=6] (4);
  \end{tikzpicture}  
\end{equation*}
It is automatically an $A_\bK$-bimodule map, and intertwines $h_\bG$ and its restriction to $A_\bK$.

Now consider the expectation $E_\bL:\cO(\bG)\to A_\bL$. Applied to an element $x$ in the right hand side of \Cref{eq:goal}, it fixes $x$ that element (because $E_\bL$ acts as the identity on $A_\bL$). On the other hand, writing $x=x_\bH x_\bM$ for 
\begin{equation*}
  x_\bH\in \bH,\ x_\bM\in \bM,
\end{equation*}
 we have
 \begin{equation*}
   E_\bL(x) = x_\bH E_\bL(x_\bM)
 \end{equation*}
because $x_\bH\in A_\bH\le A_\bL$ and $E_\bL$ is the identity on $A_\bL$. In conclusion, we would be done if we could show that $E_\bL(x_\bM)\in A_\bM\wedge A_\bL$. This is taken care of by \Cref{le.aux} below. 
\end{proof}

For the next result we will use the same notation as in the proof of \Cref{pr.mod_cpct} for coideal subalgebras $A_\bL=\cO(\bL\backslash \bG)$, expectations $E_\bL:\cO(\bG)\to A_\bL$, etc. We will also denote by $\pi_\bL$ the surjection $\cO(\bG)\to \cO(\bL)$ onto the function algebra of a quantum subgroup.

\begin{lemma}\label{le.aux}
  Let $\bM$ and $\bL$ be quantum subgroups of a compact quantum group $\bG$ such that $\bL$ normalizes $\bM$. Then, we have 
  \begin{equation*}
    E_\bL(A_\bM)\subseteq A_\bM\wedge A_\bL. 
  \end{equation*}
\end{lemma}
\begin{proof}
  Since the range of $E_\bL$ is $A_\bL$, we are trying to show that $A_\bM$ is invariant under $E_\bL$. To this end, let $f\in A_\bM$ be an arbitrary element. This means by definition that
  \begin{equation}\label{eq:m_inv}
    \pi_\bM(f_1)\otimes f_2 = 1\otimes f. 
  \end{equation}
Also by definition, the expression for the expectation is
\begin{equation*}
  E_\bL(f) = h_\bL(\pi_\bL(f_1))f_2,
\end{equation*}
and hence our goal is to prove that we have
\begin{equation*}
  h_\bL(\pi_\bL(f_1))\pi_\bM(f_2)\otimes f_3 =   1\otimes h_\bL(\pi_\bL(f_1))f_2.
\end{equation*}
More generally, we will show that in fact we have 
\begin{equation}\label{eq:am_inv}
  \pi_\bL(f_1)\otimes \pi_\bM(f_2)\otimes f_3 = \pi_\bL(f_1)\otimes 1\otimes f_2.  
\end{equation}
Moreover, by substituting $  \bL\vee\bM $ for $\bL$, we may as well assume that $\bM\trianglelefteq \bL$.  

Using the defining property of the antipode, the left hand side of \Cref{eq:am_inv} equals
\begin{equation}\label{eq:5_tensorands}
  \pi_\bL(f_1 S(f_3) f_4)\otimes \pi_\bM( f_2) \otimes f_5. 
\end{equation}
The normality assumption $\bM\trianglelefteq\bL$ implies by \Cref{le.coad-inv} that the surjection $\cO(\bL)\to \cO(\bM)$ is one of left $\cO(\bL)$-comodules under the left adjoint coaction
\begin{equation*}
  x\mapsto x_1 S(x_3)\otimes x_2. 
\end{equation*}
For this reason, \Cref{eq:5_tensorands} is the result of first applying the left adjoint $\cO(\bL)$-coaction to the left hand leg of
\begin{equation*}
  \pi_\bM(f_1)\otimes f_2,
\end{equation*}
and then subjecting the result to the operation
\begin{equation}\label{eq:bullets}
  \bullet\otimes\bullet\otimes \square \mapsto \bullet (\pi_\bL \square_1) \otimes \bullet\otimes \square_2. 
\end{equation}
The conclusion now follows from \Cref{eq:m_inv}, which ensures that the input of \Cref{eq:bullets} is $1\otimes 1\otimes f$.  
\end{proof}

\begin{remark}
  The proof of \Cref{le.aux} is a quantum version of the following classical argument that will be much more transparent: 

Given a continuous function $f$ on $\bG$, the expression for its expectation is 
\begin{equation*}
  E_\bL f = \int_\bL f(l-)\ \mathrm{d}l. 
\end{equation*}
We want to argue that if $f$ is invariant under left translation by $\bM$, then so is $E_\bL f$. In order to see this, let $m\in \bM$. We then have 
\begin{equation*}
  (E_\bL f)(m-) = \int_\bL f(lm-)\ \mathrm{d}l = \int_\bL f(lml^{-1}\cdot l-)\ \mathrm{d}l,
\end{equation*}
which, because $lml^{-1}\in \bM$ and $f$ is $\bM$-invariant, equals 
\begin{equation*}
  \int_\bL f(l-)\ \mathrm{d}l = E_\bL f. 
\end{equation*}
\end{remark}

\begin{remark}
  We note that \Cref{pr.mod_cpct} would be problematic in the
  more general setting of linearly reductive quantum groups (which is
  why we only have \Cref{pr.mod} in the latter case). 

The reason is that even classically, intersections of linearly
reductive subgroups (such as $(\bH\wedge\bM)\vee\bL)$ need not be linearly reductive again, as
\Cref{ex.lin_red_bad} below shows. 
\end{remark}

\begin{example}\label{ex.lin_red_bad}
  Let 
  \begin{equation*}
    \bG = SL_3 = SL_3(\bC).
  \end{equation*}
  Using the correspondence between complex Lie subalgebras of
  $\mathfrak{g} = \mathfrak{sl}_3$ and complex linear algebraic
  subgroups of $\bG$ (\cite[discussion preceding 3.42 and Theorem
  4.22]{milneLAG}) and the fact that this correspondence is compatible
  with intersections (\cite[Proposition 3.19]{milneLAG} or \cite[6.12]{borel}), it
  suffices to exhibit two semisimple Lie subalgebras $\mathfrak{a}$
  and $\mathfrak{b}$ of $\mathfrak{g}$ whose intersection is not
  semisimple. We take the span of
\begin{equation}\label{eq:sl2_triple}
  e=\begin{pmatrix}0&0&0\\0&0&0\\1&0&0\end{pmatrix},\ h=\begin{pmatrix}1&0&0\\0&0&0\\0&0&-1\end{pmatrix},\ f=\begin{pmatrix}0&0&1\\0&0&0\\0&0&0\end{pmatrix}
\end{equation}
for $\mathfrak{a}$. The three displayed elements
are a so-called {\it $\mathfrak{sl}_2$ triple}; this means that their
identification with
\begin{equation*}
  \begin{pmatrix}0&1\\0&0\end{pmatrix},\ \begin{pmatrix}1&0\\0&-1\end{pmatrix},\ \begin{pmatrix}0&0\\1&0\end{pmatrix}  
\end{equation*}
respectively implements an isomorphism $\mathfrak{a}\cong
\mathfrak{sl}_2$. 

Similarly, we take the conjugate of $\mathfrak{a}$ by 
\begin{equation}\label{eq:matr}
  \begin{pmatrix}
    1&1&0\\0&1&1\\0&0&1
  \end{pmatrix}.
\end{equation}
for $\mathfrak{b}$. 

Since \Cref{eq:matr} commutes with the leftmost element $e$ of
\Cref{eq:sl2_triple} and conjugates the semisimple element $h$ outside
of $\mathfrak{a}$ (as can easily be seen), the intersection
$\mathfrak{a}\cap\mathfrak{b}$ is the one-dimensional span of $e$, and
hence at the level of groups the corresponding intersection is a copy
of the (non-linearly reductive) additive algebraic group $\bG_a$ over
$\bC$ (i.e. just $\bC$ with its usual additive group structure). 
\end{example}

There is also a dual version to \Cref{pr.mod_cpct}, dealing with
discrete quantum groups in the sense of \Cref{def.red}.

Quantum subgroups $\bM\le \bG$ then correspond to Hopf subalgebras
\begin{equation*}
  k\bM \subseteq k \bG.
\end{equation*}

\begin{definition}\label{def.norm_bis}
  Let $\bG$ be an algebraic discrete quantum group. A quantum subgroup $\bL$ {\it
  normalizes} another $\bM\le \bG$ if the group algebra $k\bM$ of
the latter is invariant under the adjoint actions of $k\bL$ on
$k\bG$. 
\end{definition}

\begin{proposition}\label{pr.mod_discr}
  Let $\bG$ be an algebraic discrete quantum group, with quantum subgroups $\bL\le \bH\le \bG$ and $\bM\le \bG$ such that $\bL$ normalizes $\bM$. Then, the equality \Cref{eq:mod} holds.   
\end{proposition}
\begin{proof}
The proof is essentially the same as that of \Cref{pr.mod}, once we
substitute $k\bG$ for $\cO(\bG)$ in that result, and similarly
substitute the Hopf subalgebras
\begin{equation*}
  k\bH,\ k\bL,\ k\bM\ \subseteq\  k\bG
\end{equation*}
of $k\bG$ for 
\begin{equation*}
 A_\bL,\ A_\bH,\ A_\bM\ \subseteq\ \cO(\bG)  
\end{equation*}
respectively. 
\end{proof}

\Cref{pr.mod_cpct,pr.mod_discr} will allow us to prove the following version of the butterfly (or Zassenhaus) lemma (\cite[Vol. 1, p. 77]{kur} or \cite[Chapter 2, Lemma 5.10]{rot}) for compact and discrete quantum groups.

\begin{proposition}\label{pr.zas}
  Let $\bA'\trianglelefteq \bA$ and $\bB'\trianglelefteq \bB$ be
  quantum subgroups of either a compact or an algebraic discrete
  quantum group $\bG$. Then, we have an isomorphism
  \begin{equation*}
    \frac{\bA'\vee(\bA\wedge \bB)}{\bA'\vee(\bA\wedge \bB')} \cong \frac{\bB'\vee(\bA\wedge \bB)}{\bB'\vee(\bA'\wedge \bB)}. 
  \end{equation*}
The analogous statement holds for linearly reductive $\bG$ provided
$\bA$, $\bA'$, etc. are all normal in $\bG$. 
\end{proposition}
\begin{proof}
We focus first on the compact / discrete case, following the usual strategy (as in \cite[Vol. 1, p. 77]{kur} or the proof of \cite[Chapter 2, Lemma 5.10]{rot}, for instance) of proving that we have isomorphisms
\begin{equation*}
  \begin{tikzpicture}[auto,baseline=(current  bounding  box.center)]
    \path[anchor=base] (0,0) node (1) {$\frac{\bA'\vee(\bA\wedge \bB)}{\bA'\vee(\bA\wedge \bB')}$} +(8,0) node (2) {$\frac{\bB'\vee(\bA\wedge \bB)}{\bB'\vee(\bA'\wedge\bB)}$} +(4,-.5) node (3) {$\frac{\bA\wedge \bB}{(\bA'\wedge \bB)\vee(\bA\wedge \bB')}$};
         \draw[->] (1) to [bend right=6] node[pos=.5,auto,swap] {$\scriptstyle \cong$} (3);
         \draw[->] (2) to [bend left=6] node[pos=.5,auto] {$\scriptstyle \cong$} (3);
  \end{tikzpicture}
\end{equation*}
By symmetry, it suffices to focus on the left hand side of this
diagram. The required isomorphism will follow from the compact /
discrete quantum version of the second isomorphism theorem (\Cref{th.2nd_bis,th.2nd_bis_discr}) applied to
\begin{equation*}
  \bH= \bA\wedge \bB \text{ and } \bK=\bA'\vee(\bA\wedge \bB')
\end{equation*}
once we prove that we have
\begin{equation*}
  (\bA\wedge \bB)\wedge (\bA'\vee(\bA\wedge\bB')) = (\bA'\wedge \bB)\vee(\bA\wedge \bB').
\end{equation*}
In turn, this follows from \Cref{pr.mod_cpct} or \Cref{pr.mod_discr}
applied to $\bH=\bA\wedge \bB$, $\bL=\bA\wedge \bB'$ and $\bM=\bA'$.

As for the last claim regarding the linearly reductive case, its proof
is virtually identical, using \Cref{pr.mod} instead of \Cref{pr.mod_cpct}. 
\end{proof}

\subsection{The Schreier refinement theorem}
In this subsection we prove an analogue of the Schreier refinement
theorem for compact and discrete quantum groups (see
e.g. \cite[Chapter 2, Theorem 5.11]{rot} for the classical analogue
for ordinary discrete groups). To this aim, we need to define a quantum analogue of the notion of (sub)normal series.

\begin{definition}\label{def.series}
Let $\bG$ be either a compact or (algebraic) discrete quantum group. A finite system 
\begin{equation}\label{eq:series}
  \bG=\bG_{0}\ge\bG_{1}\ge\bG_{2}\ge\bG_{3}\ge\cdots\ge\bG_{k}= 1
\end{equation}
of closed quantum subgroups of $\bG$ is called a {\it subnormal series} of $\bG$ if every subgroup $\bG_{i}$ is a proper normal closed quantum subgroup of $\bG_{i-1}$ , 
$i\in \left\{1,2,\cdots,k\right\}$. In particular, $\bG_{1}$ is a normal closed quantum subgroup of $\bG$, $\bG_{2}$ is a normal closed quantum subgroup of $\bG_{1}$, but not necessarily of $\bG$, and so on.

A subnormal series is {\it normal} if each $\bG_i$ is normal in the
ambient group $\bG$.

The corresponding subquotient quantum groups 
\begin{equation*}
  \bG_1\backslash \bG,\ \bG_{2}\backslash \bG_{1},\ \cdots,\ \bG_{k}\backslash \bG_{k-1} 
\end{equation*}
of $\bG$ are the {\it factors} of the (sub)normal series \Cref{eq:series}.

The integer $k$ is the {\it length} of the series \Cref{eq:series}.
\end{definition}

\begin{definition}\label{def.refine}
A subnormal series 
\begin{equation}\label{eq:series2}
\bG=\bH_{0}\ge \bH_{1}\ge \bH_{2}\ge \bH_{3}\ge \cdots\ge \bH_{l}= 1  
\end{equation}
is called a {\it refinement} of the subnormal series \Cref{eq:series} if every quantum subgroup $\bG_{i}$ of \Cref{eq:series} coincides with one of the quantum subgroups $\bH_{j}$, i.e. if every quantum subgroup that occurs in \Cref{eq:series} also occurs in \Cref{eq:series2}.

In particular, every normal series is a refinement of itself. The lengths of the normal series \Cref{eq:series} and its refinement \Cref{eq:series2} of course satisfy the inequality $k\leq l$.

Two subnormal series of a compact quantum groups are called {\it equivalent} if their lengths are equal and their constituent subquotients are isomorphic up to permutation. 
\end{definition}

We are now ready for the following analogue of Schreier's refinement theorem. As we will see, its proof, given the Zassenhaus lemma (\Cref{pr.zas}) is virtually automatic.

\begin{theorem}\label{th.sch}
Any two subnormal series of a compact / discrete quantum group $\bG$
have equivalent refinements. 

The same holds for any two normal series of a linearly reductive
quantum group. 
\end{theorem}
\begin{proof}
  We focus first on the claim relating to compact and algebraic
  discrete quantum groups.

Let \Cref{eq:series,eq:series2} be two normal series of a compact quantum group $\bG$, and set
\begin{equation*}
  \bG_{ij}=\bG_{i}\vee(\bG_{i-1}\wedge\bH_{j}),\ \bH_{ij}=\bH_{j}\vee(\bH_{j-1}\wedge\bG_{i}).
\end{equation*}
For $i\in \left\{1,2,\cdots,k\right\}$ and $j\in \left\{1,2,\cdots,l\right\}$ we obtain two new refinements of \Cref{eq:series} and \Cref{eq:series2} respectively:
\begin{equation}\label{eq:ref1}
  \bG_{i-1}=\bG_{i0}> \bG_{i(j-1)}> \bG_{ij}> \bG_{ii}=\bG_{i};
\end{equation}
\begin{equation}\label{eq:ref2}
  \bH_{j-1}=\bH_{0j}> \bH_{(i-1)j}> \bH_{ij}> \bG_{kj}=\bH_{j}.  
\end{equation}
By \Cref{pr.zas}, $\bG_{ij}$ is a normal closed quantum subgroup of $\bG_{i(j-1)}$ and  $\bH_{ij}$ is normal in $\bH_{(i-1)j}$, and moreover 
\begin{equation}\label{eq:appl_zas}
  \bG_{ij}\backslash \bG_{i(j-1)}\cong \bH_{ij}\backslash \bH_{(i-1)j}.
\end{equation}
The refinements induced by \Cref{eq:ref1,eq:ref2} have the same length, and \Cref{eq:appl_zas} says that they are equivalent. 

The proof of the second claim follows similarly, using the
corresponding second half of \Cref{pr.zas}. 
\end{proof}
\begin{remark}
  This is simply an adaptation to the quantum setting of the usual
  proof of the Schreier refinement theorem (see e.g. the proof of
  \cite[Chapter 2, Theorem 5.11]{rot}). As mentioned above, once we
  have the Zassenhaus lemma the standard argument goes through
  mechanically. The same goes for the Jordan-H\"older theorem below. 
\end{remark}

\subsection{The Jordan-H\"older theorem}
In this subsection we prove analogues of the Jordan-H\"older theorem
for compact and discrete quantum groups (and a weaker form of it in
the linearly reductive case). We begin with the following definition.

\begin{definition}
A subnormal series \Cref{eq:series} is a {\it composition series} of $\bG$ if $\bG_{i}$ is a proper maximal normal closed quantum subgroup of $\bG_{i-1}$ for $1\le i\le k$.
\end{definition}

\begin{remark}\label{re.comp}
  In other words, a composition series is a subnormal series that cannot
  be refined further. 
\end{remark}

The main result of this subsection is

\begin{theorem}\label{th.jh}
Any two composition series of a compact or discrete quantum group $\bG$ are equivalent. 
\end{theorem}
\begin{proof}
This is immediate from \Cref{th.sch} together with the observation
(made in \Cref{re.comp}) that composition series cannot be refined
strictly: two composition series have equivalent refinements, and
hence they must already be equivalent. 
\end{proof}

Analogously, making use of the second half of \Cref{th.sch}, we have

\begin{proposition}\label{pr.jh}
  Any two normal series of a linearly reductive quantum group which
  are maximal with respect to refinement are equivalent. 
\qedhere
\end{proposition}

\section{Isomorphism theorems, modular law: the locally compact case}\label{se.lc}

\subsection{The second isomorphism theorem}

We shall first consider the setting of the second isomorphism theorem for
ordinary discrete groups, transported to the present framework: $\bG$
is a locally compact quantum group, $\bH\in\mathcal{QS}(\GG)$ and $\bN\in\mathcal{NQS}(\GG)$. 

In order to make sense of the statement of \Cref{th.2nd} below, note
first that according to \Cref{le.univ_quot}, $\bH\to \bG \to \bG/\bN$
factors through a morphism 
\begin{equation*}
  \bH/\bH\wedge \bN\to \bG/\bN
\end{equation*}
inducing an action of $\bH/\bH\wedge\bN$ on $\bG/\bN$.

\begin{proposition}\label{th.2nd}
Let $  \GG$ be a locally compact quantum group, $\bH\in\mathcal{QS}(\GG)$ and $\bN\in\mathcal{NQS}(\GG)$. Let us denote by $\Pi:\bH \to \bG/\bN$ the induced morphism.
Then $ \bH\wedge \bN =  \ker\Pi$. 
If  moreover $\GG = \HH\vee\bN$
then $\overline{\textrm{im}\Pi} = \bG/\bN$. 
\end{proposition}

\begin{proof}

Let us consider  homomorphisms $\Pi_1:\GG\to\GG/\bN$ and $\Pi:\HH\to\GG\to\GG/\bN$. The right quantum group homomorphism assigned to $\hh\Pi_1$  and $\hh\Pi$  will be denoted by  $\hh\alpha_1:\Linf(\hh\GG)\to\Linf(\hh\GG)\vtens\Linf(\hh{\GG/\bN})$   and   $\hh\alpha:\Linf(\hh\HH)\to\Linf(\hh\HH)\vtens\Linf(\hh{\GG/\bN})$ respectively.  Viewing $\Linf(\hh\HH)$ as a subalgebra of $\Linf(\hh\GG)$ we have
$\hh\alpha:\Linf(\hh\HH)\to\Linf(\hh\HH)\vtens\Linf(\hh{\GG/\bN})$ where $\hh\alpha = \hh\alpha_1|_{\Linf(\hh\HH)}$ (see \Cref{restriclemma}). Let  $x\in\cd(\Linf(\HH/\ker\Pi ))$. Then $x\in\Linf(\hh\HH)$ and $\hh\alpha_1(x) = x\otimes \I$. By \Cref{codual_alpha} we have  $x\in\Linf(\hh\bN)\wedge\Linf(\hh\HH)$. Thus $\Linf(\HH/\ker\Pi ) = \Linf(\HH/\HH\wedge\bN)$, i.e.
$\ker\Pi = \HH\wedge\bN$. 

Since    $\GG = \HH\vee\bN$ and $\bN\in\mathcal{NQS}(\GG)$  we conclude from  \Cref{norminter} that
\begin{equation}\label{gen_quot}\Linf(\hh\GG) =  \{\Linf(\hh{\bN})\Linf(\hh\HH)\}^{\sigma - \textrm{cls}}.\end{equation}
Using \Cref{normal_subgroup1} we can see that  for all $\omega\in\Linf(\hh\GG)_*$, $x\in\Linf(\hh\HH)$ and $y\in\Linf(\hh\bN)$ we have 
\begin{equation}\label{eq1}(\omega\otimes\id)(\hh\alpha_1(xy)) = (y\cdot\omega\otimes\id)(\hh\alpha(x)).\end{equation}
 Using \Cref{basic_eq}, \Cref{gen_quot}, \Cref{eq1} we get
\[\Linf(\hh{\GG/\bN}) = \{(\omega\otimes\id)(\hh\alpha (a)):\omega\in\Linf(\hh\HH)_*, a\in\Linf(\hh\HH)\}^{\cls }\]
 i.e. the closure of the image of $\Pi $ is also $\GG/\bN$. 
\end{proof}

Using \Cref{thm:first_iso_thm} and \Cref{th.2nd} we get 
\begin{corollary}\label{th.2nd.corr}
  The homomorphism $ \bH/\bH\wedge \bN \to \GG/\bN$ is an isomorphism
  if and only if the corresponding action of $\bH/\bH\wedge \bN$ on
  $\Linf(\GG/\bN)$ is integrable.  
\qedhere
\end{corollary}

\begin{remark}\label{re.cases}
The assumptions of \Cref{th.2nd.corr} hold trivially when $\HH$ is compact. 

Assume on the other hand that $\bN$ is compact. Then, since the action
of $\HH$ on $\Linf(\GG)$ is integrable and we have a conditional
expectation onto $\Linf(\GG/\bN)$, the integrability assumption also
holds in this case.
 
Finally, note that every morphism $\Pi$ of discrete quantum groups   automatically identifies $\HH/\textrm{ker}\Pi$ with $\overline{\textrm{im}\Pi}$. In conclusion,
\Cref{th.2nd.corr} also goes through when all quantum groups in sight are
discrete. 
\end{remark}

Let us also note that equivalent statements ib \Cref{th.2nd.corr} fails (and hence so does the second
isomorphism theorem) in general even classically, for locally compact
abelian groups, as the following example shows.

\begin{example}\label{ex.pathol} Consider the group $\bG=\bT^2\times \bR$, and the subgroups
\begin{equation*}
  \bH=\{(e^{it\theta},1,t)\ |\ t\in \bR\}\text{ and } \bN=\{(1,e^{is\phi},s)\ |\ s\in \bR\}
\end{equation*}
for real numbers $\theta$ and $\phi$ that are incommensurable
(i.e. linearly independent over $\bQ$). Then, the subgroup 
\begin{equation*}
  \{(e^{it\theta},e^{-it\phi},0)\}\subset  \bT^2\times\{0\}
\end{equation*}
of $\bH\bN$ is dense $\bT^2\times\{0\}$ and hence the closure $\bH\vee\bN$ of $\bH\bN$
contains $\bT^2\times\{0\}$. But the product of this latter group with
$\bH$ is clearly all of $\bG$, and we have $\bH\vee\bN=\bG$. 

Now, $\bH/\bH\wedge\bN$ is a one-dimensional Lie group whereas
$\bG/\bN$ is a two-dimensional one, and hence the conditions of
\Cref{th.2nd.corr} cannot possibly hold.   
\end{example}

The fundamental characteristic of \Cref{ex.pathol} is that the naive
product $\bH\bN$ is not closed in $\bG$, and hence $\bH\vee \bN$ is
``larger than expected''. Indeed, classically, it is this failure of
$\bH\bN$ to be closed that prevents the conditions of
\Cref{th.2nd.corr} from holding. This is summarized in the following
result.

\begin{proposition}\label{pr.2nd-class}
  Let $\bG$ be a classical locally compact group, and $\bH\le \bG$ and
  $\bN\trianglelefteq \bG$ closed subgroups.  

  Then, $\bH/\bH\wedge \bN$ acts integrably on $\bG/\bN$ if and only
  if for every $(\bH\wedge \bN)$-invariant closed subset $\bF$ of
  $\bH$ the product $\bF\bN$ is closed.
\end{proposition}
\begin{proof}
  According to \Cref{th.2nd.corr}, the integrability of the
  action in the statement is equivalent to the canonical map
  \begin{equation}\label{eq:can_2nd}
    \bH/(\bH\wedge \bN) \to (\bH\vee\bN)/\bN
  \end{equation}
being an isomorphism. We will use this equivalence throughout the
proof, implicitly or explicitly.

\vspace{.5cm}

($\Leftarrow$) Suppose $\bF\bN$ is closed in $\bG$ for every closed
$\bF\subseteq \bH$. Applying this to $\bF=\bH$ first, we have
$\bH\vee\bN=\bH\bN$ and hence the canonical one-to-one morphism  
\Cref{eq:can_2nd} is also onto. 

Now note that the composition
\begin{equation*}
  \bH\to \bG\to \bG/\bN
\end{equation*}
realizes $\bH/(\bH\wedge \bN)$ as a closed subgroup of the right hand
side. Indeed, it induces an embedding of the former group into
$\bG/\bN$, and the condition on $\bF\bN$ being closed means precisely
that the induced one-to-one map
\begin{equation*}
  \bH/(\bH\wedge \bN) \to \bG/\bN
\end{equation*}
is closed. 

All in all, \Cref{eq:can_2nd} is a bijective inclusion of one closed
subgroup of $\bG/\bN$, namely $\bH/(\bH\wedge \bN)$, into another,
i.e. $\bH\bN/\bN$. It is then an isomorphism, and the conclusion
follows from \Cref{th.2nd.corr}.

\vspace{.5cm}

($\Rightarrow$) Conversely, suppose the action in question is
integrable, and hence by \Cref{th.2nd.corr} the morphism
\Cref{eq:can_2nd} is bijective. The diagram 
\begin{equation*}
  \begin{tikzpicture}[auto,baseline=(current  bounding  box.center)]
    \path[anchor=base] (0,0) node (h) {$\bH$} +(2,.5) node (hk)
    {$\bH\vee\bN$} +(5,0) node (hk/k) {$(\bH\vee\bN)/\bN$} +(2,-.5)
    node (h/) {$\bH/(\bH\wedge\bN)$};
         \draw[->] (h) to [bend left=6] (hk);
         \draw[->] (hk) to [bend left=6] (hk/k);
         \draw[->] (h) to [bend right=6] (h/);
         \draw[->] (h/) to [bend right=6] node[pos=.5,auto,swap]
         {$\scriptstyle \cong$} (hk/k);
  \end{tikzpicture}
\end{equation*}
shows that $\bH\vee\bN$ is generated as a (plain, not topological)
group by $\bH$ and $\bN$. Since $\bN$ is normal, this in turn implies
$\bH\vee\bN = \bH\bN$, so that the latter product must be closed. 

Moreover, the fact that \Cref{eq:can_2nd} is a homeomorphism implies
that it is in particular closed. This means that the image of every
closed subset $\bF\subseteq \bH$ as in the statement is closed in
$\bG/\bN$, and hence its preimage $\bF\bN$ through the quotient map
$\bG\to \bG/\bN$ is closed. 
\end{proof}

Although quite explicit, the closure condition in \Cref{pr.2nd-class}
might be somewhat inconvenient to check. In view of this, one might
wonder whether the seemingly weaker condition that $\bH\bN$ be closed
in $\bG$ is sufficient. \Cref{ex.pathol-class} shows that this is not
the case, even in the case of classical {\it abelian} locally compact
groups.

Before spelling out the example, let us clarify what it is meant to
do. Placing ourselves entirely within the context of locally compact
abelian groups, consider for simplicity the case when $\bH\wedge \bN$
is trivial. Moreover, we may further assume harmlessly that the subgroup
$\bH\bN\le \bG$ (which is supposed to be closed anyway) is all of
$\bG$. 

All in all, we will have an algebraic decomposition
\begin{equation}
  \label{eq:dec}
  \bG=\bH\oplus \bN.
\end{equation}
Then, the condition from \Cref{pr.2nd-class} and its symmetric
counterpart (i.e. with the roles of $\bH$ and $\bN$ interchanged)
jointly mean precisely that the decomposition \Cref{eq:dec} is one of
{\it topological} abelian groups as well as abstract ones.  

In conclusion, in order to show that the closedness of $\bH\bN$ does not
entail the second isomorphism theorem, it suffices to exhibit a
locally compact abelian group $\bG$ which decomposes as \Cref{eq:dec}
abstractly for closed subgroups $\bH$ and $\bN$, but not
topologically. \Cref{ex.pathol-class} achieves this by choosing $\bH$
and $\bN$ to be discrete, whereas $\bG$ is not.

\begin{example}\label{ex.pathol-class}
  We take $\bG$ to be the direct product between a copy of the compact
  additive group $\bZ_p$ of $p$-adic integers for some odd prime
  number $p$, and a discrete copy $\Gamma$ of the self-same group
  $\bZ_p$ (in other words, $\Gamma$ is $\bZ_p$ as an abstract group,
  but with discrete topology). 

Now, in $\bG=\bZ_p\times\Gamma$ we have a diagonal subgroup
\begin{equation*}
  \bH = \{(g,g)\ |\ g\in \bZ_p\}
\end{equation*}
as well as an anti-diagonal one, 
\begin{equation*}
  \bN=\{(g,-g)\ |\ g\in \bZ_p\}.
\end{equation*}
We have $\bH\wedge\bN=\{0\}$ because $\bZ_p$ is torsion-free, and also
$\bH+\bN=\bG$ because $\bZ_p$ is divisible by $2$. Moreover, $\bH$ and
$\bN$ are easily seen to both be closed in $\bG$ and discrete. By
construction, though, $\bG$ is not. The preceding discussion explains
why this will do. 
\end{example}

\subsection{The third isomorphism theorem}

Recall (\cite[$\S$3.3, Theorem 19]{DF}) that this states that given normal subgroups $N$ and $H$ of $G$ with $N\le H$, we have $H/N\trianglelefteq G/N$ and moreover 
\begin{equation*}
  (G/N)/(H/N)\cong G/H.
\end{equation*}

Consider now the typical setup for the third isomorphism theorem: a locally compact quantum group $\bG$, and normal closed quantum subgroups $\bN\le \bH$ of $\bG$. Then, because the composition
\begin{equation*}
  \bN\to \bH\to \bG\to \bG/\bN
\end{equation*}
is trivial, \Cref{le.univ_quot} ensures that we have a factorization
\begin{equation*}
  \begin{tikzpicture}[auto,baseline=(current  bounding  box.center)]
    \path[anchor=base] (0,0) node (h) {$\bH$} +(4,0) node (g/n) {$\bG/\bN$} +(2,.5) node (g) {$\bG$} +(2,-.5) node (h/n) {$\bH/\bN$};
         \draw[->] (h) to [bend left=6] (g);
         \draw[->] (g) to [bend left=6] (g/n);
         \draw[->] (h) to [bend right=6] (h/n);
         \draw[->] (h/n) to [bend right=6] (g/n);
  \end{tikzpicture}
\end{equation*}
of the top composition $\bH\to \bG\to \bG/\bN$. We will now examine bottom right morphism $\bH/\bN\to \bG/\bN$. 

In general, we say that a morphism $\Pi:\bP\to \bQ$ of locally compact quantum groups {\it has trivial kernel} if the quotient quantum group 
\begin{equation*}
  \bP\to \bP/\mathrm{ker} \Pi
\end{equation*}
of \cite[Definition 4.4]{KKS} is an isomorphism. Let us recal that $\Pi:\bP\to\bQ$ induces a morphism $\Pi_1:\bP/\ker\Pi\to\bQ$ which has trivial kernel.
\begin{lemma}\label{le.triv_ker}
  The canonical morphism $\bH/\bN\to \bG/\bN$ has trivial kernel. 
\end{lemma}
\begin{proof}
 Let us consider the morphism 
  \begin{equation*}
    \eta:\bH\to \bG\to \bG/\bN.
  \end{equation*} Using \Cref{th.2nd} we see that  $\ker\eta = \HH\wedge\bN = \bN$. In particular the kernel of the induced morphism $\HH/\ker\eta \to\GG/\bN$ is trivial. 
\end{proof}

\begin{lemma}\label{le.full_im}
  The closed image of the canonical map $\bG/\bN\to \bG/\bH$ is full. 
\end{lemma}
\begin{proof}
  As noted in \Cref{normal_subgroup} the closed image of $\Pi:\GG\to\GG/\HH$ is full. Since $\bN\subset\bH = \ker\Pi$ the induced morphism  $\GG/\bN\to\GG/\HH$ exists and  its closed image coincides with the one of $\Pi:\GG\to\GG/\HH $ thus it  is also full.  
\end{proof}

Let us gather up all of the ingredients we have so far in the form of \Cref{le.triv_ker,le.full_im} into a weak version of the third isomorphism theorem (to be improved on later):

\begin{proposition}\label{pr.3rd_weak}
  Given   normal closed quantum subgroups $\bN\le \bH$ of a locally compact quantum group $\bG$, the canonical morphism $\Pi_1:\bH/\bN\to \bG/\bN$ has trivial kernel and its closed image is precisely the kernel of $\Pi_2:\bG/\bN\to \bG/\bH$. 
 
\end{proposition}
\begin{proof}
 Let $\Pi:\hh\GG\to\hh\HH$ be the morphism which is dual to the embedding $\HH\le\GG$. Let us recall  that  $\Linf(\GG/\bN)$ being a Baaj-Vaes subalgebra of $\Linf(\GG)$ can be interpreted as   $\hh{\GG/\bN}\le\hh\GG$. Using \Cref{restriclemma} to $\Pi:\hh\GG\to\hh\HH$ and $\hh{\GG/\bN}\le\hh\GG$ we conclude that the right quantum group homomorphism $\alpha_{\HH\to\GG/\bN}: \Linf(\GG/\bN)\to \Linf(\GG/\bN)\vtens\Linf(\HH)$ is the restriction of the right quantum group homomorphism $\alpha_{\HH\to\GG}:\Linf(\GG)\to \Linf(\GG)\vtens\Linf(\HH)$ to $\Linf(\GG/\bN)\subset \Linf(\GG)$. Using \Cref{le.triv_ker} we conclude that the kernel of the morphism $\HH\to\GG/\bN$ is equal $\bN$ and using \Cref{q_def} we get $\alpha_{\HH\to\GG/\bN}(\Linf(\GG/\bN))\subset\Linf(\GG/\bN)\vtens\Linf(\HH/\bN)$. Summarizing the restriction of $\alpha_{\HH\to\GG}:\Linf(\GG)\to\Linf(\GG)\vtens\Linf(\HH)$ to $\Linf(\GG/\bN)$ induces   right quantum group homomorphism
 \[\alpha_{\HH/\bN\to\GG/\bN}:\Linf(\GG/\bN)\to\Linf(\GG/\bN)\vtens\Linf(\HH/\bN). \]  In particular 
 \begin{equation}\label{quotient_N}\Linf(\GG/\HH) = \{x\in\Linf(\GG/\bN): \alpha_{\HH/\bN\to\GG/\bN}(x) = x\otimes\I\}.\end{equation}
 Recalling that $\HH/\bN\to\GG/\bN$ is denoted by  $\Pi_1$ let us consider  $\overline{\textrm{im}\Pi_1}\leq\GG/\bN$.
 Equation \Cref{quotient_N} then shows that  \begin{equation}\label{quotient_N1}\Linf((\GG/\bN)/\overline{\textrm{im}\Pi_1})=\cd(\Linf(\hh{\overline{\textrm{im}\Pi_1}})) = \Linf(\GG/\HH).\end{equation} On the other hand noting that $\Pi_2:\bG/\bN\to \bG/\bH$  is represented by the bicharacter $\ww^{\GG/\HH}\in\Linf(\hh{\GG/\HH})\vtens\Linf(\GG/\bN)$ (where we used that $\Linf(\GG/\HH)\subset \Linf(\GG/\bN)$) we get  \[\Linf((\GG/\bN)/\ker\Pi_2) = \{(\omega\otimes\id)(\ww^{\GG/\HH}):\omega\in\Linf(\hh{\GG/\HH})_*\}^{\cls  }= \Linf(\GG/\HH)\] which together with \Cref{quotient_N1} shows that $\ker\Pi_2 = \overline{\textrm{im}\Pi_1}$. 
 
\end{proof}

In order to have a full analogue of \cite[$\S$3.3, Theorem 19]{DF}, we would further want to know that the canonical morphism $\bH/\bN\to \bG/\bN$ identifies the former with a closed quantum subgroup of the latter. Moreover, in view of \Cref{pr.3rd_weak} and \Cref{thm:first_iso_thm}, this amounts to showing that the action of $\bH/\bN$ on $\bG/\bN$ is integrable.

To this end, we will first need the following Weyl-integral-formula-type result.

\begin{proposition}\label{pr.weyl}
Given a normal closed quantum subgroup $\bN\trianglelefteq \bG$ a left-invariant Haar weight $\varphi_\bG$ can be expressed as 
\begin{equation*}
  \varphi_{\bG/\bN}\circ T,
\end{equation*}
where 
\begin{equation}\label{eq:op_wt}
    \begin{tikzpicture}[auto,baseline=(current  bounding  box.center)]
    \path[anchor=base] (0,0) node (g) {$\Linf(\bG)$} +(4,.5) node (gn) {$\Linf(\bG)\otimes \Linf(\bN)$} +(8,0) node (g/n) {$\Linf(\bG/\bN)$};
         \draw[->] (g) to [bend left=6] (gn);
         \draw[->] (gn) to [bend left=6] node[pos=.5,auto] {$\scriptstyle \id\otimes\varphi_\bN$} (g/n);
         \draw[->] (g) to [bend right=6] node[pos=.5,auto,swap] {$\scriptstyle T$} (g/n);
  \end{tikzpicture}
\end{equation}
is a faithful semifinite normal operator-valued weight. 
\end{proposition}
\begin{proof}
  The fact that the composition \Cref{eq:op_wt} is a faithful normal operator-valued weight (in the sense of \cite[Definition IX.4.12]{tak2}) into the right hand side (one needs to check that it lands in the algebra of $\bN$-invariants of $\Linf(\bG)$) is essentially \cite[Proposition 1.3]{Vae01}. 

The integrability \cite[Theorem 6.2]{KKS} of the action of the closed subgroup $\bN$ on $\bG$ means by definition that $T$ is semifinite, and hence pre-composing with $T$ turns semifinite weights on $\Linf(\bG/\bN)$ into semifinite weights on $\Linf(\bG)$ (see also e.g. \cite[Definition 8.1]{Kus02}). Finally, the requisite invariance property of $\varphi_{\bG/\bN}\circ T$ is a routine computation, using the invariance properties of $\varphi_{\bG/\bN}$ and $T$.
\end{proof}

Given a morphism $\Pi:\bP\to \bQ$ of locally compact quantum groups, we will denote by $T_{\bP\to \bQ}$ the operator-valued weight
\begin{equation*}
    \begin{tikzpicture}[auto,baseline=(current  bounding  box.center)]
    \path[anchor=base] (0,0) node (q) {$\Linf(\bQ)$} +(4,0) node (qp) {$\Linf(\bQ)\vtens \Linf(\bP)$} +(8,0) node (q_2) {$\Linf(\bQ)$.};
         \draw[->] (q) -- (qp);
         \draw[->] (qp) to node[pos=.5,auto] {$\scriptstyle \id\otimes\varphi_{\bP}$} (q_2);
  \end{tikzpicture}
\end{equation*}
Let us note that in general $T_{\bP\to \bQ}$ is not semifinite.

Finally, \Cref{pr.weyl} will help in proving the missing integrability ingredient we remarked on above:

\begin{proposition}\label{pr.int}
  Given closed normal subgroups $\bN\le \bH$ of a locally compact quantum group $\bG$, the canonical action of $\bH/\bN$ on $\bG/\bN$ is integrable. 
\end{proposition}
\begin{proof}
 We have to show that the operator-valued weight $T_{\bH/\bN\to \bG/\bN}$ defined as
\begin{equation*}
    \begin{tikzpicture}[auto,baseline=(current  bounding  box.center)]
    \path[anchor=base] (0,0) node (g/n) {$\Linf(\bG/\bN)$} +(4,0) node (g/nh/n) {$\Linf(\bG/\bN)\vtens \Linf(\bH/\bN)$} +(8,0) node (g/n_2) {$\Linf(\bG/\bN)$};
         \draw[->] (g/n) -- (g/nh/n);
         \draw[->] (g/nh/n) to node[pos=.5,auto] {$\scriptstyle \id\vtens\varphi_{\bH/\bN}$} (g/n_2);
  \end{tikzpicture}
\end{equation*}
is semifinite, or equivalently, that there is at least one element of $\Linf(\bG/\bN)$ that is integrable with respect to the $(\bH/\bN)$-action (\cite[Proposition 6.2]{Kus02}). 

We have already observed via \cite[Theorem 6.2]{KKS} that actions of closed quantum subgroups are integrable, and hence $T_{\bN\to \bG}$ is semifinite. Similarly, $T_{\bH\to \bG}$ is semifinite. We will argue that for any $x\in \Linf(\bG)^+$ that is $\bH$-integrable, its image
\begin{equation*}
  T_{\bN\to \bG}(x)\in \Linf(\bG/\bN)
\end{equation*}
 is $(\bH/\bN)$-integrable; as observed, this is sufficient to finish the proof of the proposition.

First, consider the following diagram of operator-valued weights and von Neumann algebra homomorphisms, where commutativity is immediate from the definitions:
\begin{equation}\label{eq:big}
      \begin{tikzpicture}[auto,baseline=(current  bounding  box.center)]
    \path[anchor=base] (0,0) node (1) {$\Linf(\bG)$} +(0,1) node (2) {$\Linf(\bG)\vtens \Linf(\bH)$} +(5,2) node (3) {$\Linf(\bG)\vtens \Linf(\bH)\vtens \Linf(\bN)$} +(10,1) node (4) {$\Linf(\bG)\vtens \Linf(\bH/\bN)$} +(0,-1) node (5) {$\Linf(\bG)\vtens \Linf(\bN)$} +(5,-2) node (6) {$\Linf(\bG/\bN)$} +(10,-1) node (7) {$\Linf(\bG/\bN)\vtens \Linf(\bH/\bN)$};
         \draw[->] (1) -- (2);
         \draw[->] (1) -- (5);
         \draw[->] (2) to[bend left=6] (3);
         \draw[->] (3) to[bend left=6] node[pos=.5,auto] {$\scriptstyle \id\otimes\varphi_\bN$} (4);
         \draw[->] (5) to[bend right=6] node[pos=.5,auto] {$\scriptstyle \id\otimes\varphi_\bN$} (6);
         \draw[->] (6) to[bend right=6] (7);
         \draw[right hook->] (7) -- (4);
  \end{tikzpicture}
\end{equation}
Now further glue the commutative square 
\begin{equation}\label{eq:small}
      \begin{tikzpicture}[auto,baseline=(current  bounding  box.center)]
    \path[anchor=base] (0,0) node (1) {$\Linf(\bG/\bN)\otimes \Linf(\bH/\bN)$} +(4,0) node (2) {$\Linf(\bG/N)$} +(4,2) node (3) {$\Linf(\bG)$} +(0,2) node (4) {$\Linf(\bG)\otimes \Linf(\bH/\bN)$};
         \draw[->] (1) to node[pos=.5,auto] {$\scriptstyle \id\otimes \varphi_{\bH/\bN}$} (2);
         \draw[right hook->] (1) -- (4);
         \draw[right hook->] (2) -- (3);
         \draw[->] (4) to node[pos=.5,auto] {$\scriptstyle \id\otimes \varphi_{\bH/\bN}$} (3);
  \end{tikzpicture}
\end{equation}
to the right hand side of \Cref{eq:big}. 

Using the Weyl integration formula (\Cref{pr.weyl}) for the normal subgroup $\bN\trianglelefteq \bH$, we can see that the composition of the top half of \Cref{eq:big} with the top horizontal arrow of \Cref{eq:small} yields precisely the semifinite operator-valued weight $T_{\bH\to \bG}$. The commutativity of the compound diagram obtained by gluing \Cref{eq:big,eq:small} then proves our assertion that the image through $T_{\bN\to \bG}$ of an $\bH$-integrable element of $\Linf(\bG)$ is $(\bH/\bN)$-integrable, thus completing the proof. 
\end{proof}

In summary, we obtain

\begin{theorem}\label{th.3rd}
  Let $\bN\le \bH\trianglelefteq \bG$ be inclusions of closed locally compact quantum subgroups, and assume furthermore that $\bN$ is normal in $\bG$. Then, we have 
  \begin{equation*}
    \bH/\bN\trianglelefteq \bG/\bN\quad \text{and} \quad (\bG/\bN)/(\bH/\bN)\cong \bG/\bH. 
  \end{equation*}
\end{theorem}
\begin{proof}
As noted above, \Cref{thm:first_iso_thm} and \Cref{pr.3rd_weak} reduce the problem to showing that the action of $\bH/\bN$ on $\bG/\bN$ resulting from the canonical morphism $\bH/\bN\to \bG/\bN$ is integrable. This is exactly what \Cref{pr.int} says. 
\end{proof}

In fact, some of the above results generalize somewhat so as to allow us to recover standard results on topological groups in the locally compact quantum setting. For instance, the conclusion that 
\begin{equation*}
  \bH/\bN\to \bG/\bN
\end{equation*}
is a closed embedding does not actually require the normality of $\bH$, and hence \Cref{th.3rd} extends to this general setup.

\begin{theorem}\label{th.3rd_gen}
  Let $\bN\le \bH\le \bG$ be closed embeddings of locally compact quantum groups, with $\bN$ normal in $\bG$. Then the canonical morphism 
  \begin{equation*}
    \bH/\bN\to \bG/\bN
  \end{equation*}
is a closed embedding, and 
\begin{equation*}
 \Linf ((\bG/\bN)/(\bH/\bN)) =\Linf( \bG/\bH)
\end{equation*}
 
\qedhere
\end{theorem}

Let us now briefly go back to the setup of \Cref{th.2nd}: $\bH$ and $\bN$ are closed quantum subgroups of $\bG$, with $\bN$ normal. Then, by \Cref{le.univ_quot}, the composition $\bH\to \bG\to \bG/\bN$ always factors as 
\begin{equation*}
  \begin{tikzpicture}[auto,baseline=(current  bounding  box.center)]
    \path[anchor=base] (0,0) node (h) {$\bH$} +(4,0) node (g/k) {$\bG/\bN$} +(2,.5) node (g) {$\bG$} +(2,-.5) node (h/hk) {$\bH/\bH\wedge \bN$};
         \draw[->] (h) to [bend left=6] (g);
         \draw[->] (g) to [bend left=6] (g/k);
         \draw[->] (h) to [bend right=6] (h/hk);
         \draw[->] (h/hk) to [bend right=6] (g/k);
  \end{tikzpicture}
\end{equation*}
Moreover, \Cref{th.3rd_gen} ensures that we can regard $\bH\vee\bN/\bN$ as a closed subgroup of $\bG/\bN$, and an examination of the proof of \Cref{th.2nd} shows that we actually have the following amplification of \Cref{th.2nd}.

\begin{theorem}\label{th.2nd_soup}
Let $\bH\le \bG$ and $\bN\trianglelefteq \bG$ be closed quantum subgroups of a locally compact quantum group. Then, the canonical morphism 
\begin{equation*}
  \bH/\bH\wedge \bN\to \bG/\bN
\end{equation*}
has trivial kernel and closed image $\bH\vee\bN/\bN\le \bG/\bN$. 
\qedhere 
\end{theorem}

\subsection{The modular law and Zassenhauss lemma}

We now proceed to address an analogue of the Zassenhauss lemma for locally compact quantum groups. First, recall the classical (non-topological) statement, for instance as in \cite[Vol. 1, p. 77]{kur}.
\begin{proposition}\label{zass-classic}
  Let $A'\trianglelefteq A$ and $B'\trianglelefteq B$ be subgroups of a group $G$. Then, we have a canonical isomorphism 
  \begin{equation*}
    \frac{A'(A\cap B)}{A'(A\cap B')}\cong \frac{B'(A\cap B)}{B'(A'\cap B)}.
  \end{equation*}
\end{proposition}
\begin{remark}
  The statement implicitly includes the claims that the group products appearing in the formula (such as $A'(A\cap B)$) are indeed subgroups of $G$, and the groups appearing as denominators are normal in those appearing as numerators.  
\end{remark}
Recall (\cite[Vol. 1, p. 77]{kur}) that the proof typically proceeds through the second isomorphism theorem for groups (which \Cref{th.2nd} replicates) by using it to implement connecting isomorphisms  
\begin{equation*}
  \begin{tikzpicture}[auto,baseline=(current  bounding  box.center)]
    \path[anchor=base] (0,0) node (1) {$\frac{A'(A\cap B)}{A'(A\cap B')}$} +(8,0) node (2) {$\frac{B'(A\cap B)}{B'(A'\cap B)}$} +(4,-.5) node (3) {$\frac{A\cap B}{(A'\cap B)(A\cap B')}$};
         \draw[->] (1) to [bend right=6] node[pos=.5,auto,swap] {$\scriptstyle \cong$} (3);
         \draw[->] (2) to [bend left=6] node[pos=.5,auto] {$\scriptstyle \cong$} (3);
  \end{tikzpicture}
\end{equation*}
We will adopt a similar approach here, but we need some preparatory remarks. 

First, note that it is implicit in the proof sketch we have just recalled that under the assumptions of the Zassenhaus lemma we have e.g. 
\begin{equation*}
  (A\cap B)\cap (B'(A'\cap B)) = (A\cap B')(A'\cap B).
\end{equation*}
Given that $A'\cap B$ is a normal subgroup of $A\cap B$ and normalizes $B'$, this follows from the modularity law for  subgroups which we will use in the following form:
\begin{equation*}
  L\le H\le G,\quad M\le G\quad\text{and}\quad L\text{ normalizes }M\quad\Rightarrow\quad H\cap ML=(H\cap M)L.
\end{equation*}

\Cref{th.mod_lc} is an analogue of modularity in the locally compact quantum setting. Let us first prove an easy inclusion.

\begin{lemma}\label{lemma:cont_modul}
 Let $\GG$ be a locally compact quantum groups  $\bM,\bH\leq\GG$, and $\bL\le \HH$. Then  \[(\bH\wedge\bM)\vee\bL\le\bH\wedge (\bM\vee\bL).\] 
\end{lemma}
\begin{proof}
  Let us note that $\bH\wedge\bM\le\bH\wedge (\bM\vee\bL)$. Moreover, by assumption   $\bL\le\bH$  and clearly $ \bL\leq \bM\vee\bL$ thus $\bL\leq\bH\wedge (\bM\vee\bL)$. This altogether shows that $(\bH\wedge\bM)\vee\bL\le\bH\wedge (\bM\vee\bL)$. 
\end{proof}

We now turn to sufficient conditions for an inclusion reversal in
\Cref{lemma:cont_modul}. The material surrounding \Cref{def.well}
above will be needed here.

\begin{theorem}\label{th.mod_lc}
  Let $\bL\le \bH$ and $\bM$ be closed quantum subgroups of a locally compact quantum group $\GG$ 
  such that $\bL$ normalizes $\bM$. Then, we have
  \begin{equation}\label{eq:mod_lc}
    \bH\wedge (\bM\vee\bL)=(\bH\wedge\bM)\vee\bL.
  \end{equation}
if either
\begin{enumerate}
\renewcommand{\labelenumi}{(\alph{enumi})}
  \item $\bL$ is compact, or 
  \item $\bH$ is open in $\bG$. 
\end{enumerate}
\end{theorem}
\begin{proof}
We address the two versions of the result separately. 

\vspace{.5cm}

{\bf Proof of part (a).} Here we rephrase the desired conclusion in
terms of the quantum homogeneous spaces $A_\bullet$ for $\bullet=\bH$,
$\bM$, etc (see notation in \Cref{subse.well}). Since $\cd$ is an anti-isomorphism of lattices the
sought-after conclusion is
\begin{equation}\label{eq:mod_lc_bis}
  A_\bH \vee (A_\bM\wedge A_\bL) = (A_\bH\vee A_\bM)\wedge A_\bL. 
\end{equation}
Using \Cref{lemma:cont_modul} we see that  the right hand side contains the left hand side. We hence focus on proving the opposite inclusion. 

Let us first consider the case when  $\bH$ and $\bM$ are relatively well positioned in the sense of \Cref{def.well}. 
Now, as in the proof of \Cref{pr.int}, consider the operator-valued weights 
\begin{equation*}
  T_{\bullet\to \bG}:\Linf(\bG)\to A_\bullet. 
\end{equation*}
Since $\bL$ is assumed to be compact, $T=T_{\bL\to \bG}$ actually restricts to the identity on $A_\bL$, and hence also on the right hand side of \Cref{eq:mod_lc_bis}. 

On the other hand, in order to study the result of applying $T$ to the
right hand side algebra of \Cref{eq:mod_lc_bis}, it suffices by
\Cref{eq:well} to look at products
\begin{equation*}
  x=x_\bH x_\bM,\ x_\bH\in A_\bH,\ x_\bM\in A_\bM
\end{equation*}
When applied to the latter, due to the preservation by $T$ of bimodule
structures over $A_\bH\subseteq A_\bL$, $T$ produces the element 
\begin{equation*}
  x_\bH T(x_\bM).
\end{equation*}
We would be finished if we could show that $T(x_\bM)\in A_\bM\wedge A_\bL$; this is what \Cref{le.aux_cpct_lc} below does. 

In order to drop the well positioning assumption   let us consider $\bD = \HH\wedge (\bL\vee\bM)$. The containment $\Linf(\hh\HH)\wedge (\Linf(\hh\bM)\vee\Linf(\hh\bL))\subset (\Linf(\hh\HH)\wedge\Linf(\hh\bM))\vee \Linf(\hh\bL)$, which is effectively proved above under the well positioning assumption of $\HH$  and $\bL$,  is equivalent with the following containment \begin{equation}\label{eq:L_position}\Linf(\hh\bD)\wedge (\Linf(\hh\bM)\vee\Linf(\hh\bL))\subset (\Linf(\hh\bD)\wedge\Linf(\hh\bM))\vee \Linf(\hh\bL).\end{equation} Since $\bD,\bL,\bM\le \bL\vee\bM$, proving \Cref{eq:L_position}  we can substitute $\bM\vee\bL$ for $\GG$. After this substitution    the normalization assumption  of $\bM$ by $\bL$ gets replaced by  the normality of   $\bM$   in $\GG$. Using \Cref{pr.well} we see that $\bD$ and $\bM$ are well positioned  and by the first part of the proof \Cref{eq:L_position} holds, thus we are done. 

\vspace{.5cm}

{\bf Proof of part (b).} Here, we translate the claim into an equivalent statement for the underlying von Neumann algebras of the dual groups $\hh{\bG}$, $\hh{\bH}$, etc. 

Since $\bL$ normalizes $\bM$  we can use \Cref{norminter} and get 
\begin{equation*}
 \Linf(\hh{\bM\vee\bL}) =   \{\Linf(\hh\bM)\Linf(\hh\bL)\}^{\cls }.
\end{equation*}
  In conclusion, the von Neumann subalgebra of $\Linf(\hh{\bG})$ corresponding to the left hand side of \Cref{eq:mod_lc} is
\begin{equation*}
  \Linf(\hh{\bH})\wedge \{\Linf(\hh\bM)\Linf(\hh\bL)\}^{\cls }.
\end{equation*}
Similarly, the right hand side corresponds to 
\begin{equation*}
  \{(\Linf(\hh{\bH})\wedge  \Linf(\hh\bM)) \Linf(\hh\bL)\}^{\cls },  
\end{equation*}
and we seek to prove
\begin{equation}\label{eq:mod_open_goal}
 \Linf(\hh{\bH})\wedge \{\Linf(\hh\bM)\Linf(\hh\bL)\}^{\cls } =  \{(\Linf(\hh{\bH})\wedge  \Linf(\hh\bM)) \Linf(\hh\bL)\}^{\cls }.
\end{equation}
As in the first part, the inclusion of the right hand side in the left hand side is \Cref{lemma:cont_modul}, and we only need to prove `$\subseteq$'.

We will use essentially the same strategy as in the proof of part (1), substituting for $T_{\bL\to \bG}:\Linf(\bG)\to A_\bL$ from that other proof the expectation
\begin{equation*}
  T:\Linf(\hh{\bG})\to \Linf(\hh{\bH})
\end{equation*}
corresponding to the compatible Haar weights on the two von Neumann algebras (this is where the openness of $\bH$ is essential; see e.g. \cite[Theorem 7.5]{KalKS}). 

As before, applying $T$ to the left hand side of \Cref{eq:mod_open_goal} on the one hand acts as the identity, and on the other produces from a product
\begin{equation*}
  x=x_{\bM} x_{\bL},\ x_{\bM}\in \Linf(\hh{\bM}),\ x_{\bL}\in \Linf(\hh{\bL})
\end{equation*}
the element 
\begin{equation*}
  T(x_\bM)x_\bL
\end{equation*}
due to the $\Linf(\hh\bL)$-bimodule map property of $T$. The conclusion that $x=T(x)$ belongs to the right hand side of \Cref{eq:mod_open_goal} now follows from the fact that
\begin{equation*}
  T(x_\bM)\in \Linf(\hh{\bH})\wedge \Linf(\hh{\bM}),
\end{equation*}
which in turn relies on \Cref{pr.aux_discr_lc}. 
\end{proof}

\begin{remark}
The reader should note that since $\bL=\bH\wedge \bL$, the modular law is really a form of the distributive law $\bH\wedge (\bM\vee\bL)=(\bH\wedge \bM)\vee(\bH\wedge \bL)$. The latter, however, is false in general.
\end{remark}

\begin{lemma}\label{le.aux_cpct_lc}
  Let $\bG$ be a locally compact quantum group, $\bL\le \bG$ a compact quantum subgroup, and $\bM\le \bG$ a closed quantum subgroup normalized by $\bL$. Then, the expectation
  \begin{equation*}
    T:\Linf(\bG)\to A_\bL
  \end{equation*}
leaves $A_\bM$ invariant. 
\end{lemma}
\begin{proof}
Indeed, the normalization condition ensures that the right action of
$\bL$ on $\bG$ descends to an action on the quantum homogeneous space
$\bG/\bM$ via the commutative diagram (see \Cref{lemma:normalization_quotient})
\begin{equation*}
  \begin{tikzpicture}[auto,baseline=(current  bounding  box.center)]
    \path[anchor=base] (0,0) node (1) {$\Linf(\bG)$} +(0,-2) node (2)
    {$A_\bM$} +(3,-2) node (3) {$A_{\bM}\vtens \Linf(\bL)$} +(3,0)
    node (4) {$\Linf(\bG)\vtens\Linf(\bL)$};
         \draw[->] (1) -- (4);
         \draw[->] (2) -- (3);
         \draw[right hook->] (2) -- (1);
         \draw[right hook->] (3) -- (4);
  \end{tikzpicture}
\end{equation*} 
The conclusion now follows from the definition of the expectation $T$
as the coaction
\begin{equation*}
  \Linf(\bG)\to \Linf(\bG)\vtens \Linf(\bL)
\end{equation*}
followed by an application of the Haar state $\phi_\bL$ to the right hand
tensorand. 
\end{proof}

\begin{remark}
  We note that an appropriately rephrased version of \Cref{le.aux_cpct_lc} holds under the weaker requirement that   $\bL/\bL\wedge \bM$ acts integrably on $\bG/\bM$. $T$ would then restrict to a semifinite operator-valued weight 
  \begin{equation*}
    A_\bM \to A_\bL\wedge A_\bM. 
  \end{equation*}
We do, however, need compactness in the proof of \Cref{th.mod_lc} above, where the operator-valued weight $T$ had to be an expectation and hence fix $A_\bL$ pointwise. 
\end{remark}

Note that \Cref{th.mod_lc} does not hold in full generality, even for classical locally compact abelian groups. In order to see this, we can modify \Cref{ex.pathol} as follows.

\begin{example}\label{ex.pathol_bis}
Our ambient group $\bG=\bT^4\times\bR$ is written as in
\Cref{ex.pathol}, multiplicatively in the first four variables and
additively in the last. 

We then take 
\begin{equation*}
  \bM = \{(e^{ix\theta_1},\cdots,e^{ix\theta_4},x)\ |\ x\in \bR\}
\end{equation*}
and
\begin{equation*}
  \bL = \{(e^{is\phi},1,1,1,s)\ |\ s\in \bR\}.  
\end{equation*}
with $\phi$ and $\theta_i$ linearly independent over $\bQ$. Finally, let 
\begin{equation*}
  \bH = \{(e^{is\phi}, e^{it\phi},1,1,s+t)\ |\ s,t\in \bR\}. 
\end{equation*}
$\bH$ is easily seen to be a two-dimensional closed Lie subgroup of
$\bG$ that contains $\bL$ and intersects $\bM$ trivially. Since
$\bM\bL$ is dense in $\bG$, we have $\bH\wedge (\bM\vee\bL)=\bH$ on the left
hand side of \Cref{eq:mod_lc}. On the other hand, the right hand side $(\bH\wedge \bM)\vee\bL$ is $\bL$.
\end{example}

As in \Cref{se.lr} above, we have the following consequence of
modularity (i.e. of \Cref{th.mod_lc}).

\begin{proposition}\label{pr.zas_lc}
  Let $\bA'\trianglelefteq \bA$ and $\bB'\trianglelefteq \bB$ be
  either
  \begin{enumerate}
    \renewcommand{\labelenumi}{(\alph{enumi})}
     \item compact or 
     \item open 
  \end{enumerate}
  quantum subgroups of a locally compact quantum group $\bG$. Then, we have an isomorphism
  \begin{equation*}
    \frac{\bA'\vee(\bA\wedge \bB)}{\bA'\vee(\bA\wedge \bB')} \cong \frac{\bB'\vee(\bA\wedge \bB)}{\bB'\vee(\bA'\wedge \bB)}. 
  \end{equation*}
\end{proposition}
\begin{proof}
  This follows from \Cref{th.mod_lc} in much the same way in which
  \Cref{pr.zas} follows from \Cref{pr.mod_cpct,pr.mod_discr}, by
  applying the earlier result to  $\bH=\bA\wedge \bB$, $\bL=\bA\wedge
  \bB'$ and $\bM=\bA'$. 

  Everything goes through as before, modulo the observation that in
  the open case we need \Cref{pr.aux_discr_lc} in order to conclude
  that $\bA\wedge \bB$ is open in the open subgroup $\bB$, and hence
  is also open in $\bG$ by \Cref{pr.open_trans}.
\end{proof}

\begin{remark}
  In case (a) of \Cref{pr.zas_lc} it is enough that $\bL$ (and hence say $\bB'$) be compact.  
\end{remark}

\subsection{Schreier and J\"ordan-Holder-type results}

We devote this subsection to certain partial analogues of
\Cref{th.sch,th.jh,pr.jh} in the setting of locally compact quantum
groups.   

In this context, the relevant notions of (sub)normal series and
refinements thereof make sense virtually verbatim, so we point to
\Cref{def.series,def.refine} for a reminder. 

We write $\{\bG_\ell\}_{\ell\ge 0}$ for the generic (sub)normal series  
\begin{equation}\label{eq:series_lc}
  \bG=\bG_{0}\ge\bG_{1}\ge\bG_{2}\ge\bG_{3}\ge\cdots\ge\bG_{k}= 1.
\end{equation}
of closed quantum subgroups of a locally compact quantum group $\bG$.

\begin{theorem}\label{th.sch_lc}
  Let $\bG$ be a locally compact quantum group. Then, any two
  subnormal series $\{\bG_\ell\}$ and $\bG'_t$ of $\bG$ admit
  equivalent refinements, provided 
  \begin{equation*}
    \bG_\ell,\ \ell\ge 1\quad \text{and} \quad \bG'_t,\ t\ge 1
  \end{equation*}
  are 
  \begin{enumerate}
    \renewcommand{\labelenumi}{(\alph{enumi})}
     \item compact or 
     \item open. 
  \end{enumerate}
\end{theorem}
\begin{proof}
  One can simply imitate the proof of \Cref{th.sch}, making use of
  parts (a) and (b) of \Cref{pr.zas_lc} respectively for the two parts
  of the present result.  
\end{proof}

As for an analogue of \Cref{th.jh,pr.jh}, we have

\begin{theorem}\label{th.jh_lc}
  Let $\bG$ be a locally compact quantum group. Then, all composition
  series of $\bG$ consisting of 
  \begin{enumerate}
    \renewcommand{\labelenumi}{(\alph{enumi})}
     \item compact or 
     \item open 
  \end{enumerate}
  quantum subgroups are equivalent.
\end{theorem}
\begin{proof}
Just as the proof of \Cref{th.jh}, this follows mechanically once we
have \Cref{th.sch_lc} above. 
\end{proof}

The compact versions of \Cref{th.sch_lc,th.jh_lc} refer to subnormal
series \Cref{eq:series_lc} in which all $\bG_\ell$, $\ell\ge 1$ are
compact, but $\bG=\bG_0$ need not be so. Let us note that this is
equivalent to the subquotient $\bG/\bG_1$ being non-compact. Indeed,
we have

\begin{proposition}\label{pr.cpct_series}
  A locally compact quantum group $\bG$ is compact if and only if it
  admits a subnormal series \Cref{eq:series_lc} with compact quotients
  $\bG_i/\bG_{i+1}$. 
\end{proposition}
\begin{proof}
  The direct implication `$\Rightarrow$' is immediate by considering
  the trivial length-zero series, so we focus on the opposite
  implication. 

  By induction, it suffices to show that if $\bL\trianglelefteq \bG$
  is compact along with $\bG/\bL$, then so is $\bG$. This in turn
  follows from the fact that by the Weyl integration formula proven
  above (\Cref{pr.weyl}) the Haar weight of $\bG$ is a state.
\end{proof}



\def\cprime{$'$}

\Addresses

\end{document}